\documentclass[reqno,11pt]{article}
\usepackage{a4wide,xcolor,eucal,enumerate,mathrsfs}
\usepackage[normalem]{ulem}
\usepackage{pdfsync}
\usepackage[hidelinks]{hyperref}
\usepackage{amsmath,amssymb,epsfig,amsthm}
\usepackage[latin1]{inputenc}

\numberwithin{equation}{section}

\newtheorem{theorem}{Theorem}[section]

\newtheorem{corollary}[theorem]{Corollary}
\newtheorem{lemma}[theorem]{Lemma}
\newtheorem{proposition}[theorem]{Proposition}
\newtheorem{definition}[theorem]{Definition}

\newtheorem{example}[theorem]{Example}
\newtheorem{remark}[theorem]{Remark}

\newcommand{\N}{\mathbb{N}}
\newcommand{\Q}{\mathbb{Q}}
\newcommand{\R}{\mathbb{R}}

\newcommand{\V}{\mathbb{V}}

\newcommand{\cE}{{\ensuremath{\mathcal E}}}

\newcommand{\bb}{{\mbox{\boldmath$b$}}}
\newcommand{\cc}{{\mbox{\boldmath$c$}}}
\newcommand{\dd}{{\mbox{\boldmath$d$}}}
\newcommand{\ee}{{\mbox{\boldmath$e$}}}
\newcommand{\ff}{{\mbox{\boldmath$f$}}}
\renewcommand{\gg}{{\mbox{\boldmath$g$}}}

\newcommand{\mm}{{\mbox{\boldmath$m$}}}

\newcommand{\vv}{{\mbox{\boldmath$v$}}}

\newcommand{\XX}{{\mbox{\boldmath$X$}}}
\newcommand{\YY}{{\mbox{\boldmath$Y$}}}
\newcommand{\sXX}{{\mbox{\scriptsize\boldmath$X$}}}

\newcommand{\sgg}{{\mbox{\scriptsize\boldmath$g$}}}

\newcommand{\bB}{{\mbox{\boldmath$B$}}}

\newcommand{\eeta}{{\mbox{\boldmath$\eta$}}}
\newcommand{\llambda}{{\mbox{\boldmath$\lambda$}}}

\newcommand{\ppi}{{\mbox{\boldmath$\pi$}}}

\newcommand{\sfd}{{\sf d}}

\newcommand{\sfP}{{\sf P}}

\newcommand{\rme}{{\mathrm e}}

\newcommand{\rmD}{{\mathrm D}}

\newcommand{\rmI}{{\mathrm I}}

\newcommand{\dom}{\rmD}
\newcommand{\supp}{\mathop{\rm supp}\nolimits}   
\newcommand{\restr}[1]{\lower3pt\hbox{$|_{#1}$}}

\newcommand{\Leb}[1]{{\mathscr L}^{#1}}      
\newcommand{\BorelSets}[1]{{\mathscr B}(#1)}
\newcommand{\eps}{\varepsilon}  
\newcommand{\nchi}{{\raise.3ex\hbox{$\chi$}}}

\newcommand{\fr}{\hfill$\blacksquare$}                      

\def\qed{\ifmmode 
  \else \leavevmode\unskip\penalty9999 \hbox{}\nobreak\hfill
  \fi               
    \qquad           \hbox{\hskip.5em $\square$
                \hskip.1em}}

\renewcommand{\mm}{\mathfrak m} 
\newcommand{\Gq}[1]{\Gamma\!\left(#1\right)}          

\newcommand{\Probabilities}[1]{\mathscr P(#1)}          

\newcommand{\CD}{{\sf CD}}
\newcommand{\RCD}{{\sf RCD}}
\newcommand{\BE}{{\sf BE}}

\newcommand{\sqGq}[1]{\sqrt{ \Gq {#1}}}
\newcommand{\bra}[1]{\left( #1 \right)}
\newcommand{\sqa}[1]{\left[ #1 \right]}
\newcommand{\cur}[1]{\left\{ #1 \right\}}
\newcommand{\ang}[1]{\left< #1 \right>}
\newcommand{\abs}[1]{\left| #1 \right|}
\newcommand{\nor}[1]{\left\| #1 \right\|}
\newcommand{\Lbm}[1]{L^{#1}(\mm)}

\newcommand{\Lbtx}[2]{L^{#1}_t(L^{#2}_x)}

\newcommand{\vphi}{\varphi}

\renewcommand{\div}{\mathop{\rm div}\nolimits}        
\newcommand{\mmt}{\widetilde{\mm}}
\newcommand{\dbneg}{\div\bb^-}
\newcommand{\veps}{\varepsilon}
\newcommand{\down}{\downarrow}
\newcommand{\up}{\uparrow}
\newcommand{\Commutator}{{\mathscr C}}
\newcommand{\Algebra}{{\mathscr A}}

\title{Well posedness of Lagrangian flows and continuity equations \\ in metric measure spaces}

\begin{document}

\author{Luigi Ambrosio,\ Dario Trevisan\
\thanks{Scuola Normale Superiore, Pisa, \textsf{luigi.ambrosio@sns.it, dario.trevisan@sns.it}}  
}

\maketitle

\begin{abstract}
We establish, in a rather general setting, an analogue of DiPerna-Lions theory on well-posedness of flows of ODE's associated to Sobolev vector fields. Key results are a well-posedness result for the continuity equation associated to suitably defined Sobolev vector fields, via a commutator estimate, and an abstract superposition principle in (possibly extended) metric measure spaces, via an embedding into $\R^\infty$.

When specialized to the setting of Euclidean or infinite dimensional (e.g.\ Gaussian) spaces, large parts of previously known results are recovered at once. 
Moreover, the class of $\RCD(K,\infty)$ metric measure spaces introduced in \cite{AGS11b} and object of extensive recent research fits into our framework.
Therefore we provide, for the first time, well-posedness results for ODE's under low regularity assumptions on the velocity and in a nonsmooth context. 
\end{abstract}

\tableofcontents

\section{Introduction}

The theory of DiPerna-Lions, initiated in the seminal paper \cite{DiPerna-Lions89}, provides existence, stability and uniqueness results for ODE's
associated to large classes of non-smooth vector fields, most notably that of Sobolev vector fields. In more recent times the first
author extended in \cite{Ambrosio03} the theory to include $BV$ vector fields and, at the same time, he introduced a more probabilistic axiomatization
based on the duality between flows and continuity equation, while the approach of \cite{DiPerna-Lions89} relied on characteristics and
the transport equation.
In more recent years the theory developed in many different directions, including larger classes of vector fields, quantitative
convergence estimates, mild regularity properties of the flow, and non-Euclidean spaces, including infinite-dimensional ones. 
We refer to the Lecture Notes \cite{cetraro} and \cite{bologna} for more exhaustive, but still incomplete, description of the developments on this topic.

Aim of this paper is to extend the theory of well posedness for the continuity equation and the theory of flows
to metric measure spaces $(X,d,\mm)$. Roughly speaking, and obviously under additional structural assumptions, we prove
that if $\{\bb_t\}_{t\in (0,T)}$ is a time-dependent family of Sobolev vector fields then there is a unique flow associated to $\bb_t$, namely a family of
absolutely continuous maps $\{\XX(\cdot,x)\}_{x\in X}$ from $[0,T]$ to $X$ satisfying:
\begin{itemize}
\item[(i)]  $\XX(\cdot,x)$ solves the possibly non-autonomous ODE associated to $\bb_t$ for $\mm$-a.e.\ $x\in X$;
\item[(ii)] the push-forward measures $\XX(t,\cdot)_\#\mm$ are absolutely continuous w.r.t.\ $\mm$ and have uniformly
bounded densities.
\end{itemize}
Of course the notions of ``Sobolev vector field'' and even ``vector field'', as well as the notion of solution to the
ODE have to be properly understood in this nonsmooth context, where not even local coordinates are available.
As far as we know, these are the first well-posedness results for ODE's under low regularity assumptions and in
a nonsmooth context. 

One motivation for writing this paper has been the theory of ``Riemannian" metric measure spaces developed
by the first author in collaboration with N.~Gigli and G.~Savar\'e, leading to a series of papers \cite{AGS11a}, \cite{AGS11b}, \cite{AGS12} 
and to further developments in \cite{Gigli2012}, \cite{Gigli-splitting}. In this perspective, it is important to develop new calculus tools in metric
measure spaces. For instance, in the proof of the splitting theorem in \cite{Gigli-splitting} a key role is played by the flow
associated to the gradient of a $c$-concave harmonic function, whose flow lines provide the fibers of the product decomposition;
therefore a natural question is under which regularity assumption on the potential $V$ the gradient flow associated
to $V$ has a unique solution,  where uniqueness is not understood pointwise, but in the sense of the DiPerna-Lions
theory (see Theorem~\ref{thm:uniflow} and Theorem~\ref{thm:ediss2} for a partial answer to this question). We also point out the recent paper \cite{GB-14}, where continuity equations in metric measure spaces are introduced and studied in connection with absolutely continuous curves with respect to the Wasserstein distance $W_2$, thus relying mainly on a ``Lagrangian'' point of view.

\smallskip

The paper is basically organized in three parts: in the ``Eulerian'' one, which has an independent interest, we study the well-posedness of  continuity
equations, in the ``Lagrangian'' one we define the notion of solution to the ODE and relate well-posedness of the
continuity equation to existence and uniqueness of the flow (in the same spirit of \cite{Ambrosio03}, \cite{cetraro}, 
where the context was Euclidean). Eventually, in the third part we see how a large class of previous results can be
seen as particular cases of ours. On the technical side, these are the main ingredients: for the first part, a new intrinsic 
way to write down the so-called commutator estimate, obtained with $\Gamma$-calculus tools (this point of view
is new even for ``nice'' spaces as Euclidean spaces and Riemannian manifolds); for the second part, a more
general version of the so-called superposition principle (see for instance \cite[Theorem~8.2.1]{Ambrosio-Gigli-Savare05}, in the setting of
Euclidean spaces), that allows to lift, not canonically in general, nonnegative solutions
to the continuity equation to measures on paths. 
 
\smallskip
We pass now to a more detailed description of the three parts. 

\smallskip
\noindent
{\bf Part 1.} This part consists of five sections, from Section~\ref{sec:setup} to Section~\ref{sec:curvature}.  Section~\ref{sec:setup} is devoted to the description of our abstract setup, which is the typical one of
$\Gamma$-calculus and of the theory of Dirichlet forms: for the moment the distance is absent and we are 
given only a topology $\tau$ on $X$ and a reference measure $\mm$ on $X$, which is required to be Borel, nonnegative and $\sigma$-finite. On
$\Lbm 2$ we are given a symmetric, densely defined and strongly local Dirichlet form $(\cE, D(\cE))$ whose semigroup
$\sfP$ is assumed to be Markovian. We write $\V := D(\cE)$ and assume that a {\it Carr\'e du Champ} $\Gamma: \V\times \V\to\Lbm 1$
is defined, and that we are given a ``nice'' algebra $\Algebra$ which is dense in $\V$ 
and which plays the role of the $C^\infty_c$ functions in the theory of distributions. 

Using $\Algebra$, we can define in Section~\ref{sec:derivations} ``vector fields'' as derivations, in the same spirit of \cite{weaver} (see also
\cite{Ambrosio-Kirchheim} for parallel developments in the theory of currents): a derivation $\bb$ is a linear
map from $\Algebra$ to the space of real-valued Borel functions on $X$, satisfying the Leibniz rule $\bb(fg)=f\bb(g)+g\bb(f)$, and a pointwise $\mm$-a.e.\ bound in terms of $\Gamma$. We will use the more intuitive notation
(since differentials of functions are co-vectors) $f\mapsto df(\bb)$ for the action of a derivation $\bb$ on $f$. An important
example is provided by \emph{gradient} derivations $\bb_V$ induced by $V\in\V$, acting as follows:
$$
df(\bb_V):=\Gamma(V,f).
$$
Although we will not need more than this, we would like to mention the forthcoming paper \cite{gigli-WIP}, which provides
equivalent axiomatizations, in which the Leibniz rule is not an axiom anymore, and it is shown that gradient derivations generate,
in a suitable sense, all derivations.  
Besides the basic example of gradient derivations, 
the {\it Carr\'e du Champ} provides, by duality, a natural pointwise norm on derivations; such duality can be used to define, via
integration by parts, a notion of divergence $\div\bb$ for a derivation (the divergence depends only on $\mm$, not on $\Gamma$).
In Section~\ref{sec:existence} we prove existence of solutions to the weak formulation of the continuity equation $\frac{d}{dt}u_t+\div (u_t\bb_t)=0$ 
induced by a family $(\bb_t)$ of derivations, namely
$$
\frac{d}{dt}\int fu_t d\mm=\int df(\bb_t) u_t d\mm\qquad\forall f\in\Algebra.
$$

The strategy of the proof is classical: first we add a viscosity term and get a $\V$-valued solution by Hilbert space techniques, then
we take a vanishing viscosity limit. Together with existence, we recover also higher (or lower, since our measure $\mm$ might be not finite
and therefore no inclusion between $L^p$ spaces might hold) integrability estimates on $u$, depending on the initial
condition $\bar u$. Also, under a suitable assumption \eqref{eq:mass_preserving} on $\Algebra$, we prove that the $L^1$ norm is independent
of time. Section~\ref{sec:commutators} is devoted to the proof of uniqueness of solutions to the continuity equation. The classical proof in 
\cite{DiPerna-Lions89} is based on a smoothing scheme
that, in our context, is played by the semigroup $\sfP$ (an approach already proved to be successful in \cite{AF-09}, \cite{Trevisan-13}, in Wiener spaces).
For $t$ fixed, one has to estimate carefully the so-called commutator
$$
\Commutator^\alpha (\bb_t,u_t):=\div ((\sfP_\alpha u_t)\bb_t)-\sfP_\alpha (\div (u_t\bb_t))
$$
as $\alpha\to 0$.
The main new idea here is to imitate Bakry-\'Emery's $\Gamma$-calculus (see e.g.\ the recent monograph \cite{BGL-13}), interpolating and writing, at least formally,
\begin{eqnarray}\label{eq:introcommu}
\Commutator^\alpha (\bb_t,u_t)&=&\int_0^\alpha \frac{d}{ds}\sfP_{\alpha-s} ( \div ( \sfP_s(u_t)\bb_t ) ) ds\\&=&
\int_0^\alpha \bigl[-\Delta \sfP_{\alpha-s} (\div (\sfP_s(u_t)\bb_t))+\sfP_{\alpha-s} (\div (\Delta\sfP_s(u_t)\bb_t))\bigr] ds.\nonumber
\end{eqnarray}
It turns out that an estimate of the commutator involves only the symmetric part of the derivative (this, in the Euclidean case,
was already observed in \cite{Capuzzo} for regularizations induced by even convolution kernels). This structure can be recovered in our context: inspired by the definition of Hessian
in \cite{Bakry-94} we define the symmetric part $ D^{sym}\cc$ of the gradient  a deformation $\cc$ by
\begin{equation}\label{eq:introdsym2}
\int  D^{sym}\cc(f,g) d\mm := -\frac{1}{2} \int \sqa{df(\cc)\Delta g + dg(\cc) \Delta f - (\div\cc) \Gamma(f,g)}  d\mm.
\end{equation}
Using this definition in \eqref{eq:introcommu} (assuming here for simplicity $\div \bb_t = 0$) we establish the identity
\begin{equation}\label{eq:introcommu-nice}
\int f \, \Commutator^\alpha (\bb_t,u_t) d\mm = 2\int_0^\alpha \int D^{sym}\bb_t( \sfP_{\alpha- s} f , \sfP_s u_t) d\mm ds, \qquad\forall f\in\Algebra.
\end{equation} 
Then, we assume the validity of the estimates (see Definition~\ref{def:deformation} for a more general
setup with different powers) 
\begin{equation}\label{eq:introdsym}
\biggl|\int D^{sym}\bb_t(f,g) d\mm \biggr| \le c \biggl(\int\Gamma(f)^2 d\mm\biggr)^{1/4} \biggl(\int\Gamma(g)^2 d\mm\biggr)^{1/4},
\end{equation}
which, in a smooth context, amount to an $L^2$ control on the symmetric part of derivative.
Luckily, the control \eqref{eq:introdsym} on $D^{sym}\bb_t$ can be combined with \eqref{eq:introcommu} and
\eqref{eq:introcommu-nice} to obtain strong convergence to $0$ of the commutator as $\alpha \to 0$ and therefore
well-posedness of the continuity equation. This procedure works 
assuming some regularizing properties of the semigroup $\sfP$, especially the validity of 
$$
\biggl(\int \Gamma(\sfP_t f)^2 d\mm \biggr)^{1/4}  \le \frac{c}{\sqrt{t}} \bra{\int |f|^4 d\mm}^{1/4}, \quad\text{for every $f \in L^2\cap\Lbm 4$,  $t \in (0,1)$,}
$$
for some constant $c\ge 0$ (see Theorem~\ref{thm:uniqueness}). In particular, these hold assuming an abstract curvature lower bound on the underlying space, as discussed in Section~\ref{sec:curvature},
where we crucially exploit the recent results in \cite{Savare-13} and \cite{Ambrosio-Mondino-Savare-13} to show that our structural
assumptions on $\sfP$ and on $\Algebra$ are fulfilled in the presence of lower bounds on the curvature. Furthermore, gradient derivations
associated to sufficiently regular functions satisfy \eqref{eq:introdsym}.

Finally, we remark that, as in \cite{DiPerna-Lions89}, analogous well-posedness results could be obtained for weak solutions
to the inhomogeneous transport equation
$$ 
\frac{d}{dt} u_t+ du_t(\bb_t)=c_tu_t+w_t .
$$
under suitable assumptions on $c_t$ and $w_t$.
We confined our discussion to the case of the homogeneous continuity equation (corresponding to $c_t=-\div\bb_t$ and $w_t=0$)
for the sake of simplicity and for the relevance of this PDE in connection with the theory of flows.

\smallskip
\noindent
{\bf Part 2.} This part consists of two sections. In Section~\ref{sec:superpo} we show how solutions $u$ to the continuity equation 
$\frac{d}{dt}u_t+\div (u_t\bb_t)=0$
can be lifted to measures $\eeta$ in $C([0,T];X)$. Namely, we would like that $(e_t)_\#\eeta=u_t\mm$ for all $t\in (0,T)$
and that $\eeta$ is concentrated on solutions $\eta$ to the ODE $\dot\eta=\bb_t(\eta)$. This statement is well-known
in Euclidean spaces (or even Hilbert spaces), see \cite[Thm.~8.2.1]{Ambrosio-Gigli-Savare05}; in terms of currents, it could be seen as a particular
case of Smirnov's decomposition \cite{Smirnov} of 1-currents as superposition of rectifiable currents. Here we realized
that the most appropriate setup for the validity of this principle is $\R^\infty$, see Theorem~\ref{thm:superpoRinfty}, where
only the Polish structure of $\R^\infty$ matters and neither distance nor reference measure come into play.

In order to extend this principle from $\R^\infty$ to our abstract setup we assume the existence of a sequence $(g_k)\subset\{f\in\Algebra:\
\|\Gamma(f)\|_\infty\leq1\}$ satisfying:
\begin{equation}\label{eq:introGstar1}
\text{${\rm span}(g_k)$ is dense in $\V$ and any function $g_k$ is $\tau$-continuous,} 
\end{equation}
\begin{equation}\label{eq:introGstar2}
\text{$\exists\lim_{n\to\infty}g_k(x_n)$ in $\R$ for all $k$}\quad\Longrightarrow\quad\text{$\exists\lim_{n\to\infty} x_n$ in $X$.} 
\end{equation}
This way, the embedding $J:X\to\R^\infty$ mapping $x$ to $(g_k(x))$ provides an homeomorphism of $X$
with $J(X)$ and we can first read the
solution to the continuity equation in $\R^\infty$ (setting $\nu_t:=J_\#(u_t\mm)$, with an appropriate choice of the
velocity in $\R^\infty$) and then pull back the lifting obtained
in $\Probabilities{C([0,T];\R^\infty)}$ to obtain $\eeta\in\Probabilities{C([0,T];X)}$, see Theorem~\ref{thm:superpo}. 
It turns out that $\eeta$ is concentrated on curves $\eta$ satisfying
\begin{equation}\label{eq:understood}
\frac{d}{dt}(f\circ\eta)=df(\bb_t)\circ\eta \qquad\text{in the sense of distributions in $(0,T)$, for all $f\in\Algebra$,}
\end{equation}
which is the natural notion of solution to the ODE $\dot\eta=\bb_t(\eta)$ in our context (again, consistent with the fact
that a vector can be identified with a derivation). We show, in addition, that this property implies absolute
continuity of $\eeta$-almost every curve $\eta$ with respect to the possibly extended distance $d(x,y):=\sup_k|g_k(x)-g_k(y)|$,
with metric derivative $|\dot\eta|$ estimated from above with $|\bb_t|\circ\eta$.
Notice also that, in our setup, the distance appears only now. Also, we remark that a similar change of variables
appears in the recent paper \cite{Kolesnikov-Rockner-13}, but not in a Lagrangian perspective: it is used therein to prove 
well-posedness of the continuity equation when the reference measure is log-concave (see Section~\ref{sec:examples-log-concave}). 

Section~\ref{sec:RLF} is devoted to the proof of Theorem~\ref{thm:uniflow}, which links well-posedness of the continuity equation
in the class of nonnegative functions $L^1_t(L^1_x\cap L^\infty_x)$ with initial data $\bar u\in L^1\cap\Lbm\infty$ to the existence
and uniqueness of the flow $\XX$ according to (i), (ii) above, where (i) is now understood as in \eqref{eq:understood}. The proof
of Theorem~\ref{thm:uniflow} is based on two facts: first, the possibility to lift solutions $u$ to probabilities $\eeta$, discussed in the previous section; 
second, the fact that the restriction of $\eeta$ to any Borel set still induces a solution to the continuity equation with the same velocity
field. Therefore we can ``localize'' $\eeta$ to show that, whenever some branching of trajectories occurs, then there is non-uniqueness at
the level of the continuity equation. 

Let us comment that, in this abstract setting, it seems more profitable to the authors to deal uniquely with continuity equations, instead of transport equations as in \cite{DiPerna-Lions89}, since the latter require in its very definition a choice of  ``coordinates'', while the former arises naturally as the description of evolution of underlying measures.

\smallskip
\noindent
{\bf Part 3.} This part consists of Section~\ref{sec:examples} only, where we specialize the general theory to settings where continuity equations and 
associated flows have already been considered, and to $\RCD(K,\infty)$-metric measure spaces. Since the transfer mechanism of well-posedness 
from the PDE to the ODE levels is quite general, we mainly focus on the continuity equation. Moreover, in these particular settings (except for $\RCD(K,\infty)$ spaces), the proof of existence  for solutions turns out to be a much easier task than in the general framework, 
due to explicit and componentwise approximations by smooth vector fields. Therefore, we limit ourselves to compare uniqueness results.

In Section~\ref{sec:diperna-lions}, we show how the classical DiPerna-Lions theory of \cite{DiPerna-Lions89} fits into our setting: in short, we recover almost all the well-posedness results in \cite{DiPerna-Lions89}, with the notable exception of  $W^{1,1}_{\rm loc}$-regular vector fields. In Section~\ref{sec:riemannian} we also describe how our techniques provide intrinsic proofs, i.e.\ without reducing to local coordinates, of analogous results for weighted Riemannian manifolds.

In Section~\ref{sec:examples-wiener} and Section~\ref{sec:examples-daprato}, we deal with (infinite dimensional) Gaussian frameworks, comparing our results to those established respectively in \cite{AF-09}, \cite{DaPrato-Flandoli-Rockner-13}: large parts of these can be obtained as consequences of our general theory, which turns out to be more flexible e.g.\ we can allow for vector fields that do not necessarily take values in the Cameron-Martin space (see at the end of Section~\ref{sec:examples-daprato}), which is not admissible in \cite{AF-09} or in \cite{DaPrato-Flandoli-Rockner-13}. In Section~\ref{sec:examples-log-concave} we consider the even more general setting of log-concave measures and make a comparison with some of the results contained in \cite{Kolesnikov-Rockner-13}. The strength of our approach is immediately revealed, e.g.\ we are not limited as in \cite{Kolesnikov-Rockner-13} to uniformly log-concave measures.

We conclude in Section~\ref{sec:RCD} by describing how the theory specializes to the setting of $\RCD(K,\infty)$-metric measure spaces, that is one of our original motivations for this work. We show that Lagrangian flows do exist in many cases (Theorem~\ref{thm:ediss2}) and provide instances of so-called \emph{test plans}. In the case of gradient derivations, we also show that the trajectories satisfy a global energy dissipation identity (Theorem~\ref{thm:ediss}).

\smallskip
\noindent {\bf Acknowledgments.}
The first author acknowledges the support of the ERC ADG GeMeThNES.
The second author has been partially supported by PRIN10-11 grant from MIUR
for the project Calculus of Variations. Both authors are members of the GNAMPA group of the Istituto 
Nazionale di Alta Matematica (INdAM).

The authors thank G.\ Savar\'e for pointing out to them, after reading a preliminary version of this work, the possibility to 
dispense from the full curvature assumption, an observation that led to the present organization of the paper. 
They also thank F.\ Ricci and G.\ Da Prato for useful discussions.

\section{Notation and abstract setup}\label{sec:setup}

Let $(X,\tau)$ be a Polish topological space, endowed with
a $\sigma$-finite Borel measure $\mm$ with full support (i.e.\ $\supp\mm=X$) and
\begin{equation}
 \label{E-assumptions}
 \begin{gathered}
    \text{a strongly local, densely defined and symmetric Dirichlet form $\cE$ on
      $L^2(X,\BorelSets{X},\mm)$}\\
    \text{enjoying a \emph{Carr\'e du Champ}
    $\Gamma: D(\cE)\times D(\cE)\to L^1(X,\BorelSets{X},\mm)$ and}\\
  \text{generating
  a Markov semigroup $(\sfP_t)_{t\ge0}$ on $L^2(X,\BorelSets{X},\mm)$.}
\end{gathered}
\end{equation}

The precise meaning of \eqref{E-assumptions} is recalled below in this section.

To keep notation simple, we write $\Lbm p$ instead of $L^p(X,\BorelSets{X},\mm)$ and denote $\Lbm p$ norms by $\nor{\cdot}_{p}$. We also write $\Lbm 0$ for the space of $\mm$-a.e.\ equivalence classes of Borel functions $f: X \mapsto [-\infty, + \infty]$ that take finite values $\mm$-a.e.\ in $X$.

Since $(X,\tau)$ is Polish and $\mm$ is $\sigma$-finite, the spaces $\Lbm p$ are separable for $p\in [1,\infty)$. We shall also use
the duality relations $$(\Lbm p+\Lbm q)^*=L^{p'}\cap\Lbm{q'}\qquad\text p,\,q\in [1,\infty)$$
and the notation $\|\cdot\|_{L^p+L^q}$, $\|\cdot\|_{L^{p'}\cap L^{q'}}$. In addition, we will use that the spaces 
$\Lbm p$, $1\leq p\leq\infty$ (and $p =0$) are complete lattices with respect to the order relation induced by the inequality $\mm$-a.e.\ in $X$. This follows at once from the general fact that, for any family of Borel functions $f_i:X\to [-\infty,+\infty]$ there exists $f:X\to [-\infty,+\infty]$ Borel such that $f\geq f_i$
$\mm$-a.e.\ in $X$ for all $i\in I$ and $f\geq g$ $\mm$-a.e.\ in $X$ for any function $g$ with the same property. Existence
of $f$ can be achieved, for instance, by considering the maximization of \[J\mapsto\int \tan^{-1}(\sup\limits_{i\in J}f_i)\vartheta d\mm\] among the
finite subfamilies $J$ of $I$, with $\vartheta$ positive function in $\Lbm 1$ (notice that the pointwise supremum could lead to a function which
is not $\mm$-measurable).

\subsection{Dirichlet form and Carr\'e du champ}

A symmetric Dirichlet form $\cE$ is a $\Lbm 2$-lower semicontinuous quadratic form satisfying the Markov property
\begin{equation}\label{eq:Markovian}
\cE(\eta\circ f)\leq \cE(f)\qquad\text{for every normal contraction $\eta:\R\to\R$},
\end{equation}
i.e.~a $1$-Lipschitz map satisfying $\eta(0)=0$. We refer to \cite{Bouleau-Hirsch91,Fukushima-Oshima-Takeda11} 
for equivalent formulations of \eqref{eq:Markovian}. Recall that
	$$ \V := \dom(\cE)\subset \Lbm 2, \quad \text{endowed with} \, \nor{f}_{\V}^2:= \nor{f}_2^2 + \cE(f) $$ 
is a Hilbert space. Furthermore, $\V$ is separable because $\Lbm 2$ is separable (see \cite[Lemma~4.9]{AGS11b} for the simple
proof).

We still denote by $\cE(\cdot,\cdot):\V\times \V \to\R$ the 
associated continuous and symmetric bilinear form
\begin{displaymath}
  \cE(f,g):=\frac 14\Big(\cE(f+g)-\cE(f-g)\Big).
\end{displaymath}
We will assume strong locality of $\cE$, namely
\[
\forall\,f,\,g\in\V:\quad 
\cE(f,g)=0,\quad\text{if $(f+a)g=0$,  $\mm$-a.e.\ in $X$, for some $a\in\R$.}
\]
It is possible to prove (see for instance \cite[Prop.~2.3.2]{Bouleau-Hirsch91}) that $\V\cap\Lbm\infty$ is an algebra with respect to pointwise
multiplication, so that for every $f\in\V\cap\Lbm\infty$ the linear form on $\V\cap\Lbm\infty$
\begin{equation}
  \label{eq:65bis}
  \Gamma[f;\varphi]:=2\cE(f,f\varphi)-\cE(f^2,\varphi),\quad\varphi\in\V\cap\Lbm\infty,
\end{equation}
is well defined and, for every normal contraction $\eta:\R\to\R$, it
satisfies \cite[Prop.~2.3.3]{Bouleau-Hirsch91}
\begin{equation}
  \label{eq:106}
  0\le\Gamma[\eta\circ f;\varphi]\le 
  \Gamma[f;\varphi]\le \|\varphi\|_\infty\,\cE(f)\quad
  \text{for all $f,\,\varphi\in\V\cap\Lbm\infty$, $\varphi\ge0$.}
\end{equation}
The inequality \eqref{eq:106} shows that for every nonnegative $\varphi\in
\V\cap\Lbm\infty$ the function $f\mapsto\Gamma[f;\varphi]$ is a 
quadratic form in $\V\cap\Lbm\infty$ which satisfies the Markov property 
and can be extended by continuity to $\V$.

We assume that for all $f\in\V$ the linear form $\varphi\mapsto\Gamma [f;\varphi]$
can be represented by a an absolutely continuous measure w.r.t.\ $\mm$
with density $\Gq f\in L^1_+(\mm)$, the so-called \emph{Carr\'e du champ}.
Since $\cE$ is strongly local, \cite[Thm.~I.6.1.1]{Bouleau-Hirsch91} yields 
the representation formula
\begin{equation}\label{eq:energymeasurebis}
\cE(f,f)=\int_X \Gq f d\mm,\quad\text{for all  $f\in\V$}.
\end{equation}
It is not difficult to check that 
$\Gamma$ as defined by \eqref{eq:energymeasurebis}
(see e.g.\ \cite[Def.~I.4.1.2]{Bouleau-Hirsch91})
is a quadratic continuous map defined in $\V$ with 
values in $L^1_+(\mm)$, and that $ \Gamma[f-g;\varphi]\geq 0$ for all
$\varphi\in\V\cap\Lbm\infty$ yields
\begin{equation}\label{eq:boundg1}
|\Gamma(f,g)|\leq\sqGq{f}\sqGq{g},\qquad\text{$\mm$-a.e.\ in $X$.}
\end{equation}

We use the $\Gamma$ notation also for the symmetric, bilinear and continuous map 
\begin{displaymath}
	\Gamma(f,g):=\frac14\Big(\Gamma(f+g)-\Gamma(f-g)\Big)\in\Lbm 1\qquad f,\,g\in\V,
\end{displaymath}
which, thanks to \eqref{eq:energymeasurebis}, represents the bilinear form $\cE$ by the formula
\begin{displaymath}
  \cE(f,g)= \frac 12 \int_X\Gamma(f,g) d\mm\quad\text{for all $f,\,g\in\V$.}
\end{displaymath}
Because of the Markov property and locality, $\Gamma(\cdot,\cdot)$ satisfies the chain rule
\cite[Cor.~I.7.1.2]{Bouleau-Hirsch91}
\begin{equation}
	\label{eq:Gchain}
		\Gamma(\eta(f),g)=\eta'(f)\Gamma(f,g)\quad		
		\text{for all $f,\,g\in\V$, $\eta:\R\to\R$ Lipschitz with $\eta(0)=0$,}
\end{equation}
and the Leibniz  rule:
\begin{displaymath}
	\Gamma(fg,h)=f\Gamma(g,h)+g\Gamma(f,h)\quad
	\quad\text{for all $f,\,g,\,h\in\V\cap L^\infty(\mm)$.}
\end{displaymath}
Notice that by \cite[Thm.~I.7.1.1]{Bouleau-Hirsch91}
\eqref{eq:Gchain} is well defined, since for every Borel set $N\subset \R$ 
(as the set where $\eta$ is not differentiable)
\begin{equation}
  \label{eq:72}
  \Leb1(N)=0\quad\Rightarrow\quad
  \Gq f=0\quad \text{$\mm$-a.e.\ on }f^{-1}(N).
\end{equation}

For $p\in [1,\infty]$, we introduce the spaces 
\begin{equation}\label{eq:sobolev-spaces}
\V_p:=\left\{u\in\V \cap \Lbm p:\ \int (\Gamma(u))^{p/2} d\mm<\infty\right\}\qquad p\in [1,\infty),
\end{equation}
with the obvious extension to $p=\infty$.
As in \cite[\S I.6.2]{Bouleau-Hirsch91}, one can endow each $\V_p$ with the norm
\begin{equation}\label{eq:sobolev-spaces-norms }\nor{f}_{\V_p} = \nor{f}_\V + \nor{f}_p + \| \Gamma(u)^{1/2} \|_p,\end{equation}
thus obtaining a Banach space, akin to the intersection of classical Sobolev spaces $W^{1,2} \cap W^{1,p}$. Notice that $\V_2 = \V$, with an equivalent norm. The Banach space structure plays a major role only starting from Section~\ref{sec:commutators}, but the notation $f \in \V_p$ is conveniently used throughout.

\subsection{Laplace operator and Markov semigroup}

The Dirichlet form $\cE$ induces a densely defined, negative and selfadjoint operator 
$\Delta:D(\Delta)\subset\V\to\Lbm 2$, defined by the integration by parts formula
$\cE(f,g)=-\int_X g\,\Delta f d\mm$ for all $g\in\V$.
The operator $\Delta$ is of ``diffusion'' type, since it satisfies  the following chain rule
for every $\eta\in C^2(\R)$ with $\eta(0)=0$ and bounded first and
second derivatives (see \cite[Corollary I.6.1.4]{Bouleau-Hirsch91}): whenever $f\in D(\Delta)$ 
with $\Gq f\in\Lbm 2$, then $\eta(f)\in D(\Delta)$ and
\begin{equation}
	\label{eq:Gchain-laplace}
	\Delta\eta(f)=\eta'(f)\Delta f+\eta''(f)\Gq f.
\end{equation}
The ``heat flow'' $\sfP_t$ associated to $\cE$ is well defined starting
from any initial condition $f\in\Lbm 2$. Recall that in this framework
the heat flow $(\sfP_t)_{t\ge0}$ is an analytic Markov semigroup 
and that $f_t=\sfP_t f$ can be characterized as the unique $C^1$ map 
$f:(0,\infty)\to\Lbm 2$, with values in $D(\Delta)$,  satisfying
\begin{displaymath}
\begin{cases}
\displaystyle\frac{d}{d t} f_t=\Delta f_t&\text{for $t\in (0,\infty)$,}\\
\lim\limits_{t\down 0}f_t=f&\text{in $\Lbm 2$.}
\end{cases}
\end{displaymath}
Because of this, $\Delta$ can equivalently be characterized in terms of
the strong convergence $(\sfP_t f-f)/t\to\Delta f$ in $\Lbm 2$ as
$t \downarrow 0$. 

Furthermore, we have the
regularization estimates (in the more general context of
gradient flows of convex functionals, see for instance 
\cite[Thm.~4.0.4(ii)]{Ambrosio-Gigli-Savare05})
\begin{equation}\label{eq:regularization2}
\cE(\sfP_t u,\sfP_tu)\leq\inf_{v\in\V}\bigl\{\cE(v,v)+\frac{\|v-u\|_2^2}{2t}\bigr\}<\infty,\qquad\forall 
t>0,\,\,u\in\Lbm 2,
\end{equation}
\begin{equation}\label{eq:regularization}
\|\Delta \sfP_t u\|_2^2\leq\inf_{v\in D(\Delta)}\bigl\{\|\Delta 
v\|_2^2+\frac{\|v-u\|_2^2}{t^2}\bigr\}<\infty,\qquad\forall 
t>0,\,\,u\in\Lbm 2.
\end{equation}
For $p\in (1,\infty)$, we shall also need an $L^p$ version of \eqref{eq:regularization}, namely
	\begin{equation}\label{eq:p-reverse-Delta-inequalities} 
	\nor{ \Delta \sfP_t f }_p \le \frac{c^\Delta_p}{t} \nor{f}_p,\quad \text{for every $f \in L^p\cap \Lbm 2$ and every $t \in (0,1)$.} 
	\end{equation}
This can be obtained as a consequence of the fact that $\sfP$ is analytic \cite[Thm.~III.1]{Stein-70}: it is actually equivalent to it, see \cite[\S X.10]{Yoshida-80}.   

As an easy corollary of \eqref{eq:p-reverse-Delta-inequalities}, we obtain the following estimate.

\begin{corollary}\label{coro:stein}
Let $p \in (1,\infty)$ and let $c^\Delta_p$ be the constant in \eqref{eq:p-reverse-Delta-inequalities}. Then 
$$
\nor{ \sfP_t f - \sfP_{t-t'} f}_p \le \min\cur{c_p^\Delta\log\bra{ 1 + \frac{t'}{t-t'}}, 2} \nor{f}_p,\qquad\forall f\in L^p\cap \Lbm 2
$$
for every $t,\,t'\in (0,1)$, with $t'\le t$.
\end{corollary}
\begin{proof}
The estimate with the constant $2$ follows from $L^p$ contractivity. For the other one, we apply \eqref{eq:p-reverse-Delta-inequalities} as follows:
\[ \nor{\sfP_t f - \sfP_{t-t'} f}_p \le \int_0^{t'} \nor{ \Delta \sfP_{t-t'+r} f }_p dr\le \int_0^{t'} \frac{c_p^\Delta}{t-t'+r} dr \nor{f}_p = 
c_p^\Delta\log\bra{ 1 + \frac{t'}{t-t'}} \nor{f}_p.\]
\end{proof}

One useful consequence of the Markov property is the $L^p$ contraction
of $(\sfP_t)_{t\ge0}$  from $L^p\cap L^2$ to $L^p\cap L^2$. 
Because of the density of $L^p\cap L^2$ in $L^p$ when $p\in [1,\infty)$, this allows to
extend uniquely $\sfP_ t$ to 
a strongly continuous semigroup of linear contractions in
$\Lbm p$, $p\in [1,\infty)$, for which we retain the same notation.
Furthermore, $(\sfP_t)_{t\ge0}$ is sub-Markovian (see 
\cite[Prop.~I.3.2.1]{Bouleau-Hirsch91}), since it preserves
one-sided essential bounds, namely $f\leq C$ (resp. $f\geq C$) 
$\mm$-a.e.\ in $X$ for some $C\geq 0$ (resp. $C\leq 0$) implies $\sfP_t f\leq C$ 
(resp. $\sfP_tf\geq C$) $\mm$-a.e.\ in $X$ for all  $t\geq 0$.

Finally, it is easy to check, using $L^1$-contractivity of $\sfP$, that the dual semigroup $\sfP_t^\infty:\Lbm{\infty}\to\Lbm{\infty}$
given by
\begin{equation}\label{eq:dual_P}
\int g\sfP^\infty_t f d\mm=\int f\sfP_t g d\mm,\qquad
f\in\Lbm\infty,\,\,g\in\Lbm 1
\end{equation}
is well defined. It is a contraction semigroup in $\Lbm{\infty}$, sequentially $w^*$-continuous, and it coincides with
$\sfP$ on $L^2\cap\Lbm\infty$.

\subsection{The algebra $\Algebra$}\label{subsec:algebra}

Throughout the paper we assume that an algebra $\Algebra \subset \V$ is prescribed, with
\begin{equation}\label{eq:basic_algebra}
\Algebra\subset \bigcap_{p \in [1,\infty]} \Lbm p, \quad \text{$\Algebra$ dense in $\V$}
\end{equation}
and
\begin{equation}\label{eq:stability_composition}
\Phi(f_1,\ldots,f_n)\in\Algebra\quad\text{whenever $\Phi\in C^1(\R^n)$, $f_1,\ldots,f_n\in\Algebra$.}
\end{equation}

Additional conditions on $\Algebra$ will be considered in specific sections of the paper. A particular role is played by the condition $\Algebra \subset \V_p$, for $p\in [2,\infty]$. By interpolation, if such an inclusion holds, then it holds for every $q$ between $2$ and $p$. About the inclusion $\Algebra \subset \V_p$ for $p \in [1,2]$, we prove:

\begin{lemma}\label{lem:refining-algebra}
Let $\mathcal{A} \subset \V$ be dense in $\V$  and satisfy \eqref{eq:stability_composition}. Then, there exists $\Algebra \subset \mathcal{A}$, such that \eqref{eq:basic_algebra} and \eqref{eq:stability_composition} hold, and
\begin{equation}\label{eq:algebra-density} \text{$\Algebra$ is contained and dense in $\V_p$, for every $p \in [1,2]$.} \end{equation}
\end{lemma}

In particular, without any loss of generality, we assume throughout that \eqref{eq:algebra-density} holds.

\begin{proof}
We define
\[ \Algebra = \cur{ \Phi(f)\, : \,  \text{ $f \in \mathcal{A}$, $\Phi\in\mathcal F$}} \subset \mathcal{A}, \]
where $\mathcal F$ consists of all functions $\Phi:\R\to\R$ bounded and Lipschitz, continuously differentiable and null at the origin, with
$\Phi'(x)/x$ bounded in $\R$.
By the chain rule and H\"older inequality, it follows that $\Algebra\subset\Lbm p$ for all $p\in [1,\infty]$ and that
\eqref{eq:stability_composition} holds. We address the density of $\Algebra$
in $\V_p$ for $p\in [1,2]$.

We consider Lipschitz functions $\phi_n: \R \mapsto [0,1]$ such that $\phi_n(z) = 0$ for $\abs{z}\le 1/(2n)$ and for $\abs{z} \ge 2n$,
while $\phi_n(z) = 1$ for $\abs{z} \in [1/n, n]$, and we set $\Phi_n(z) = \int_0^z \phi_n(t) dt$. Notice that $\Phi_n\equiv 0$ on $[-1/(2n),1/(2n)]$, 
that $\Phi_n$ belongs to $\mathcal F$, and that $\Phi'_n(z) = \phi_n(z) \to 1$ as $n \to \infty$.  It is easily seen, by chain rule, that
$\Phi_n(f)\to f$ in $\V_p$ as $n\to\infty$ for all $f\in\V_p$, therefore density is achieved if we show that all functions $\Phi_n(f)$ belong to the closure
of $\Algebra$. Since by assumption there exist $f_k\in\mathcal{A}$ convergent to $f$ in $\V$ (and $\mm$-a.e.\ on $X$), 
it will be sufficient to show that, for every $n \ge 1$, 
$\Phi_n(f_k)$ converge to $\Phi_n(f)$ in $\V_p$ as $k\to \infty$.

We claim that, as $k\to \infty$, $\phi_n(f_k)$ converge to $\phi_n(f)$ in $\Lbm q$ for every $q\in [1,\infty)$. To prove the claim, it suffices to
consider separately the sets $\{|f|\geq 1/(3n)\}$ and $\{|f|<1/(3n)\}$. On the first set, which has finite $\mm$-measure,
we can use $\mm$-a.e.\  and dominated convergence to achieve the thesis, taking into account the boundedness
of $\phi_n$ (since $n$ is fixed); on the second set we have
$$
|\phi_n(f_k)-\phi_n(f)|=\chi_{\cur{|f_k|\geq 1/(2n)}}|\phi_n(f_k)-\phi_n(f)|\leq
\chi_{\{|f_k-f|\geq 1/(6n)\}}\min\bigl\{2,{\rm Lip}(\phi_n)|f_k-f|\bigr\}
$$
and we can use H\"older's inequality for $q<2$ and uniform boundedness for $q\geq 2$.

To show convergence of $\Phi_n(f_k)$ to $\Phi_n(f)$ in $\V_p$ as $k \to \infty$, 
we use the following straightforward identity, valid for any $h_1,\,h_2\in\V$ and $\Phi\in\mathcal F$:
\[\Gamma( \Phi(h_1) - \Phi(h_2)) =\bra{ \Phi'(h_1) -\Phi'(h_2)}^2\Gamma(h_1,h_2)  + \Phi'(h_1)^2 \Gamma(h_1,h_1-h_2)+ \Phi'(h_2)^2 \Gamma(h_2, h_2- h_1). \]
Adding and subtracting $\Phi'(h_2)^2 \Gamma(h_1, h_1-h_2)$,
and taking $\Phi =\Phi_n$, since $0\leq \phi_n\leq 1$ we obtain the inequality
\begin{equation}\label{eq:gamma-difference}
\begin{split}
 \Gamma( \Phi_n(h_1) -\Phi_n(h_2))^{1/2} \le & 
 \abs{ \phi_n(h_1) - \phi_n(h_2)}\Gamma(h_1)^{1/4}\Gamma(h_2)^{1/4} \\
  & + 2 |\phi_n(h_1) - \phi_n(h_2))|^{ 1/2 } \Gamma(h_1)^{1/4} \Gamma(h_1 -h_2)^{1/4}\\
   & +\phi_n(h_2) \Gamma(h_1 - h_2)^{1/2}.
 \end{split}
  \end{equation}
Finally, we let $h_1 = f$ and $h_2 = f_k$ in this inequality and use the  convergence of $\phi_n(f_k)$ to $\phi_n(f)$ in every $\Lbm q$ space, for $q \in [1,\infty)$, as
well as H\"older's inequality,  to deduce that the right hand side above converges to $0$ in $\Lbm p$ as $k \to \infty$.
\end{proof}

We also deduce density in $L^p\cap L^q$-spaces, thanks to the following lemma.

\begin{lemma}\label{lem:newdensity}
There exists a countable set ${\mathscr D}\subset\Algebra$ dense in $L^p\cap\Lbm 
q$, for any $1\leq p\leq q<\infty$, and
$w^*$-dense in $\Lbm\infty$.
\end{lemma}
\begin{proof} Since $\V$ is dense in $\Lbm 2$ and we assume that $\Algebra$ is dense in $\V$, 
we obtain that $\Algebra$ is dense in $\Lbm 2$.

We consider first the case $p=q\in [2,\infty]$. Let $h\in\Lbm {p'}$.
Assuming $\int h\varphi d\mm=0$ for all $\varphi\in\Algebra$, to prove 
density in the $w^*$ topology
(and then in the strong topology if $p<\infty$) we have to prove that $h=0$.
Let $\delta>0$, set $f_\delta={\rm sign} h\chi_{\{|h|>\delta\}}$ (set 
equal to $0$ wherever $h=0$) and find an equibounded sequence
$(\varphi_n)\subset\Algebra$ convergent in $\Lbm 2$ to $f_\delta$. 
Since $(\varphi_n)$ are uniformly bounded in $\Lbm\infty$,
we obtain strong convergence to $f_\delta$ in $L^p$ for $p\in 
[2,\infty)$ and $w^*$-convergence for $p=\infty$. It follows that
$\int_{\{|h|>\delta\}}|h| d\mm=0$ and we can let $\delta\downarrow0$ to get 
$h=0$.

To cover the cases $p=q\in [1,2)$, by interpolation we need only to 
consider $p=1$. Given $f\in\Lbm 1$ nonnegative, we can
find $\varphi_n\in\Algebra$ convergent to $\sqrt{f}$ in $\Lbm 2$. It 
follows that the functions $\varphi_n^2$ belong to $\Algebra$
and converge to
$f$ in $\Lbm 1$. In order to remove the sign assumption on $f$ we split 
in positive and negative part.

Finally, in the case $p<q$ we can use the density of bounded functions to reduce ourselves to the
case of approximation of a bounded function $f\in L^p\cap\Lbm{q}$ by functions in $\Algebra$. Since 
$f$ can be approximated by equibounded functions $f_n\in\Algebra$ in $L^p$ norm, we need only
to use the fact that $f_n\to f$ also in $L^q$ norm. 

Finally, a simple inspection of the proof shows that we can achieve the same density result
with a countable subset of $\Algebra$, since $\V$ is separable.
\end{proof}

\begin{remark} \label{rem:Feller} {\rm Under the additional condition
\begin{equation}\label{eq:Feller}
\text{$\Algebra$ is invariant under the action of $\sfP_t$,}
\end{equation}
our basic assumption that $\Algebra$ is dense in $\V$
can be weakened to the assumption that $\Algebra$ is dense
in $\Lbm 2$; indeed, standard semigroup theory shows that an invariant subspace is dense in $\V$ if and only if it is dense in
$\Lbm 2$, see for instance \cite[Lemma~4.9]{AGS11b}, but also Lemma~\ref{lemma:feller-density} below. \fr
}\end{remark}

\section{Derivations}\label{sec:derivations}

Since $\Algebra$ might be regarded as an abstract space of test functions, we introduce derivations as linear operators acting on it, 
satisfying a Leibniz rule and a pointwise $\mm$-a.e.\ upper bound in terms of $\Gamma$ (even though for some results an integral
bound would be sufficient). 

\begin{definition}[Derivation]
	A derivation is a linear operator $\bb: \Algebra\to \Lbm 0$, $f\mapsto df(\bb)$, satisfying
		$$ d(fg)(\bb) =  f dg(\bb) + g df(\bb), \quad \text{$\mm$-a.e.\ in $X$, for every $f,\,g\in\Algebra$}$$
	and
			$$\abs{df(\bb)} \le g \sqGq f ,\quad \text{$\mm$-a.e.\ in $X$, for every $f\in\Algebra$}$$
	for some $g\in\Lbm 0$. The smallest function $g$ with this property will be denoted by $\abs{\bb}$.   
	For $p,\,q\in [1,\infty]$, we say that a derivation $\bb$ is in $L^p+L^q$ if $|\bb|\in \Lbm p + \Lbm q$.   
\end{definition}

Existence of the smallest function $g$ can easily be achieved using the fact that $\Lbm 0$
is a complete lattice, i.e.\ considering the supremum among all functions $f\in\Algebra$ of the expression
$\abs{df(\bb)} \Gq f^{-1/2}$ (set equal to $0$ on the set $\cur{\Gq f = 0}$).

N.\ Gigli pointed out to us that linearity and the $\mm$-a.e.\ upper bound are sufficient to entail ``locality'' and thus Leibniz and chain rules, with a proof contained in the work in preparation \cite{gigli-WIP}, akin to that of \cite[Thm.~3.5]{Ambrosio-Kirchheim}. Since our work focuses on the continuity equation and related Lagrangian flows, but not on the fine structure of the space of derivations, for the sake of simplicity, we have chosen to retain this slightly redundant definition and deduce only the validity of the chain rule.

\begin{proposition}[Chain rule for derivations]\label{prop-chain-rule-derivations}
	Let $\bb$ be a derivation and let $\Phi: \R^n \to \R$ be a smooth function, with $\Phi(0) = 0$.
	Then, for any $\ff = (f_1,\ldots,f_n)\in\Algebra^n$ there holds
	\begin{equation}
	\label{eq:chain-rule-general}
	d\Phi(\ff)(\bb) = \sum_{i=1}^n \partial_i \Phi(\ff) df_i(\bb), \quad \text{$\mm$-a.e.\ in $X$.}
	\end{equation}	
\end{proposition}

\begin{proof}
Since $\Phi(\ff)\in\Algebra$, $\bb(\Phi(\ff))$ is well defined. Arguing by induction and linearity, Leibniz rule entails that 
\eqref{eq:chain-rule-general} holds when $\Phi$ is a polynomial in $n$ variables, with $\Phi(0) =0$. Since $\ff$ is bounded, the thesis follows by approximating $\Phi$ with a sequence $(p_k)$ of polynomials, converging to $\Phi$, together with their derivatives, uniformly on compact sets. 
\end{proof}

\begin{remark}[Derivations $u\bb$]\label{rem:product-derivation}{\rm
Let $\bb$ be a derivation in $L^q$ for some $q\in [1,\infty]$ 
and let $u \in \Lbm r $, with $q^{-1}+r^{-1}\leq 1$. Then, 
$f\mapsto udf(\bb)$ defines a derivation $u\bb$ in $L^{s'}$, where $(s')^{-1}=q^{-1}+r^{-1}$, i.e.
$q^{-1} + r^{-1} +s^{-1} =1$. 
By linearity, similar remarks apply when $\bb$ is a derivation in $L^p+L^q$.}\fr
\end{remark}

\begin{example}[Gradient derivations]{\rm 
The main example is provided by derivations $\bb_g$ induced by $g\in\V$, of the form 
\begin{equation}\label{eq:defbbg}
f\in\Algebra\mapsto  df(\bb_g):=\Gamma(f,g)\in\Lbm1.
\end{equation} 
These derivations belong to $L^2$, because \eqref{eq:boundg1} yields $\abs{\bb_g}\leq\sqGq{g}$. Since $\Algebra$ is dense in $\V$, it is not difficult to show that equality holds.

By linearity, 
the $L^\infty$-module generated by this class of examples (i.e.\ finite sums 
$\sum_i\chi_i\bb_{g_i}$ with $\chi_i\in\Lbm\infty$ and $g_i\in\V$) still consists of derivations in $L^2$.}\fr
\end{example}

\begin{definition}[Divergence]\label{def:divergence}
	Let $p,\,q \in [1,\infty]$, assume that $\Algebra \subset \V_{p'} \cap \V_{q'}$ and let $\bb$ be a derivation in $L^p + L^q$. The distributional divergence $\div\bb$ is the linear operator on $\Algebra$ defined by
	\[ \Algebra \ni f \mapsto -\int df(\bb) d\mm.\]
	We say that $\div\bb \in \Lbm p+\Lbm q$ if the distribution $\div\bb$ is induced by $g\in \Lbm p+\Lbm q$, i.e.\ 
	$$\int df(\bb) d\mm = - \int f g d\mm, \quad \text{for all $f\in\Algebra$}.$$
	Analogously, we say that $\dbneg\in \Lbm p$ if there exists a nonnegative $g\in\Lbm p$ such that
	$$\int df(\bb) d\mm \le \int f g d\mm,  \quad \text{for all $f\in\Algebra$, $f\ge 0$}.$$
\end{definition}

Notice that we impose the additional condition $\Algebra \subset \V_{p'} \cap \V_{q'}$, to ensure integrability of $df(\bb)$.

As we did for $\abs{\bb}$, we define $\dbneg$ as the smallest nonnegative function $g$ in $\Lbm p$ for which the inequality above holds.
Existence of the minimal $g$ follows by a simple convexity argument, because the class of admissible $g$'s is convex and closed
in $\Lbm p$ (if $p=\infty$, one has to consider the $w^*$-topology).

\begin{example}[Divergence of gradients]\label{example:derivation-gradient}
{\rm The distributional divergence of the ``gradient'' derivation $\bb_g$ induced by $g\in\V$ as in \eqref{eq:defbbg} coincides with the Laplacian $\Delta g$,
still understood in distributional terms.}
\end{example}

Although the definitions given above are sufficient for many purposes, the following extensions will be technically useful in Section~\ref{sec:apriori} (and in Section~\ref{sec:commutators} for the case $q \in [1,\infty)$).

\begin{remark}[Derivations in $L^2+L^\infty$ extend to $\V$] \label{remark:bounded-derivations}
{\rm When a derivation $\bb$ belongs to $L^2+L^\infty$, we can use the density of $\Algebra$ in $\V$ to extend uniquely
$\bb$ to a derivation, still denoted by $\bb$, defined on $\V$, with values in the space $\Lbm 1+\Lbm 2$ and continuous. For all $u\in\V$, it still satisfies
$$
|du(\bb)|\leq \abs{\bb} \sqrt{\Gamma(u)},\qquad\text{$\mm$-a.e.\ in $X$.}
$$
A similar remark holds for derivations belonging to $L^q + L^\infty$, for some $q \in [1,\infty)$, if $\Algebra$ is assumed to be dense in $\V_r$, for some $r \in [1,\infty)$ with $q^{-1} +r^{-1} \le 1$. The extension is then a continuous linear operator $\bb$ mapping $\V_r$ into $\Lbm{s'} + \Lbm 2$, where $q^{-1} + r^{-1} + s^{-1} = 1$.}\fr
\end{remark}

By a similar density argument as above, any derivation $\bb$ could be extended uniquely to a derivation defined on $\V$, with values in the space $\Lbm 0$. However, such an extension is not useful when dealing with integral functionals defined initially on $\Algebra$, e.g.\ that of divergence or weak solutions to the continuity equation, because these are not continuous with respect to the topology of $\Lbm 0$. Therefore, we avoid in what follows to consider such an extension, except for the case in the remark above.

We conclude this section noticing that if $\bb$ is a derivation $L^2 +L^\infty$, with $\div \bb \in \Lbm 2 + \Lbm \infty$, 
the following integration by parts formula can be proved by approximation with functions in $\Algebra$:
\begin{equation}\label{eq:extended_integration}
\int du(\bb) f d\mm=-\int df(\bb) u d\mm+\int uf \div\bb d\mm\qquad\forall f\in\Algebra,\,\,\forall u\in\V.
\end{equation}

\section{Existence of solutions to the continuity equation}\label{sec:existence}

Let $I=(0,T)$ with $T\in (0,\infty)$.
In this section we prove existence of weak solutions to the continuity equation
\begin{equation}\label{eq:continuity-equation}
\frac{d}{dt}u_t+\div (u_t\bb_t)=w_t u_t\qquad\text{in $I\times X$,}
\end{equation}
under suitable growth assumptions on $\bb_t$ and its divergence.

\begin{remark}\label{rem:algebra-lipschitz}\rm{
Starting from this section, we always assume that $\Algebra$ is contained in $\V_\infty$, i.e.\ $\Gamma(f) \in \Lbm \infty$ for every $f \in \Algebra$. We are motivated by the examples and by the clarity that we gain in the exposition, although some variants of our results could be slightly reformulated and proved without this assumption.}\fr
\end{remark}

Before we address the definition of \eqref{eq:continuity-equation}, let us remark that a Borel family of derivations $\bb = (\bb_t)_{t \in I}$ is by 
definition a map $t \mapsto\bb_t$, taking values in the space of derivations on $X$, such that there exists a Borel function $g: I \times X \mapsto [0,\infty)$ satisfying
$$
|\bb_t|\leq g(t,\cdot), \quad \text{$\mm$-a.e.\ in $X$, for a.e.\ $t\in I$.}
$$
As in the autonomous case we denote by $|\bb|$ the smallest function $g$ (in the $\mathcal{L}^1\otimes \mm$-a.e.\ sense) with this property. We say that Borel family of derivations $(\bb_t)_{t\in I}$ belongs to
$L_t^r(L^p_x+L^q_x)$ if $|\bb|\in L^r_t(L^p_x+L^q_x)$.

\begin{definition}[Weak solutions to the continuity equation with initial condition $\bar u$]\label{def:continuity-equation}
Let $p,\,q\in [1,\infty]$, $\bar u\in L^p \cap \Lbm q$, let $(\bb_t)_{t\in I}$ be a Borel family of derivations in $L^1_t(L^{p'}_x + L^{q'}_x)$ and 
let $w\in L^1_t(L^{p'}_x + L^{q'}_x)$. 
We say that $u\in L^\infty_t (L^p_x \cap L^q_x)$ solves \eqref{eq:continuity-equation} with the initial condition $u_0=\bar u$
in the weak sense if
\begin{equation}\label{eq:fokker-planck}
 	 \int_0^T \int \sqa{-\psi'\vphi -\psi d\vphi(\bb_t)-w_t} u_t d\mm dt=  \psi(0)\int \vphi\bar u d\mm
	 \end{equation}
	 for all $\vphi\in\Algebra$ and all $\psi\in C^1([0,T])$ with $\psi(T)=0$.
\end{definition}

As usual with weak formulations of PDE's, the definition above has many advantages, the main one is to provide a meaning to 
\eqref{eq:continuity-equation} without any regularity assumption on $u$. 
Notice that, without the assumption $\Algebra \subset \V_\infty$, one could define  
e.g.\ weak solutions $u \in L^\infty_t (L^2_x)$ to the equation associated to $\bb$ in $L^1_t(L^\infty_x)$.

In order to prove the mass-conservation property of solutions to the continuity equation we assume the existence
of $(f_n)\subset\Algebra$ satisfying
\begin{equation}\label{eq:mass_preserving}
\text{$0\leq f_n\leq 1$, $f_n\uparrow 1$ $\mm$-a.e.\ in $X$, $\sqrt{\Gamma(f_n)}\rightharpoonup 0$ weakly-$*$ in $\Lbm\infty$.}
\end{equation}

The following theorem is our main result about existence: we address the case $w =0$ only, the general case 
following from a Duhamel's principle that we do not pursue here. 

\begin{theorem}[Existence of weak solutions in $L^\infty_t(L^1_x \cap L^2_x)$]\label{thm:existence-lp}
Assume that $\Algebra\subset\V_\infty$, let $\bar u\in L^1\cap \Lbm r$ for some $r\in [2,\infty]$ 
and let $\bb=(\bb_t)_{t\in I}$ be a Borel family of derivations with $|\bb| \in L^1_t(L^2_x+L^\infty_x)$, 
$\div \bb \in  L^1_t(L^2_x + L^\infty_x)$, and $\dbneg\in L^1_t(L^\infty_x)$.\\ 
Then, there exists a weakly continuous in $[0,T)$ (in duality with $\Algebra$) 
solution $u\in L^\infty_t(L^1_x \cap L^r_x)$ of \eqref{eq:continuity-equation} according to
Definition~\ref{def:continuity-equation} with $u_0=\bar u$ and $w_t=0$.
Furthermore, if $\bar u\geq 0$, we can build a solution $u$ in such a way that
$u_t\geq 0$ for all $t\in I$. Finally, if \eqref{eq:mass_preserving} holds, then
\begin{equation}\label{eq:mass_preserving1}
\int u_t d\mm=\int \bar u d\mm\qquad\forall t\in [0,T).
\end{equation}
\end{theorem}

To prove existence of a solution $u$ to \eqref{eq:continuity-equation} with $w_t=0$, we rely on a suitable approximation of the equation. 
Following a classical strategy, 
we approximate the original equation by adding a diffusion term, 
i.e.\ we solve, still in the weak sense of duality with
test functions $\psi(t)\vphi(x)$,
	\begin{equation}
	\label{eq:continuity-equation-coercive} \partial_t u_t +\div\bra{u_t \bb_t} = \sigma \Delta u_t,
	\end{equation}
where $\sigma >0$. By Hilbert space techniques, we show existence of a solution with some extra regularity,  namely $u \in L^2(I; \V)$.
We use this extra regularity to derive a priori estimates and then we take weak limits as $\sigma\downarrow0$.

Let us remark that such a technique forces the introduction of stronger assumptions than those known to prove existence in particular 
classes of spaces (e.g.\ Euclidean or Gaussian), where ad hoc methods are available: here, we trade some strength in the result 
in favour of generality.

\subsection{Auxiliary Hilbert spaces}

In all what follows, we consider the Gelfand triple
\[	\V \subset\Lbm 2 = (\Lbm 2)^* \subset\V',\]
i.e.\ we regard $\V$ as a dense subspace in $\V'$ (proper if $\V \neq \Lbm 2$) by means of 
$$\phi \mapsto (\phi^*: f \mapsto \int f \phi d\mm).$$ 
Notice that this is different from the identification $\V \sim \V'$ provided by the Riesz-Fischer theorem applied to the Hilbert space $\V$ (which has been applied to $\Lbm 2$ instead).

Given a vector space $F$, we introduce a space of $F$-valued test functions on $I$, namely
	\[ \Phi_F  := \text{span}\cur{\psi \cdot \phi\, : \, \psi \in C^1([0,T]),\,\, \psi(T) = 0,\,\, \phi \in F }.\]
We notice that, for every $\vphi \in \Phi_F$, the function $t \mapsto \vphi_t$ is Lipschitz and continuously differentiable from $I$ to $F$, and there exists $\vphi_0 = \lim_{t \down 0} \vphi_t$ in $F$ (while $\lim_{t \up T} \vphi_t = 0$ in $F$ by construction). 

Assuming that $F$ is a separable Hilbert space, starting from $\Phi_F$ one can consider completions with respect to different norms. The classical space
\[ L^2(I; F), \qquad \ang{ \vphi, \tilde{\vphi} }_{L^2(F)} = \int_I \ang{ \vphi_t, \tilde{\vphi}_t}_F dt,\]
is indeed the closure of $\Phi_F$ with respect to the norm induced by the scalar product above. Similarly, the space $H^1(I; F)$ is obtained 
by completing $\Phi_F$ with respect to the norm
\[  \ang{ \vphi, \tilde{\vphi}}_{H^1(F)} = \int_I \ang{ \vphi_t, \tilde{\vphi}_t}_F +\ang{ \frac{d}{dt}\vphi_t, \frac{d}{dt} \tilde{\vphi}_t}_F dt.\]
Arguing by mollification as in the case $F = \R^n$, it is not difficult to prove that $H^1(I; F)= W^{1,2}(I; F)$, where the latter space is defined as the subspace of functions $\vphi \in L^2(I;F)$ such that there exists $g \in L^2(I;F)$, which represent the distributional derivative of $\vphi$, i.e.
\[ \int_I \ang{ \vphi_t, \frac{d}{dt} \tilde{\vphi}_t }_F dt = - \int_0^T \ang{g_t, \tilde{\vphi}_t}_F dt, \quad \text{for every $\tilde{\vphi}\in \Phi_F$ with $\tilde{\vphi}_0 = 0$}.\]
\subsection{Existence under additional ellipticity}

We address now the existence of some $u \in L^2(I;\V)$ that solves the following weak formulation of \eqref{eq:continuity-equation-coercive} with the
initial condition $u_0=\bar u$: 
 	\begin{equation}\label{eq:fokker-planck-variational}
 	 \int_0^T \int \sqa{-\partial_t \vphi_t -d\vphi_t(\bb_t)} u_t + \sigma \Gamma(\vphi_t, u_t) d\mm dt=  \int\vphi_0\bar u d\mm
	 \qquad\forall \vphi \in\Phi_\Algebra.
 	 \end{equation}
We still assume that $\Algebra\subset\V_\infty$, and that $\sigma\in (0,1/2]$, $|\bb|\in L^\infty_t(L^2_x+L^\infty_x)$, 
$\dbneg\in L^\infty_t(L^\infty_x)$, $\bar u\in\Lbm 2$. 
Notice that the assumptions on $|\bb|$ and $\div \bb^-$ are stronger than that in Theorem~\ref{thm:existence-lp}, but only with respect to integrability in time.

We obtain, together with existence, the a priori estimate:
\begin{equation}\label{apriorisigma}
\nor{e^{-\lambda t} u}_{L^2(I;\V)}\leq\frac{\nor{\bar u}_2}{\sigma}
\quad\text{with}\quad\lambda:=\frac 12\|\dbneg\|_{\infty}+\sigma.
\end{equation}

To this aim, we change variables setting $h_t=e^{-\lambda t}u_t$ and we pass to this equivalent weak formulation
\begin{equation}\label{eq:fokker-planck-variationalstar}
 	 \int_0^T \int \sqa{-\partial_t \vphi_t +\lambda -d\vphi_t(\bb_t)} h_t + \sigma \Gamma(\vphi_t, h_t) d\mm dt=  \int \vphi_0\bar u d\mm
	 \qquad\forall \vphi \in\Phi_\Algebra.
 	 \end{equation}

From now on we shall use the notation $\mmt$ for the product measure $\Leb{1}\otimes\mm$ in $I\times X$.
Existence of $h$ is a consequence of J.-L.\ Lions' extension of Lax-Milgram Theorem, whose statement is recalled below 
(see \cite[Thm.~III.2.1, Corollary III.2.3]{Showalter97}) applied with $H=L^2(I;\V)$,
$V=\Phi_\Algebra$ endowed with the norm
\begin{equation}\label{eq:norm-phi}
	 \nor{\vphi}_V^2 = \nor{\varphi}_{L^2(I;\V)}^2 + \nor{\varphi_0}^2_2,
\end{equation}
$$
B(\vphi,h) = \int \sqa{-\partial_t \vphi +\lambda \vphi -d\vphi(\bb)} h + \sigma \Gamma(\vphi,h) d\mmt,\qquad \ell(\varphi) = \int \vphi_0\bar u d\mm.
$$

\begin{theorem}[Lions]\label{thm:lions}
Let $V$, $H$ be respectively a normed and a Hilbert space, with $V$ continuously embedded in $H$, with $\nor{v}_H\leq\nor{v}_V$
for all $v\in V$, and let
$B:V\times H\to\R$ be bilinear, with $B(v,\cdot)$ continuous for all $v\in V$. If $B$ is coercive, namely there exists $c>0$ satisfying
$B(v,v)\geq c\nor{v}^2_V$ for all $v\in V$, then for all $\ell\in V'$ there exists $h\in H$ such that $B(\cdot,h)=\ell$ and
\begin{equation}\label{apriorih}
\nor{h}_H\leq\frac{\nor{\ell}_{V'}}{c}.
\end{equation}
\end{theorem}

Let us start by proving continuity, thus let $\varphi \in V$. The linear functional $h \mapsto B(\varphi, h)$ is $L^2(I;\V)$-continuous for all
$\phi\in V$, since we can estimate $\abs{B(\varphi,h)}$ from above with
\[\nor{h}_{L^2(I;\V)} \sqa{\nor{\partial_t\varphi}_{L^1_t(L^2_x)} + \lambda \|\varphi\|_{L^1_t(L^2_x)} +
\nor{\bb}_{L^2_t(L^2_x+L^\infty_x)}\|\sqGq{\varphi}\|_{L^\infty_t(L^2_x\cap L_x^\infty)} 
+ \sigma \|\sqGq{\varphi} \|_{L^1_t(L^2_x)}} .\]
The functional $\ell$ satisfies $\|\ell\|_{V'}\leq \nor{\bar u}_2$, immediately from the definition of $\nor{\cdot}_V$ in \eqref{eq:norm-phi}.

To conclude the verification of the assumptions of Theorem~\ref{thm:lions}, we show coercivity (here the change of variables we did and the choice of $\lambda$ play a role):
	\begin{equation}
	\label{eq:lambda-coercivity}
	\begin{split}
		\int \sqa{\lambda\varphi -d\varphi(\bb)} \varphi d\mmt
		& = \lambda \nor{\varphi}^2_{L^1_t(L^2_x)} -\frac{1}{2}\int d\varphi^2(\bb)d\mmt\\
		& \ge \lambda\nor{\varphi}^2_{L^1_t(L^2_x)} -\frac{1}{2}\int\varphi^2\dbneg d\mmt \\
		& \ge (\lambda-\frac 12 \|\dbneg\|_{\infty})\nor{\varphi}_{L^1_t(L^2_x)}^2=\sigma \nor{\varphi}_{L^1_t(L^2_x)}^2.
	\end{split}
	\end{equation}
Since $\varphi \in V=\Phi_\Algebra$, it holds $\partial_t \varphi_t^2 = 2 \varphi_t \partial_t \varphi_t$ 
and $\int  - 2\varphi_t \partial \varphi d\mmt =  \int \varphi_0^2 d\mm$. Hence, inequality \eqref{eq:lambda-coercivity} entails that
\[ \int \sqa{ -\partial_t\varphi+\lambda\varphi - d\varphi(\bb)}\varphi + \sigma \Gq{ \varphi} d\mmt \ge
\frac 12 \int \varphi_0^2 d\mm+
 \sigma\nor{ \varphi}_{L^1_t(L^2_x)}^2 + \sigma\|\sqGq{\varphi}\|^2_{L^1_t(L^2_x)},\]
Since $\sigma\leq 1/2$, it follows from these two inequalities that 
\begin{equation}\label{apriorih1}
B(\varphi,\varphi)\geq\sigma\|\varphi\|_V^2.
\end{equation}

Finally, \eqref{apriorisigma} follows at once from \eqref{apriorih} and \eqref{apriorih1}, taking into account that $\nor{\ell}_{V'}\leq\nor{\bar u}_2$.

\subsection{A priori estimates}\label{sec:apriori}

In this section we still consider weak solutions to
\begin{equation}\label{eq:fokker-planck-variational2}
 	 \int_0^T \int -\sqa{\partial_t \vphi_t + d\vphi_t(\bb_t)} u_t + \sigma \Gamma(\vphi_t, u_t) d\mm dt=  \int \vphi_0 \bar u d\mm
	 \qquad\forall \vphi \in\Phi_\Algebra
 	 \end{equation}
obtained in the previous section. In order to state pointwise in time $L^r$ estimates in space, we use the following remark.

\begin{remark}[Equivalent formulation]\label{rem:form}
{\rm
Assuming $\Algebra\subset\V_\infty$, $u\in L^2(I;\V)$ and $|\bb|\in L^1_t(L^2_x+L^\infty_x)$, an equivalent formulation of \eqref{eq:fokker-planck-variational2}, in terms of absolute continuity and pointwise derivatives w.r.t.\ time, is the following: we are requiring that, for every $f \in\Algebra$, $t\mapsto\int fu_t d\mm$
is absolutely continuous in $I$ and that its a.e.\ derivative in $I$ is $\int (df(\bb_t)u_t + \sigma\Gamma(f,u_t)) d\mm$. 
In addition, the Cauchy initial condition is encoded by
\begin{equation}\label{eq:initial_condition}
\lim_{t \down 0} \int f u_t d\mm = \int f \bar u d\mm, \quad \text{for every $f\in\Algebra$}
\end{equation}
(notice also that $\bar u$ is uniquely determined by \eqref{eq:initial_condition}, thanks to the density of $\Algebra$ in $\Lbm 2$).

Indeed, it is clear that the definition above implies
the formula for the distributional derivative, because for absolutely continuous functions the two concepts coincide; the converse can be 
obtained using the set ${\mathscr D}$ of Lemma~\ref{lem:newdensity}
to redefine $u_t$ in a negligible set of times in order to get a weakly continuous representative in the duality with $\Algebra$, see
\cite[Lemma~8.1.2]{Ambrosio-Gigli-Savare05} for details.\fr
}\end{remark}

We prove, by a suitable approximation, the following result:

\begin{theorem}\label{thm:apriori} 
Assume that $\Algebra\subset\V_\infty$, $|\bb|\in L^\infty_t(L^2_x+L^\infty_x)$, $\div\bb\in L^\infty_t(L^2_x+L^\infty_x)$, 
$\dbneg\in L^\infty_t(L^\infty_x)$, and that the initial condition $\bar u$ belongs
 to $L^p\cap \Lbm q$, with $1\le p \le 2 \le q \le \infty$. Then there exists a weakly continuous (in duality with $\Algebra$) 
 solution  
 $$u\in L^\infty_t(L^p_x\cap L^q_x)\cap L^2(I;\V)$$ 
to \eqref{eq:fokker-planck-variational2} satisfying:
\begin{equation}
 	\label{eq:apriori-estimates-lp}
 	\sup_{(0,T)}\nor{u_t^\pm}_r\le\nor{\bar u^\pm}_{ r} \exp\bra{(1-\frac 1 r)\nor{\dbneg}_{L^1_t(L^\infty_x)} },
\end{equation}
for every $r\in [p,q]$. In particular, if $\bar u \geq 0$, then $u_t\geq 0$ for all $t\in (0,T)$.
\end{theorem}

At this stage, it is technically useful to introduce another formulation of the continuity equation, 
suitable for $\V$-valued solutions $u$, with the derivation acting on $u$.

\begin{remark}[Transport weak formulation]\label{rem:transport_weak}{\rm
Using \eqref{eq:extended_integration} we obtain an equivalent weak formulation of
\eqref{eq:fokker-planck-variational2}, namely
\begin{equation}\label{eq:fokker-planck-variational222}
 	 \int_0^T \int -u_t\partial_t\vphi_t +du_t(\bb_t)\vphi_t +u_t\vphi_t\div\bb_t
	 + \sigma \Gamma(\vphi_t, u_t) d\mm dt=  \int \vphi_0 \bar u d\mm
	 \qquad\forall \vphi \in\Phi_\Algebra.
 	 \end{equation}
\fr}\end{remark}

\begin{remark}[Basic formal identity]
{\rm Before we address the proof of the a priori estimates, let us remark that these, and uniqueness as well, strongly rely on the formal identity 
\begin{equation}\label{eq:formal-renormalization}
 \frac{d}{dt} \int \beta(u_t) d\mm = -\int \sqa{ \beta'(u_t) u_t - \beta(u_t) }\div\bb_t d\mm,
\end{equation}
which comes from chain rule in \eqref{eq:fokker-planck} and the formal identity $\int \div (\beta(u_t)\bb_t)=0$.
To establish existence, however, this computation is made rigorous by approximating the PDE (by vanishing viscosity, in Theorem \ref{thm:apriori},
or by other approximations), while to obtain uniqueness in Section~\ref{sec:proof-uniqueness} we approximate $u$. In both cases technical assumptions on $\bb$ will be needed.\fr
}\end{remark}

A natural choice in \eqref{eq:formal-renormalization} is a convex ``entropy'' function $\beta:\R\to\R$ with $\beta(0)=0$. In order to give a meaning
to the identity \eqref{eq:formal-renormalization} also when $\beta$ is not $C^1$ ($z\mapsto z^+$ will be a typical choice of $\beta$) we define
\begin{equation}\label{def:Legendre}
{\cal L}_\beta(z):=
\begin{cases}
z\beta_+'(z)-\beta(z) &\text{if $z\geq 0$;}\\
z\beta_-'(z)-\beta(z) &\text{if $z\leq 0$,}
\end{cases}
\end{equation}
where we write $\beta_\pm'(z) := \lim_{y \to z^\pm} \beta'(y)$.
Notice that the convexity of $\beta$ and the condition 
$\beta(0)=0$ give that ${\cal L}_\beta$ is nonnegative;
for instance, if $z\geq 0$, there holds
$$
\beta(0)=0\geq\beta(z)-z\beta_-'(z)\geq \beta(z)-z\beta_+'(z).
$$
The argument for $z\leq 0$ follows from ${\cal L}_{\tilde{\beta}} (-z) ={\cal L}_\beta (z)$, where $\tilde{\beta}(z)=\beta(-z)$. 

In order to approximate $\beta$ with functions with linear growth in $\R$, we will consider the approximations
\begin{equation}\label{def:betan}
\beta_n(z):=
\begin{cases} 
\beta(-n)+\beta'_-(-n)(z+n) &\text{if $z<-n$;} \\
\beta(z) &\text{if $-n\leq z\leq n$;}\\
\beta(n)+\beta'_+(n)(z-n) &\text{if $z>n$,}
\end{cases}
\end{equation}
that satisfy ${\cal L}_{\beta_n}(z) = {\cal L}_{\beta}(-n\lor z \land n)$, so that ${\cal L}_{\beta_n}\uparrow {\cal L}_\beta$
as $n\to\infty$. On the other hand, in order to pass from smooth to nonsmooth $\beta$'s, we
will also need the following property, whose proof is elementary and motivates our
precise definition of ${\cal L}_\beta$ in \eqref{def:Legendre}:
\begin{equation}\label{eq:stabLeb}
\limsup_{i\to\infty}{\cal L}_{\beta_i}\leq
{\cal L}_\beta\quad\text{whenever $\beta_i$ are convex, $\beta_i\to\beta$ uniformly on compact sets.}
\end{equation}

{\bf Proof of Theorem~\ref{thm:apriori}.} By Remark~\ref{rem:form} we can assume with no loss of generality that
$t\mapsto u_t$ is weakly continuous in $[0,T)$, in the duality with $\Algebra$.

We assume first that a weak solution 
$u$ satisfies the strong continuity property
\begin{equation}\label{eq:strong_continuity_time0}
\lim_{t\down 0}u_t=\bar u\quad\text{in $\Lbm 2$.}
\end{equation}
We shall remove this assumption at the end of the proof.

We claim that for any convex function $\beta:\R\to [0,\infty)$
satisfying $\beta(0)=0$ and $\beta'(z)/z$ bounded on $\R$, the inequality
\begin{equation}\label{eq:basicapriori} 
\frac{d}{dt} \int \beta(u_t) d\mm \le \int  {\cal L}_\beta(u_t)\div\bb_t^- d\mm\end{equation}
holds in the sense of distributions in $(0,T)$. The assumption on the behaviour of $\beta$ near to the origin is needed to ensure
that both $\beta(u)$ and ${\cal L}_\beta(u)$ belong to $L^2_t(L^1_x)$, since at present we only know that $u\in L^2_t(L^2_x)$.
By approximation, taking \eqref{def:betan} and \eqref{eq:stabLeb} into account, we can assume with no
loss of generality that $\beta\in C^1$ with bounded derivative.

In the proof of \eqref{eq:basicapriori}, motivated by the necessity to get strong differentiability w.r.t.\ time,
we shall use the regularization $u^s_t:=\sfP_su_t$ and the following elementary remark 
(\cite[Prop.~III.1.1]{Showalter97}).

\begin{remark}\label{rem:weak-derivative-absolute-continuity} {\rm Let $X$ be a Banach space, 
let $f,\,g \in L^1((0,T); X)$ satisfy $\partial_t f = g$ in the weak sense, namely
$$
-\int_0^T\psi'(t)\int \phi(f) d\mm dt=\int_0^T\psi(t)\int \phi(g) d\mm dt,
$$
for every $\psi\in C^1_c(0,T)$, $\phi\in\mathcal D\subset X^*$, dense w.r.t.\ the $\sigma(X^*, X)$-topology. 
Then, $f$ admits a unique absolutely continuous representative from $I$ to $X$ and this representative is 
strongly differentiable a.e.\ in $I$, with derivative equal to $g$.}
\fr
\end{remark}

Notice that $X$ may not have the Radon-Nikodym property so that
 it might be the case that not all absolutely continuous maps with values in $X$ are strongly 
 differentiable a.e.\ in their domain. Indeed, we are going to apply it with $X = \Lbm 1 +\Lbm 2$, so that $X^* = L^2 \cap\Lbm\infty$, 
 and $\mathcal D= \Algebra$. 

It is immediate to check, replacing $\vphi$ in \eqref{eq:fokker-planck-variational222} by $\sfP_s\vphi$ and using \eqref{eq:extended_integration},
that for any $s>0$ the function $t\mapsto u^s_t$ solves
$$
\frac{d}{dt} u^s_t+\div (\bb_t u^s_t)=\sigma\Delta u^s_t+\Commutator^s_t
$$ 
in the weak sense of duality with $\Algebra$, where $\Commutator^s_t$ is the commutator between semigroup and divergence, namely
$$
\Commutator^s_t:=\div (\bb_t u^s_t)-\sfP_s(\div (\bb_tu_t)).
$$

Therefore, using \eqref{eq:extended_integration} once more and expanding 
$$
\Commutator^s_t=u^s_t\div\bb_t+du^s_t(\bb_t)-\sfP_s(u_t\div\bb_t)-\sfP_s(du_t(\bb_t))
$$
we may use the assumption $\div\bb\in L_t^\infty(L^2_x +L^\infty_x)$ and the continuity of derivations to
obtain that $\Commutator^s_t\to 0$ strongly in $L^2_t(L^1_x + L^2_x)$ as $s\downarrow0$. Similarly, 
expanding $\div (\bb_t u^s_t)=u_t^s\div\bb_t+du^s_t(\bb_t)$ and using the regularization estimate \eqref{eq:regularization} to
estimate the Laplacian term in the derivative of $u_t$ we obtain
$\frac{d}{dt} u_t^s\in L^2_t(L^1_x+L^2_x)$ in the weak sense of duality with $\Algebra$, therefore 
by Remark~\ref{rem:weak-derivative-absolute-continuity}, $t\mapsto u^s_t$ is strongly 
$(L^1+L^2)$-differentiable a.e.\ in $(0,T)$ and absolutely continuous.

Since $\beta$ is convex, we can start from the inequality
$$
\int\beta(u^s_t)d\mm-\int\beta(u^s_{t'})d\mm\leq \int \beta'(u^s_t)(u^s_t-u^s_{t'}) d\mm 
$$
and we can use the uniform boundedness of $\beta'(z)$ and of $\beta'(z)/z$ to obtain that 
$\beta'(u^s_t)\in L^2_t(L^2_x\cap L^\infty_x)$, hence
$$
\int\beta(u^s_t)d\mm-\int\beta(u^s_{t'})d\mm\leq g(t)\biggl|\int_{t'}^t\|\frac{d}{dr} u^s_r\|_{L^1+L^2} dr\biggr| 
$$
with $g(t)=\|\beta'(u^s_t)\|_{L^2\cap L^\infty}\in L^2(0,T)$. Since (again by the convexity of $\beta$) $t\mapsto \int\beta(u^s_t) d\mm$
is lower semicontinuous, a straightforward application of a calculus lemma \cite[Lemma~2.9]{AGS11a} entails that $t\mapsto \int\beta(u^s_t) d\mm$ is absolutely
continuous in $(0,T)$ and that
\begin{equation}\label{eq:veryhard}
\frac{d}{dt}\int\beta(u^s_t) d\mm=\int\beta'(u^s_t) \bigl[-\div (\bb_t u^s_t)+\sigma\Delta u^s_t+\Commutator^s_t\bigr] d\mm
\end{equation}
for a.e.\ $t\in (0,T)$.

Since $\beta(u^s_t)\in\V$ we get $\int\beta'(u^s_t)\Delta u^s_t d\mm=-\int\beta''(u^s_t)\Gamma(u^s_t) d\mm\leq 0$, hence
 we may disregard this term. Using twice the chain rule and $\div\bb\in L^2_t(L^2_x+L^\infty_x)$, ${\cal L}_\beta(u)\in L^2_t(L^1_x)$ gives
\begin{eqnarray*}
\frac{d}{dt}\int \beta(u^s_t)d\mm&\leq& -\int \beta'(u^s_t)u^s_t\div\bb_t+d\beta(u^s_t)(\bb_t)d\mm+\int \beta'(u^s_t)\Commutator_t^s d\mm\\
&=&-\int (\beta'(u^s_t)u^s_t-\beta(u^s_t))\div\bb_t d\mm+\int \beta'(u^s_t)\Commutator_t^s d\mm\\
&\leq&\int (\beta'(u^s_t)u^s_t-\beta(u^s_t))\div\bb_t^- d\mm+\int \beta'(u^s_t)\Commutator_t^s d\mm.
\end{eqnarray*}
Eventually, since $\beta'(u^s_t)$ are bounded in $L^2_t(L^2_x\cap L^\infty_x)$, uniformly w.r.t.\ $s$, 
we let $s\downarrow0$ to obtain \eqref{eq:basicapriori} (convergence of the first term in the right hand side follows from dominated convergence and convergence in $L^2(\mm)$ of $u^s_t \to u_t$). 

\noindent
{\bf Proof of \eqref{eq:apriori-estimates-lp}.} Let $r \in [p,q]$, let $\beta(z) = (z^+)^r$ and notice that $\mathcal{L}_\beta(z)=(r-1)\beta(z)$. We cannot apply directly \eqref{eq:basicapriori} to $\beta$, because $\beta'(z)/z$ is unbounded near $0$. If $r<2$, we let
$$
\beta_n(z) := 
\begin{cases}
\displaystyle{\frac{(z^+)^2}{2\epsilon^{2-r}}} &\text{if $z\leq\epsilon$;}\\ \\
\displaystyle{(z^+)^r-\frac {\epsilon^r}2} &\text{if $z\geq\epsilon$,}
\end{cases}
$$
where $\veps = 1/n$, so that $\beta_n$ are convex, $\beta_n'(z)/z$ is bounded, ${\cal L}_{\beta_n}\leq\beta_n$ and $\beta_n\to\beta$ as $n\to\infty$.

If $r \ge 2$, we use the approximations $\beta_n$ in \eqref{def:betan},
that satisfy ${\cal L}_{\beta_n}(z) = {\cal L}_{\beta}(z \land n)$, so that we still have ${\cal L}_{\beta_n}\leq (r-1)\beta_n$, and $\beta_n'(z)/z$
is bounded. 

Now in both cases it is sufficient
to apply Gronwall's lemma to the differential inequality \eqref{eq:basicapriori} with $\beta=\beta_n$ 
(here we use the assumption \eqref{eq:strong_continuity_time0}, to ensure that the value at $0$ is the expected one) and then let $n \to \infty$ to conclude
with Fatou's lemma.

The correspondent inequalities for $\beta(z) = (z^-)^r$ are settled similarly.

Finally, the assumption \eqref{eq:strong_continuity_time0} can be removed considering the solutions $u^\eps_t$
relative to the same initial condition and to the derivations
$$
\bb^\eps_t:=\begin{cases} \bb_t &\text{if $t\in [\epsilon, T)$;}\\
0&\text{if $t\in (0,\epsilon)$.}
\end{cases}
$$
Since $u^\eps_t$ coincides with $\sfP_{\sigma t}\bar u$ for $t\in (0,\epsilon)$, \eqref{eq:strong_continuity_time0}
is fulfilled. Then, we can take weak limits in $L^\infty_t(L^p_x \cap L^q_x) \cap L^2(I;\V)$ as $\eps\downarrow0$ to obtain a function 
$u$ satisfying the desired properties.

\subsection{Vanishing viscosity and proof of Theorem~\ref{thm:existence-lp}} \label{sec:vanishing-viscosity}

Let $\bb = (\bb_t)_{t\in I}$ and $\bar u\in L^1\cap \Lbm r$ ($r\ge 2$) satisfy the assumptions of Theorem~\ref{thm:existence-lp}. Let $\delta>0$, let $\rho$ be a mollifying kernel in $C^1_c(0,1)$ and set $\bb^\delta_t := \int_0^1 \bb_{t+s\delta} \rho(s)ds$ (where $\bb_t =0$ for $t > T$) i.e.\ we let
\[ \varphi \mapsto d\varphi(\bb^\delta_t) =  \int_0^1  d\varphi(\bb_{t+s\delta}) \rho(s) ds.\]
Since $|\bb|\in L^1_t(L^2_x)$, $\div \bb \in L^1_t(L^2_x + L^\infty_x)$, it follows that $|\bb^\delta|\in L^\infty_t(L^2_x)$, 
$\div\bb^\delta\in L^\infty_t(L^2_x + L^\infty_x)$ and the assumption $\div \bb^- \in L^1_t(L^1_x\cap L^\infty_x)$ entails  
$(\div \bb^\delta)^- \in L^\infty_t(L^1_x\cap L^\infty_x)$. Moreover, as $\delta \downarrow0$, $d\varphi (\bb^\delta)$ 
converges to $d\varphi (\bb)$ in $L^1_t(L^2_x +L^\infty_x)$, for every $\varphi\in\Algebra$ and $\nor{(\div \bb^\delta)^{-}}_{L^1_t(L^\infty_x)}$ converges to $\nor{(\div \bb)^{-}}_{L^1_t(L^\infty_x)}$.

For fixed $\delta >0$, consider a sequence $u^n = u^{\delta,n}$ of solutions 
to \eqref{eq:fokker-planck-variational2} with $\bb^\delta$ in place of $\bb$, $\sigma=1/n$, $n\geq 2$, 
as provided by Theorem~\ref{thm:apriori} with $p=1$ and $q=r$, and notice that \eqref{apriorisigma} gives 
$$\frac 1n\nor{e^{-(1+\lambda)t} u^n}_{L^2(I;\V)}\leq \nor{\bar u}_2,$$ 
so that $v^n:=u^n/n$ is bounded in $L^2(I;\V)$.
We would like to pass to the limit as $n\to\infty$ in 
\begin{equation}\label{eq:fokker-planck-variational4}
 	 \int_0^T \int -\sqa{\partial_t \vphi_t + d\vphi_t(\bb_t^\delta)} u^n_t +  \Gamma(\vphi_t, v^n_t) d\mm dt=  \int \vphi_0 \bar u d\mm
	 \qquad\forall \vphi \in\Phi_{\Algebra}.
 	 \end{equation}
Inequality \eqref{eq:apriori-estimates-lp} entails that $(u^n)_n$ is bounded in $L^\infty_t (L^1_x \cap L^r_x)$ and so $v^n$ weakly converges to
$0$ in $L^2(I;\V)$. In addition, there exists a subsequence $n(k)$ such that $(u^{n(k)})_k$ converges,  
in duality with $L^1_t (L^2_x + L^\infty_x)$, to some $u:= u^\delta \in L^\infty_t (L^1_x \cap L^r_x)$. This gives that 
$u^\delta$ is a weak solution to the continuity equation with $\bb^\delta$ in place of $\bb$. 

We then let $\delta \downarrow0$ and extract again a subsequence $\delta(k)$ such that $(u^{\delta(k)})_k$ converges,  
in  duality with $L^ 1_t (L^{r'}_x + L^\infty_x)$, to some $u:= L^\infty_t (L^1_x \cap L^r_x)$
and is a weak solution to the continuity equation, thus concluding the proof of 
Theorem~\ref{thm:existence-lp}, except for conservation of mass.

Finally, we prove conservation of mass for any weak solution to the continuity equation, assuming existence
of $f_n\in\Algebra$ as in \eqref{eq:mass_preserving}. The proof is based on the simple observation that
our assumptions on $\bb$ and $u$ imply $\cc:=u\bb\in L^1_t(L^1_x)$, and therefore
$$
\lim_{n\to\infty}\int_0^T\int |df_n(\cc_t)| d\mm dt=0.
$$
Since 
$$
\lim_{n\to\infty}\int u_t f_n d\mm=\int u_t d\mm\quad\forall t\in [0,T)\qquad\text{and}\qquad
\frac{d}{dt}\int u_t f_n d\mm=\int df_n(\cc_t) d\mm,
$$
we conclude that $\int u_t d\mm=\int \bar u d\mm$ for all $t\in [0,T)$.

\section{Uniqueness of solutions to the continuity equation}\label{sec:commutators}
In this section, we provide conditions that ensure uniqueness, in certain classes, for the continuity equation: these involve further regularity of 
$\bb$, expressed in terms of bounds on its divergence and its deformation (introduced below), density assumptions of $\Algebra$ in $\V_p$ and the validity of 
inequalities which correspond, in the smooth setting, to integral bounds on the gradient of the kernel of $\sfP$.

\begin{definition}[$L^p$-$\Gamma$ inequality]
Let $p\in [1,\infty]$. We say that the $L^p$-$\Gamma$ inequality holds if there exists $c_p>0$ satisfying
$$
\Big\|\sqGq{ \sfP_t f} \Big\|_p \le \frac{c_p}{\sqrt{t}} \nor{f}_p, \quad\text{for every $f \in L^2\cap\Lbm p$,  $t \in (0,1)$.}
$$
\end{definition}

Although the $L^p$-$\Gamma$ inequality is expressed for $t \in (0,1)$, from its validity and $L^p$ contractivity of $\sfP$, we easily deduce that
\begin{equation} \label{eq:p-reverse-gamma-inequalities}
\nor{ \sqGq{ \sfP_t f} }_p \le c_p (t\land 1)^{-1/2}  \nor{f}_p,\quad \text{for every $f \in L^2\cap \Lbm p$, $t \in (0,\infty)$.}
\end{equation}
Notice also that, thanks to \eqref{eq:regularization2}, the $L^2$-$\Gamma$ inequality always holds, with $c_2=1/\sqrt{2}$. 
By semilinear Marcinkiewicz interpolation, 
we obtain that if the $L^p$-$\Gamma$ inequality holds then, for every $q$ between $2$ and $p$, the $L^q$-$\Gamma$ inequality holds as well.

\begin{definition}[Derivations with deformations of type $(r,s)$]\label{def:deformation}
Let $q\in [1,\infty]$, let $\bb$ be a derivation in $L^q+L^\infty$, with $\div\bb\in\Lbm q+\Lbm\infty$, let $r,\,s\in [1,\infty]$
with $q^{-1}+r^{-1}+s^{-1}=1$ and assume that $\Algebra$ is dense both in $\V_r$ and in $\V_s$. We say that the deformation of $\bb$ is of type $(r,s)$ 
if there exists $c \ge 0$ satisfying
\begin{equation}\label{eq:ineq_deformation}
\biggl|\int D^{sym}\bb(f,g) d\mm \biggr| \le c \|\sqrt{\Gamma(f)}\|_r \|\sqrt{\Gamma(g)}\|_s  ,
\end{equation}
for all $f\in \V_r$ with $\Delta f\in L^r\cap\Lbm 2$ and all $g\in \V_s$ with $\Delta g\in L^s\cap\Lbm 2$, where
	\begin{equation}
		\label{eq:distributional-deformation}
		 \int D^{sym}\bb(f,g) d\mm := -\frac{1}{2} \int \sqa{ df(\bb)\Delta g + dg(\bb) \Delta f - (\div\bb) \Gamma(f,g) } d\mm.
	\end{equation}
	We let $\nor{D^{sym}\bb}_{r,s}$ be the smallest constant $c$ in \eqref{eq:ineq_deformation}.
\end{definition}

The density assumption of $\Algebra$ in $\V_r$ and $\V_s$ is necessary to extend the derivation $\bb$ to all of $\V_r$ and $\V_s$, by Remark~\ref{remark:bounded-derivations}. Notice that the expression $\int D^{sym}\bb(f,g) d\mm$ is symmetric with respect to $f$, $g$, so the role of $r$ and $s$ above can be interchanged.

\begin{remark}[Deformation in the smooth case]\label{rem:deformation-smooth}{\rm
Let $(X,\ang{\cdot,\cdot})$ be a compact Riemannian manifold, let $\mm$ be its associated Riemannian volume and let $\Gamma(f,g) = \ang{\nabla f,\nabla g}$. Let $df(\bb) = \ang{b,\nabla f}$ for some smooth vector field $b$ and let $Db$ be the covariant derivative of $b$. The expression
\[ \ang{\nabla g,\nabla \ang{b,\nabla f} } + \ang{ \nabla f, \nabla \ang{b,\nabla g} }-\ang{b,\nabla \ang{\nabla f,\nabla g}}  = 
\ang{Db\nabla g,\nabla f}  + \ang{Db\nabla f,\nabla g}\]
gives exactly twice the symmetric part of the tensor $Db$, i.e.\ $2\ang{D^{sym} bf,g}$. Integrating over $X$ and then integrating by parts, we obtain twice the expression in \eqref{eq:distributional-deformation}, so that the derivation $\bb$ associated to a smooth field $b$ is of type $(r,s)$ if $\abs{D^{sym}b} \in \Lbm q$, 
where $q \in [1,\infty]$ satisfies $q^{-1} + r^{-1} + s^{-1} =1$.}\fr
\end{remark}

\begin{theorem}[Uniqueness of solutions]\label{thm:uniqueness}
Let $1<s\leq r<\infty$, $q\in (1,\infty]$ satisfy $q^{-1} + r^{-1} + s^{-1}= 1$. Assume the existence of $(f_n) \subset \Algebra$ as in \eqref{eq:mass_preserving} and that, for $p\in\cur{r,s}$, $\Algebra$ is dense in $\V_p$ and the $L^p$-$\Gamma$ inequality holds. Let $\bb=(\bb_t)_{t\in (0,T)}$ be a Borel family of derivations, with
\[ \text{ $|\bb|\in L^1_t(L^q_x+L^\infty_x)$, $\div\bb\in L^1_t(L^q_x+L^\infty_x)$ and $\nor{ D^{sym}\bb_t}_{r,s} \in L^1(0,T)$.}\]

Then, there exists at most one weak solution $u$ in $(0,T)\times X$ to the continuity equation
$\frac d{dt} u_t+\div (u_t\bb_t)=0$ in the class
	$$ \cur{u\in L^\infty_t(L^r_x\cap L^2_x):\ \text{$t\mapsto u_t$ is weakly continuous in $[0,T)$}},$$
for every initial condition $\bar u\in L^r\cap\Lbm 2$.
\end{theorem}

The proof of this result is given in Section~\ref{sec:proof-uniqueness} and relies upon the strong convergence to $0$ as
$\alpha\downarrow0$ of the commutator between divergence and action of the semigroup
\begin{equation}
 \Commutator^\alpha(u_t,\bb_t) :=  \div ((\sfP_\alpha u_t) \bb_t) - \sfP_\alpha( \div (u_t \bb_t)) ,
 \end{equation}
proved in Lemma~\ref{lemma:commutator-estimate} in the next Section. We end this section with some comments on the density assumption on $\Algebra$.

\begin{remark}[On the density of $\Algebra$ in $\V_p$]\label{rem:role-of-density}\rm{
The assumption that $\Algebra \subset \V_p$ is dense for $p \in \cur{r,s}$ is fundamental to show that the semigroup approximation $t \mapsto \sfP_\alpha u_t$ is a solution to another continuity equation, \eqref{eq:ce-alpha} below. This follows by the extension of the derivation on $\V_p$ provided by Remark~\ref{remark:bounded-derivations}. One could argue that the invariance condition \eqref{eq:Feller} is sufficient to define $\bb(\sfP_\alpha f)$, whenever $f \in \Algebra$: indeed Theorem~\ref{thm:uniqueness} holds, assuming \eqref{eq:Feller} in place of the density of $\Algebra$ in $\V_p$, and the same proof goes through, with minor modifications (e.g.\ in Definition~\ref{def:deformation} above we require $f$, $g \in \Algebra$). In view of Remark~\ref{rem:Feller}, one could also wonder whether \eqref{eq:Feller} and the $L^p$-$\Gamma$ inequality are sufficient to entail density in $\V_p$: the next lemma provides a partial affirmative answer 
(see Proposition~\ref{prop:BE-density} for an application of the lemma, assuming curvature lower bounds).}\fr
\end{remark}

\begin{lemma}\label{lemma:feller-density}
Let $p \in [2,\infty)$, assume that \eqref{eq:Feller} and the $L^p$-$\Gamma$ inequality hold and that
\begin{equation}\label{eq:semigroup-approximating} \limsup_{t \down 0} \nor{\sqGq{ \sfP_t f} }_p  \le \nor{ \sqGq{ f} }_p, \quad 
\text{for every $f \in \V_p$.}\end{equation}
Then, $\Algebra$ is dense in $\V_p$.
\end{lemma}

\begin{proof} Let $f \in \V_p$. Notice first that, since $\sfP_t f$ converges to $f$ in $\V$ as $t\downarrow0$, Fatou's lemma gives
\[ \nor{ \sqGq{ f} }_p \le \liminf_{t \down 0}  \nor{ \sqGq{ \sfP_t f} }_p\]
which combined with \eqref{eq:semigroup-approximating} gives convergence of $\Gamma\bra{ \sfP_t f}^{1/2}$ to 
$\Gamma\bra{f}^{1/2}$ in $\Lbm p$.

To prove density, we let $f \in \V_p$ and consider the functions $\Phi_n:\R\to\R$, with derivative $\phi_n$, introduced in Lemma~\ref{lem:refining-algebra}: since, by the chain rule, $\Phi_n(f)$ converge to $f$ in $\V_p$, it is sufficient to approximate each $\Phi_n(f)$ in $\V_p$ with elements of $\Algebra$.

We first show that $\lim_{t\downarrow0} \Phi_n(\sfP_t f) = \Phi_n(f)$ in $\V_p$. Since convergences in $\V$ and in $\Lbm p$ are obvious, we prove 
$\Gamma(\Phi_n(\sfP_t f)  - \Phi_n(f))^{1/2} \to 0$ in $\Lbm p$. We let $h_1 =\sfP_t f$ and $h_2 = f$ in \eqref{eq:gamma-difference} to get
\begin{eqnarray}\label{eq:verysplit}
 \Gamma( \Phi_n(\sfP_t f) -\Phi_n(f))^{1/2} &\le & 
 \abs{ \phi_n(\sfP_t f) - \phi_n(f)} \Gamma(f)^{1/4}\Gamma(\sfP_tf)^{1/4}\nonumber\\
  & + &2|\phi_n(\sfP_t f) -\phi_n(f)|^{1/2} \Gamma(f)^{1/4} (\Gamma(f)^{1/4}+\Gamma(\sfP_tf)^{1/4})\\
   & + & \phi_n(f) \Gamma(\sfP_t f - f)^{1/2}.\nonumber
  \end{eqnarray}
  To handle the integral of the $p$-power of the last term in the right hand side, 
  we notice that, since $\Gamma(\sfP_t f)^{1/2}$ converge to $\Gamma(f)^{1/2}$ in $\Lbm p$, 
  they converge also in $L^p(\mm')$ with $\mm'=\phi_n(f)^p\mm$. Since $\mm'$ is finite we obtain that $\Gamma(\sfP_t f)^{p/2}$ are equi-integrable with respect
to $\mm'$ and the Lebesgue-Vitali convergence ensures convergence to 0. The first term can be handled similarly,
adding and subtracting  $\abs{ \phi_n(\sfP_t f) - \phi_n(f)}^p \Gamma(f)^{p/4}\Gamma(f)^{p/4}$ and using
the dominated convergence, since $0\leq \phi_n\leq 1$; the integral of the $p$-th power of the last term can be estimated with dominated 
convergence for $\int |\phi_n(\sfP_t f) -\phi_n(f)|^{p/2} \Gamma(f)^{p/2} d\mm$ and with the same argument we used for the first term for 
$\int |\phi_n(\sfP_t f) -\phi_n(f)|^{p/2} \Gamma(f)^{p/4} \Gamma(\sfP_t f)^{p/4} d\mm$.

We proceed then to approximate $\Phi_n(\sfP_t f)$ in $\V_p$ by elements of $\Algebra$, at fixed $n\geq 1$ and $t>0$.
Let $(f_k) \subset \Algebra$ converge to $f$ in $L^2 \cap \Lbm p$. We show that $\Phi_n( \sfP_t f_k)$ converge to $\Phi_n( \sfP_t f)$ in $\V_p$. Notice that  
$\Phi_n( \sfP_t f_k)$ belong to $\Algebra$, because of \eqref{eq:Feller} and \eqref{eq:stability_composition}. Since convergence in $L^ 2 \cap \Lbm p$ holds, convergence in $\V_p$ follows again by \eqref{eq:gamma-difference} with $h_1 = \sfP_t f_k$ and $h_2 = \sfP_t f$, because
\begin{equation*}
\begin{split}
 \Gamma( \Phi_n(\sfP_t f_k) -\Phi_n(\sfP_t f ))^{1/2} \le & 
 \abs{ \phi_n(\sfP_t f_k) - \phi_n(\sfP_t f)} \Gamma(\sfP_t f)^{1/4}\Gamma(\sfP_tf_k)^{1/4}\\
  & + 2|\phi_n(\sfP_t f_k) -\phi_n(\sfP_t f)|^{1/2} \Gamma(\sfP_t f)^{1/4} (\Gamma(\sfP_tf_k)^{1/4} +\Gamma(\sfP_t f)^{1/4})\\
   & +\phi_n(\sfP_t f) \Gamma(\sfP_t f_k - \sfP_t f )^{1/2}.
 \end{split}
  \end{equation*}
By the $L^2$-$\Gamma$ inequality and the $L^p$-$\Gamma$ inequality, $\Gamma(\sfP_t f_k)^{1/2}$ converges to $\Gamma(\sfP_t f)^{1/2}$
in $L^2 \cap \Lbm p$ as $k \to \infty$ and we can argue as we did in connection with \eqref{eq:verysplit} to obtain that 
$\Gamma( \Phi_n(\sfP_t f_k) -\Phi_n(\sfP_t f ))^{1/2}\to 0$ in $\Lbm p$.
\end{proof}

Actually, the proof above entails the following result. Let $p \in [1,\infty)$, assume that the $L^p$-$\Gamma$ inequality holds, and let $\Algebra \subset \V$ satisfy \eqref{eq:stability_composition}, \eqref{eq:Feller}, \eqref{eq:semigroup-approximating}, and dense in $L^2 \cap \Lbm p$. Then $\Algebra$ is dense in $\V_p$. Finally, notice that this gives another proof of Remark~\ref{rem:Feller}.



\subsection{The commutator lemma}

We first collect some easy consequences of the $L^r$-$\Gamma$ inequality, for some $r \in (1,\infty)$, 
which allow for an approximation of the derivation $\bb$, via the action of 
$\sfP_\alpha$, as expressed in the next proposition. We denote by $\bB^\alpha$ the linear operator thus obtained, to stress the fact that it is not a derivation.

\begin{proposition}\label{prop:consequences-gradient-estimates}
Let $r,\,s\in (1,\infty)$, $q\in (1,\infty]$ satisfy $q^{-1} + r^{-1} + s^{-1}= 1$. 
Let $\bb$ be a derivation in $L^q+L^\infty$ assume that $\Algebra \subset \V_r$ is dense and that the $L^r$-$\Gamma$ inequality holds.
\begin{enumerate}[(i)]
\item For every $\alpha\in (0,\infty)$, the map
	$$ \Algebra\ni f \mapsto  d (\sfP_\alpha f) (\bb)$$
 extends uniquely to $\bB^\alpha\in \mathscr{L}(L^r\cap\Lbm 2, \Lbm {s'}+\Lbm 2)$, with 
 \begin{equation}\label{eq:b^t-estimate} 
 \nor{\bB^\alpha} \le \max\{c_r,c_2\} (\alpha\land 1)^{-1/2} \nor{\bb}_{L^q+L^\infty}.
 \end{equation}
\item For all $f \in L^r\cap\Lbm 2$ the map $\alpha\mapsto \bB^\alpha(f)$ is continuous from $(0,\infty)$ to $\Lbm {s'}+\Lbm 2$ 
and, if $\Delta f\in L^r\cap\Lbm 2$, it is $C^1((0,\infty); \Lbm {s'}+\Lbm 2)$, with
	\[ \frac{d}{d\alpha}\bB^\alpha( f) = \bB^\alpha(\Delta f). \]
\item Assume that $u\in L^r\cap\Lbm 2$, $\div\bb\in\Lbm q+\Lbm\infty$. Then,
	\begin{equation}
	\label{eq:divbetaptub}
	 \div( \beta(\sfP_\alpha u) \bb) = \beta(\sfP_\alpha u) \div \bb + \beta'(\sfP_\alpha u) \bB^\alpha(u)\in\Lbm {s'}+\Lbm 2
	 \end{equation}
	 for all $\alpha>0$ and all $\beta\in C^1(\R)\cap {\rm Lip}(\R)$ with $\beta(0)=0$.
	 In particular \eqref{eq:divbetaptub} with $\beta(z)=z$ gives
\begin{equation}
	\label{eq:divptub}
	 \div( (\sfP_\alpha u) \bb) = (\sfP_\alpha u)\div \bb + \bB^\alpha(u)\in\Lbm {s'}+\Lbm 2.
	 \end{equation}
\item Assume $u \in L^r\cap\Lbm 2$ and $\div \bb \in \Lbm q+\Lbm\infty$. Then
	$\Commutator^\alpha(\sfP_\delta u, \bb) \in \Lbm{s'}+\Lbm 2$ for every $\delta>0$ and
	\begin{equation}\label{eq:convergenza_sui_buoni}
	\lim_{\alpha\downarrow0} \nor{\Commutator^\alpha(\sfP_\delta u, \bb)}_{L^{s'}+L^2} = 0.
	\end{equation}
	 \end{enumerate}
\end{proposition}

\begin{proof}
\textit{(i)}. By Remark~\ref{remark:bounded-derivations}, if $\cc$ is a derivation in $L^q$, then we can extend it to a linear operator on $\V_r$, thus $d(\sfP_\alpha f)(\cc)$ is well defined. Since the $L^r$-$\Gamma$ inequality holds, for every $f\in\Algebra$, we get
$$
\nor{d(\sfP_\alpha f)(\cc)}_{s'}\le \nor{\cc}_q \nor{ \sqGq{ \sfP_\alpha f}}_r \le c_r (\alpha \land 1)^{-1/2}\nor{\cc}_q \nor{f}_r. 
$$
Analogously, if $\cc$ is a derivation in $L^\infty$, $d(\sfP_\alpha f)(\cc)$ is well defined and there holds
$$
\nor{d(\sfP_\alpha f)(\cc)}_2\le\nor{\cc}_\infty\nor{ \sqGq{\sfP_\alpha f}}_2 \le c_2 (\alpha \land 1)^{-1/2} \nor{\cc}_\infty\nor{f}_2.
$$
This gives $\nor{\bB^\alpha(f)}_{L^{s'}+L^2} \le \max\{c_r,c_2\}(\alpha \land 1)^{-1/2}\nor{\bb}_{L^q+L^\infty}
\|f\|_{L^r\cap L^2}$ on $\Algebra$.
By density of $\Algebra$ in $L^r\cap\Lbm 2$, this provides the existence of $\bB^s$ and the estimate on its norm.

\textit{(ii)}. The semigroup law and the uniqueness of the extension give
\[ \bB^{\alpha+\sigma}(f)= \bB^\alpha (\sfP_\sigma f), \text{ for every $f\in L^r\cap\Lbm 2$, $\alpha,\,\sigma \in (0,\infty)$.} \]
Then, continuity follows easily, combining identity with \eqref{eq:b^t-estimate} and the strong continuity of $\sfP_s$:
\[ \nor{ \bB^{\alpha+\sigma}(f) - \bB^\alpha(f) }_{L^{s'}+L^2} \le \max\{c_r,c_2\}(\alpha \land 1)^{-1/2}\nor{\bb}_{L^q+L^\infty}
\nor{\sfP_\sigma f - f}_{L^r\cap L^2}.\]
A similar argument shows differentiability if $\Delta f\in L^r\cap\Lbm 2$.

\textit{(iii)}. We obtain \eqref{eq:divptub} by \eqref{eq:extended_integration}. By the chain rule, the identity \eqref{eq:divbetaptub} follows.

\textit{(iv)}. To prove that $\Commutator^\alpha(\sfP_\delta u, \bb)\in \Lbm{s'}+\Lbm 2$, it is sufficient to apply \eqref{eq:divptub} twice, to get
\[ -\Commutator^\alpha(\sfP_\delta u,\bb) = \sfP_\alpha\sqa{ (\sfP_\delta u ) \div \bb }+ \sfP_\alpha(\bB^\delta(u))  - 
(\sfP_{\alpha+\delta} u) \div\bb - \bB^{\alpha+\delta}(u)\in\Lbm{s'}+\Lbm 2. \]
By strong continuity of $\alpha\mapsto\sfP_\alpha$ at $\alpha=0$ and continuity of $\alpha\mapsto\bB^\alpha(u)$ in
$(0,\infty)$, the same expression shows that $\Commutator^\alpha(\sfP_\delta u,\bb) \to 0$ in $\Lbm{s'}+\Lbm 2$ as $\alpha\downarrow0$.
\end{proof}

We are now in a position to state and prove the following crucial lemma.

\begin{lemma}[Commutator estimate]\label{lemma:commutator-estimate}
Let $r,\,s\in (1,\infty)$, $q\in (1,\infty]$ satisfy $q^{-1} + r^{-1} + s^{-1}= 1$. 
Let $\bb$ be a derivation in $L^q+L^\infty$ of type $(r,s)$ with $\div\bb\in \Lbm q+\Lbm  \infty$.
Assume that $\Algebra$ is dense in $\V_p$ and that the $L^p$-$\Gamma$ inequality holds, for $p\in\cur{r,s}$.
Then
 	\begin{equation}
	\label{eq:commutator-estimate}
	\nor{\Commutator^\alpha(u,\bb) }_{L^{s'}+L^2} \le c \nor{u}_{L^r\cap L^2} 
	\sqa{\nor{ D^{sym}\bb}_{r,s} +  \nor{\div \bb}_{L^q+L^\infty} } 
		\end{equation}
for all $u\in L^r\cap\Lbm 2$, and all $\alpha\in (0,1)$,
where $c$ is a constant depending only on the constants $c_r$, $c_s$ in \eqref{eq:p-reverse-gamma-inequalities} and
the constants $c^\Delta_r$ and $c^\Delta_s$ in \eqref{eq:p-reverse-Delta-inequalities}.\\
Moreover, $\Commutator^\alpha (u,\bb) \to 0$ in $\Lbm{s'}+\Lbm 2$ as $\alpha\downarrow0$.
\end{lemma}

\begin{proof} 
For brevity, we introduce the notation $g^\alpha := \sfP_\alpha g$. By duality and density, inequality \eqref{eq:commutator-estimate} is equivalent to the validity of 
\begin{equation}
\label{eq:commutator-duality}
 \int df^\alpha(\bb) u d\mm - \int df(\bb) u^\alpha d\mm 
\le c  \sqa{ \nor{ D^{sym}\bb}_{r,s} +  \nor{\div b}_{L^q+L^\infty} } \nor{u}_{L^r\cap L^2} \nor{f}_{L^s\cap L^2},
\end{equation}
for every $f$ of the form $f=\sfP_\veps \varphi$, for some $\varphi\in\Algebra$, $\veps>0$. Since both sides are continuous in $u$ with respect to 
$L^r\cap\Lbm 2$ convergence, it is also enough to establish it in a dense set: we let therefore $u =\sfP_\delta v$ for some $v\in\Algebra$, $\delta >0$.

We also notice that, by Proposition~\ref{prop:consequences-gradient-estimates}, we know that for such a choice of $u$, 
$\Commutator^\alpha(u,\bb) \to 0$ in $\Lbm{s'}+\Lbm 2$ as $\alpha\downarrow0$. Thus, once \eqref{eq:commutator-estimate} is obtained, 
the same convergence as $\alpha\downarrow0$ holds for every $u\in L^r\cap\Lbm 2$, from a standard density argument.

Then, we have to estimate
\[ \int df^\alpha(\bb) u d\mm - \int df(\bb) u^\alpha d\mm = F(\alpha) - F(0),\]
where we let $F(\sigma) = \int df^\sigma(\bb)u^{\alpha-\sigma} d\mm$, for $\sigma \in [0,\alpha]$. Our assumption on $f = \sfP_\veps \varphi$ entails, 
via Proposition~\ref{prop:consequences-gradient-estimates}, that the map $\sigma \mapsto df^\sigma (\bb) =\bB^{\veps}(\varphi^\sigma)$ is 
$C^1([0,\alpha],\Lbm{s'}+\Lbm 2)$, with
\[ \frac{d}{d\sigma} \sqa{df^\sigma(\bb)} = \bB^\veps( \Delta \varphi^\sigma ).\]
On the other hand, \eqref{eq:p-reverse-Delta-inequalities} entails that $\Delta u  = \Delta \sfP_\delta v\in L^r\cap\Lbm 2$ 
and so $\sigma \mapsto u^\sigma$ in $C^1([0,\alpha], L^r\cap\Lbm 2)$. Thus, we are in a position to apply Leibniz rule to obtain
\[  F(\alpha) - F(0) =  \int_0^\alpha \biggl(\int  \bB^\veps( \Delta \varphi^\sigma ) u^{\alpha-\sigma} - df^\sigma(\bb) \Delta u^{\alpha-\sigma} d\mm\biggr) d\sigma.\]
By applying \eqref{eq:divptub} with $\Delta \varphi^\sigma$ in place of $u$, we integrate by parts to obtain
\[ \int \bB^\veps( \Delta\varphi^\sigma ) u^{\alpha-\sigma} d\mm = - \int  \Delta f^\sigma  du^{\alpha-\sigma}(\bb) + (\div \bb) (\Delta f^\sigma) u^{\alpha-\sigma}d\mm.\]

We now estimate separately the terms
\[ I= -\int  \Delta f^\sigma  du^{\alpha-\sigma}(\bb)+ df^\sigma(\bb) \Delta u^{\alpha-\sigma} d\mm, \qquad II := - \int(\div \bb) (\Delta f^\sigma) u^{\alpha-\sigma} d\mm,\]
at fixed $\sigma \in (0,\alpha)$ and then integrate over $\sigma$.

To handle the first term, we add and subtract $\int (\div \bb)\Gamma(f^\sigma, u^{\alpha-\sigma})d\mm$, and thus recognize twice the deformation of $\bb$, 
applied to $f^\sigma$ and $u^{\alpha-\sigma}$, which are admissible functions in the sense of Definition~\ref{def:deformation}, because of 
\eqref{eq:p-reverse-gamma-inequalities} and \eqref{eq:p-reverse-Delta-inequalities}:
\[ I = 2 \int  D^{sym}\bb(f^\sigma, u^{\alpha-\sigma})d\mm - \int (\div\bb) \Gamma(f^\sigma, u^{\alpha-\sigma})d\mm.\]
We use the assumption on $ D^{sym}\bb$, $\div \bb$ and $L^r$-$\Gamma$ and $L^s$-$\Gamma$ as well as $L^2$-$\Gamma$
inequalities to obtain that
\[\abs{I} \le \sqa{2\nor{ D^{sym}\bb}_{r,s} + \nor{\div\bb}_{L^q+L^\infty}} \frac{c}{\sqrt{\alpha(\alpha-\sigma)} } \nor{ f}_{L^s\cap L^2} \nor{u}_{L^r\cap L^2},\]
with $c=c_r+c_s+c_2$. To handle integration over $\sigma \in (0,\alpha)$, we use
\[ \int_0^\alpha \frac{d\sigma}{\sqrt{\sigma(\alpha-\sigma)} } = \pi. \]

To estimate the second term, we add and subtract
\[ \int (\div \bb) (\Delta f^\sigma)  u^\alpha d\mm = \frac{d}{d\sigma} \int (\div \bb)  f^\sigma u^\alpha d\mm,\]
obtaining
\[ II = \int (\div\bb) (\Delta f^\sigma) (u^\alpha - u^{\alpha-\sigma}) d\mm - \frac{d}{d\sigma } \int (\div \bb) f^\sigma  u^{s} d\mm.\]
We then estimate the first part of $II$ by means of \eqref{eq:p-reverse-Delta-inequalities} and Corollary~\ref{coro:stein}, to get
\[ 
\frac{c^\Delta}{\sigma}\min\cur{2,c^\Delta\log\bra{ 1 + \frac{\sigma}{\alpha-\sigma}} } \nor{f}_{L^s\cap L^2} \nor{u}_{L^r\cap L^2},\]
with $c^\Delta=c^\Delta_s+c^\Delta_r+c^\Delta_2$.

The remaining part of $II$ is estimated once we integrate over $\sigma \in (0,\alpha)$, as
\[ -\int_0^\alpha\frac{d}{d\sigma} \int (\div\bb) f^\sigma  u^\alpha d\mm d\sigma = \int \div
\bb (f - f^\alpha) u^\alpha\le 2 \nor{\div \bb}_{L^q+L^\infty} \nor{f}_{L^s\cap L^2} \nor{u}_{L^r\cap L^2}.\]

To conclude, we notice that
\begin{eqnarray*} \int_0^\alpha \min\cur{ \frac{2}{\sigma}, \frac{c^\Delta}{\sigma} \log\bra{ 1 + \frac{\sigma}{\alpha-\sigma}} } d\sigma&\le& 
\max\{2,c^\Delta\}\int_0^\alpha \min\cur{ \frac 1 \sigma , \frac {1 } {\alpha-\sigma} } d\sigma\\&=& 2\log 2\max\{2,c^\Delta\},
\end{eqnarray*}
thus the proof of \eqref{eq:commutator-duality} is complete.
\end{proof}

\begin{remark}[Time-dependent commutator estimate]{\rm
By integrating the commutator estimate with respect to time, we can achieve a similar estimate for time-dependent
derivations $\bb$ of type $(r,s)$ satisfying
$$
\text{ $|\bb| \in L^1_t(L^q_x+L^\infty_x)$, $\div \bb \in L^1_t(L^q_x+L^\infty_x)$ and $\nor{ D^{sym}\bb_t}_{r,s} \in L^1(I)$,}
$$
still  assuming the validity of the $L^p$-$\Gamma$ inequalities for $p\in\cur{r,s}$:
 	$$
	\int_I\nor{\Commutator^\alpha(u_t,\bb_t) }_{L^{s'}+L^2}\,dt \le c 
	\nor{u}_{L^\infty_t(L^r_x\cap L^2_x)} \sqa{\int_I \nor{ D^{sym}\bb_t}_{r,s} +  \nor{\div\bb_t}_{L^q+L^\infty} dt} 
	$$
	for all $u\in L^\infty_t(L^r_x\cap L^2_x)$ and $\alpha\in (0,\infty)$.
Moreover, dominated convergence gives 
\begin{equation}
	\label{eq:commutator-estimate-cont}
\lim_{\alpha\downarrow0}\int_I\nor{\Commutator^\alpha (u_t,\bb_t)}_{L^{s'}+L^2}\,dt =0.
\end{equation}
}\fr
\end{remark}

\subsection{Proof of Theorem~\ref{thm:uniqueness}}\label{sec:proof-uniqueness}

The proof of Theorem~\ref{thm:uniqueness} is similar to that of Theorem~\ref{thm:apriori}, but it crucially exploits 
Lemma~\ref{lemma:commutator-estimate} to show that the error terms are negligible. 

Let $(f_n)\subset\Algebra$ be a sequence given by \eqref{eq:mass_preserving}. Starting from 
$|z|^{1+r/s}$, we define $\beta$ as in \eqref{def:betan}, namely
$$
\beta(z):=
\begin{cases} 1+\frac{r+s}{s}(z-1) &\text{if $z>1$;}\\
 |z|^{1+r/s} &\text{if $|z|\leq 1$;}\\
1-\frac{r+s}{s}(z+1) &\text{if $z<-1$,}
\end{cases}
$$
so that ${\cal L}_\beta\leq (r/s)\beta$ and $\beta$ has linear growth at infinity.

By the linearity of the equation we can assume $\bar u=0$ and the goal is to prove that $u=0$. 
We first extend the time interval $I = (0,T)$ to $(-1, T)$, setting $\bb_t = 0$ for $t \in (-1,0)$ and given the weakly continuous (in duality with
$\Algebra$) solution in $[0,T)$, with $u \in L^\infty(L^r_x\cap L^2_x)$, we extend it to a weakly continuous solution in $(-1, T)$, setting $u_t = 0$ for $t\in (-1,0)$. 

For every $\alpha>0$, let $u^\alpha_t =\sfP_\alpha u_t \in L^\infty(L^r_x\cap L^2_x)$. As in the proof of Theorem~\ref{thm:apriori},
replacing $\vphi$ in \eqref{eq:fokker-planck-variational222} by $\sfP_s\vphi$ (recall Remark~\ref{rem:role-of-density}) we can check 
that $t\mapsto u^\alpha_t$ is a weakly continuous solution to the continuity equation
\begin{equation}\label{eq:ce-alpha} \partial_t u^\alpha_t + \div( u^\alpha_t \bb_t ) =  \Commutator^\alpha (u_t, \bb_t). \end{equation}
By \eqref{eq:divptub} in Proposition~\ref{prop:consequences-gradient-estimates} and \eqref{eq:commutator-estimate} in  
Lemma~\ref{lemma:commutator-estimate}, this equation entails that 
$$
\frac{d}{dt} u^\alpha_t=\Commutator^\alpha (u_t,\bb_t)-\div( u^\alpha_t \bb_t )\in L^1_t(L^{s'}_x+L^2_x)
$$
for a.e.\  $t\in (-1,T)$. Since $t\mapsto \int f_n\beta(u^\alpha_t) d\mm$ is lower semicontinuous (because $\beta$ is convex and $t\mapsto u_t$ is 
weakly continuous) and since $|\beta'(z)|\sim |z|^{r/s}$ near the origin and $r\geq s$ imply that $\beta'(u^\alpha_t)$ is uniformly bounded
in $L^s\cap\Lbm{2}$, we can argue as in the proof of \eqref{eq:veryhard}
to obtain that $t\mapsto\int f_n\beta(u^\alpha_t) d\mm$ is absolutely continuous and 
$$
\frac{d}{dt}\int f_n\beta(u^\alpha_t) d\mm=
\int f_n\beta'(u^\alpha_t)\frac{d}{dt} u^\alpha_t d\mm=
\int f_n\beta'(u^\alpha_t)\Commutator^\alpha (u^\alpha_t,\bb_t)-f_n\beta'(u^\alpha_t)\div( u^\alpha_t \bb_t ) d\mm
$$
for a.e.\ $t\in I$. Now, setting $\Psi_n(t,\alpha):=\int f_n\beta(u^\alpha_t) d\mm$,
identities \eqref{eq:divbetaptub} and \eqref{eq:divptub} in Proposition~\ref{prop:consequences-gradient-estimates} give
	$$
         \frac{d}{dt}\Psi_n(t,\alpha)=\int f_n\beta'(u^\alpha_t)\Commutator^\alpha (u^\alpha_t,\bb_t) d\mm	
	- \int f_n\div( \beta(u^\alpha_t) \bb_t) + f_n{\cal L}_\beta(u^\alpha_t)\div\bb_t d\mm
	$$
for a.e.\ $t\in I$. Hence, denoting $L_t:=(r/s)\|\div\bb_t^-\|_\infty\in L^1(-1,T)$, we can use the inequality ${\cal L}_\beta\leq (r/s)\beta$ to get
$$
         \frac{d}{dt}\Psi_n(t,\alpha)\leq L_t\Psi_n(t,\alpha)+
         \int f_n\beta'(u^\alpha_t)\Commutator^\alpha (u^\alpha_t,\bb_t) d\mm	
	+ \int \beta(u^\alpha_t) df_n(\bb_t) d\mm.
	$$
Now we let $\alpha\downarrow0$ and use the strong convergence of commutators in $\Lbm{s'}+\Lbm{2}$ and the boundedness of
$\beta'(u^\alpha_t)$ in $L^s\cap\Lbm{2}$ to obtain that $t\mapsto\int_X f_n\beta(u_t)$ is absolutely continuous, and that
$$
         \frac{d}{dt}\int_X f_n\beta(u_t) d\mm\leq L_t\int_X f_n\beta(u_t) d\mm	
	+ \int \beta(u_t) df_n(\bb_t) d\mm.
	$$
By integration, taking into account that $\int_X f_n\beta(u_t) d\mm\equiv 0$ on $(-1,0)$, we get 
$$
\log\biggl(\frac 1\delta\int_Xf_n\beta(u_t) d\mm+1\biggr)\leq \|L\|_1+\int_0^T \int \beta(u_s) df_n(\bb_s) d\mm ds
\quad\text{for all $t\in (-1,T)$ and all $\delta>0$.}
$$
Eventually we use \eqref{eq:mass_preserving} and the monotone convergence theorem to obtain
$$
\log\biggl(\frac 1\delta\int_X\beta(u_t) d\mm+1\biggr)\leq\|L\|_1\qquad\text{for all $t\in (-1,T)$ and $\delta>0$.}
$$
Letting $\delta\downarrow0$ gives $u=0$.

\section{Curvature assumptions and their implications}\label{sec:curvature}

In this section we add to the basic setting \eqref{E-assumptions} a suitable curvature condition, and see the implication of this assumption
on the structural conditions of density of $\Algebra$ in the spaces $\V_p$ and the existence of $f_n\in\Algebra$ in
\eqref{eq:mass_preserving} made in the previous sections.

In the sequel $K$ denotes a generic but fixed real number, and
$\rmI_K$ denotes the real function
  \begin{displaymath}
    \rmI_{K}(t):=\int_0^t \rme^{K r} dr=
    \begin{cases}
      \frac 1K(\rme^{K t}-1)&\text{if }K\neq 0,\\
      t&\text{if }K=0.  
    \end{cases}   
  \end{displaymath}

\begin{definition}[Bakry-\'Emery conditions]\label{def:BE}
We say that $\BE_2(K,\infty)$ holds if
  \begin{equation}
    \label{eq:BE}
    \Gq{\sfP_t f}\le \rme^{-2 K t}\,\sfP_t\big(\Gq f \big)\quad
    \text{$\mm$-a.e.\ in $X$, for every $f\in\V$, $t\ge 0$.}
  \end{equation}
We say that $\BE_1(K,\infty)$ holds if
  \begin{equation}
    \label{eq:BE-improved}
   \sqrt{\Gq{\sfP_t f}}\le \rme^{- K t}\,\sfP_t\big( \sqrt{\Gq f} \big)\quad
    \text{$\mm$-a.e.\ in $X$, for every $f\in\V$, $t\ge 0$.}
  \end{equation}
\end{definition}

We stated both the curvature conditions for the sake of completeness only, but we remark that $\BE_2(K,\infty)$ is sufficient for many of the results we are interested in this section. Obviously, $\BE_1(K,\infty)$ implies $\BE_2(K,\infty)$; the converse, first proved by Bakry in \cite{Bakry-85}, has
been recently extended to a nonsmooth setting by Savar\'e (see \cite[Corollary 3.5]{Savare-13}) under
the assumption that $\cE$ is {\it quasi-regular}. The quasi-regularity property has many equivalent characterizations,
a transparent one is for instance in terms of the existence of a sequence of compact sets $F_k\subset X$ such that
$$
\bigcup_k\left\{f\in\V:\ \text{$f=0$ $\mm$-a.e.\ in $X\setminus F_k$}\right\}
$$
is dense in $\V$.

The validity of the following inequality is actually equivalent to $\BE_2(K,\infty)$, see for instance \cite[Corollary 2.3]{AGS12} for a proof.

\begin{proposition} [Reverse Poincar\'e inequalities]
If $\BE_2(K,\infty)$ holds, then
 \begin{equation}
    \label{eq:77bis}
    2\rmI_{2K}(t) \Gq{\sfP_t f}\le \sfP_t{f^2}-\big(\sfP_t f\big)^2\qquad 
    \mm\text{-a.e.\ in $X$,}
  \end{equation}
 for all $t>0$, $f\in\Lbm 2$.
\end{proposition}

\begin{corollary} [$L^p$-$\Gamma$ inequalities]\label{coro:LpGamma}
If $\BE_2(K,\infty)$ holds, then $L^p$-$\Gamma$ inequalities hold for $p\in [2,\infty]$.
\end{corollary}
\begin{proof}
The validity of $L^p$-$\Gamma$ inequalities for $p\in [2,\infty]$ is obtained integrating \eqref{eq:77bis},
\[ (2\rmI_{2K}(t))^{p/2} \int \Gq{\sfP_t f}^{p/2} d\mm \le \int (\sfP_t{f^2})^{p/2} d\mm \le\int f^p d\mm \]
and using $2\rmI_{2K}(t)^{-1} = O(t^{-1})$ as $t\downarrow0$.
\end{proof}

Another consequence of $\BE_2(K,\infty)$ is the following higher integrability of $\Gamma(f)$, recently proved 
in \cite[Thm.~3.1]{Ambrosio-Mondino-Savare-13} assuming higher integrability of $f$ and $\Delta f$.

\begin{theorem}[Gradient interpolation]
  \label{thm:interpolation}
  Assume that $\BE_2(K,\infty)$ holds and let $\lambda\ge K^-$, $f\in L^2\cap\Lbm\infty$. 
  If $p\in\{2,\infty\}$ and $\Delta f\in\Lbm p$, then $\Gq f\in\Lbm p$ 
  and
  \begin{equation}
    \label{eq:87}
    \big\|\Gq f\|_p\le c\|f\|_\infty\,\|\Delta f + \lambda f\|_p
  \end{equation}
  for a universal constant $c$ (i.e.\ independent of $\lambda$, $K$, $X$, $\mm$).
\end{theorem}

Finally, we will need two more consequences of the $\BE_2(K,\infty)$ condition, proved under the quasi-regularity
assumption in \cite{Savare-13}: the first one, first proved in \cite[Lemma~3.2]{Savare-13} and then slightly improved in
 \cite[Thm.~5.5]{Ambrosio-Mondino-Savare-13}, is the implication
\begin{equation}\label{eq:GammaeV}
f\in\V,\quad\Delta f\in\Lbm 4\quad\Longrightarrow\quad \Gamma(f)\in\V.
\end{equation}
In particular, this implication provides $L^4$ integrability of $\sqrt{\Gamma(f)}$, consistently with the integrability of the
Laplacian. Moreover, it will be particularly useful the quantitative estimate, first proved in \cite[Thm.~3.4]{Savare-13} 
and then slightly improved in \cite[Corollary~5.7]{Ambrosio-Mondino-Savare-13}:
\begin{equation}
     \label{eq:57}
  \Gq{\Gq f}\le 4\gamma_{2,K}[f]\,\Gq f\quad
  \text{$\mm$-a.e.~in $X$, whenever $f\in\V$, $\Delta f\in\Lbm 4$.}
\end{equation}
The function $\gamma_{2,K}[f]$ in \eqref{eq:57} is nonnegative, it satisfies the $L^1$ estimate
  \begin{equation}
  \label{eq:81}
  \int\gamma_{2,K}[f] d\mm
  \le \int_X \bigl((\Delta f)^2-K\Gq f\bigr) d\mm
  \end{equation} 
and it can be represented as the density w.r.t.\ $\mm$ of the nonnegative (and possibly singular w.r.t.\ $\mm$) measure defined by
  \begin{equation}
    \label{eq:61}
    \V\ni\varphi \mapsto 
    \int_X -\frac 12\Gamma(\Gq f,\varphi)+
    \Delta f\,\Gamma(f,\varphi)+\bigr((\Delta f)^2-K\Gq f\bigr)\varphi d\mm.
  \end{equation}
  The nonnegativity of this measure is one of the equivalent formulations of $\BE_2(K,\infty)$, see \cite[\S 3]{Savare-13}
  for a more detailed discussion.

\subsection{Choice of the algebra $\Algebra$}

We first prove that the following ``minimal'' choice for the algebra $\Algebra$ provides \eqref{eq:basic_algebra},  \eqref{eq:stability_composition} and optimal density conditions.

\begin{proposition} \label{prop:BE-density}
Under assumption $\BE_2(K,\infty)$, the algebra
\begin{equation}\label{eq:first_choice_algebra}
\Algebra_1:=\left\{f\in\bigcap_{1\leq p\leq\infty}\Lbm p:\ f\in\V,\,\,\sqrt{\Gamma(f)}\in\bigcap_{1\leq p\leq\infty}\Lbm p
\right\}
\end{equation}
satisfies \eqref{eq:basic_algebra}, \eqref{eq:stability_composition} and it is dense in every space $\V_p$, for $p \in [1,\infty)$.
\end{proposition}
\begin{proof} Since  \eqref{eq:stability_composition} is obviously satisfied by the chain rule, 
we need only to show density of $\Algebra_1$. First, we consider the algebra $\mathcal{A} = \V_2 \cap \V_\infty$, which satisfies the invariance condition \eqref{eq:Feller} because of \eqref{eq:BE}. Moreover, for $p \in [2,\infty)$, the validity of the $L^p$-$\Gamma$ inequality entails that $\mathcal{A}$ is dense in $L^2 \cap L^p$, and taking the $L^{p/2}$ norm in \eqref{eq:BE} gives that \eqref{eq:semigroup-approximating} holds. By Lemma~\ref{lemma:feller-density} (actually, the remark below its proof) we conclude that $\mathcal{A}$ is dense in $\V_p$, for every $p \in [2,\infty)$.

To establish density of $\Algebra_1$ in $\V_p$ for $p \in [1,\infty)$ it is sufficient to notice that the ``refining'' procedure in Lemma~\ref{lem:refining-algebra} applied to $\mathcal{A}$ preserves all the densities in $\V_p$ for $p \in [2,\infty)$, and provides an algebra contained in $\Algebra_1$.
\end{proof}

Retaining the density condition and the algebra property, one can also consider classes smaller than $\Algebra_1$, including
for instance bounds in $\Lbm p$ for the Laplacian.

\subsection{Conservation of mass}

In this section we prove that the curvature condition, together with the conservativity condition $\sfP^\infty_t 1=1$
for all $t>0$ (recall that $\sfP^\infty_t:\Lbm\infty\to\Lbm\infty$ is the dual semigroup in \eqref{eq:dual_P}), imply 
the existence of a sequence $(f_n)\subset\Algebra_1$ as in \eqref{eq:mass_preserving}. Notice that the conservativity
is loosely related to a mass conservation property, for the continuity equation with derivation induced by the logarithmic derivative of the density;
therefore, even though sufficient conditions adapted to the prescribed derivation $\bb$ could be considered as well, it is natural
to consider the conservativity of $\sfP$ in connection with \eqref{eq:mass_preserving}.

\begin{proposition} If $\BE_2(K,\infty)$ holds and $\sfP$ is conservative, then there exist $(f_n)\subset\Algebra_1$ satisfying
\eqref{eq:mass_preserving}.
\end{proposition}
\begin{proof} Let $(g_n) \subset L^1\cap\Lbm\infty$ be a non-decreasing sequence of functions (whose existence
is ensured by the $\sigma$-finiteness assumption on $\mm$) with
$$ \text{$0 \le g_n \le 1$ for every $n\ge 1$ and $\lim_{n \to \infty} g_n =1$, $\mm$-a.e.\ in $X$.}$$
These conditions imply in particular that $g_n\to 1$ weakly$^*$ in $\Lbm\infty$. 

Let $h_n = \int_0^1 \sfP_s g_n ds=\int_0^1\sfP_s^\infty g_n ds$ 
and define $f_n:= \sfP_1 h_n=\sfP^\infty_1 h_n$. By linearity and continuity of $\sfP^\infty$ we obtain that 
$f_n\to\sfP_1^\infty 1=1$ weakly$^*$ in $\Lbm\infty$. In addition, expanding the squares, it is easily seen that
$$
\lim_{n\to\infty} \int (1-f_n)^2 v d\mm=0\qquad\forall v\in\Lbm 1.
$$
Hence, by a diagonal argument we can assume (possibly extracting a subsequence) that $f_n\to 1$ $\mm$-a.e.\ in $X$.

Since $h_n\leq 1$, the reverse Poincar\'e inequality \eqref{eq:77bis} entails
\[ \Gamma(f_n) \le \frac{\sfP_1 h_n^2-(f_n)^2}{2\rmI_{2K}(1)} \le \frac{1-(f_n)^2}{2\rmI_{2K}(1)},\quad \text{$\mm$-a.e.\ in $X$}.\]
Taking the square roots of both sides and using the a.e.\ convergence of $f_n$ we obtain, thanks to dominated
convergence, that $\sqrt{\Gamma(f_n)}$ weakly$^*$ converge to $0$ in $\Lbm\infty$.

Finally, we discuss the regularity of $f_n$. Since
$$
\Delta f_n=\int_1^2\Delta \sfP_s g_n ds=\sfP_2 g_n-\sfP_1 g_n\in\Lbm\infty,
$$
we can use Theorem~\ref{thm:interpolation} to obtain $\sqrt{\Gamma(f_n)}\in\Lbm\infty$. In order to obtain integrability
of the gradient for powers between $1$ and $2$ we can replace $f_n$ by $k_n:= \Phi_1(f_n)/\Phi_1(1)$, with $\Phi_1: \R \to \R$ as introduced in Lemma~\ref{lem:refining-algebra}.
\end{proof}

\subsection{Derivations associated to gradients and their deformation}\label{sec:hessian-be}

In this section, we study more in detail the class of ``gradient'' derivations $\bb_V$ in \eqref{eq:defbbg}. More generally, we 
analyze the regularity of the derivation $f\mapsto \omega\Gamma(f,V)$ associated to sufficiently regular $V$ and $\omega$ in $\V$.

For $p\in (1,\infty]$, let us denote 
\begin{equation}
D_{L^p}(\Delta):=\bigl\{f\in\V\cap\Lbm p:\ \Delta f\in\Lbm p\bigr\}.
\end{equation}
Thanks to the implication \eqref{eq:GammaeV},  $D_{L^4}(\Delta)\subset\V_4$ and
the Hessian
\begin{equation}\label{eq:defHessian} (f,g)\mapsto  H[V](f,g) := 
\frac 1 2 \sqa{ \Gamma(f, \Gamma(V,g)) + \Gamma(g, \Gamma(V,f))- \Gamma(V, \Gamma(f,g))}\in\Lbm 1,
\end{equation}
is well defined on $D_{L^4}(\Delta) \times D_{L^4}(\Delta) $. Notice that the expression is symmetric in $(f,g)$, that $(V, f, g) \mapsto H[V](f,g)$ is multilinear, and that
\[H[V](f,g_1 g_2) =H[V](f,g_1)g_2 + g_1 H[V](f,g_2).\] 
By \cite[Thm.~3.4]{Savare-13}, we have the estimate 
\begin{equation}\label{eq:hessian-estimate}
 \abs{ H[V](f,g)  } \le \sqrt{ \gamma_{2,K}[V]} \sqGq f \sqGq g,\quad \text{$\mm$-a.e.\ in $X$},
\end{equation}
for every $f, \, g \in D_{L^4}(\Delta)$. 

\begin{theorem}\label{gradients_have_deformation}
If $\BE_2(K,\infty)$ holds and $\cE$ is quasi-regular, then for all $V\in D(\Delta)$, $\omega\in\V\cap\Lbm\infty$ 
with $\sqrt{\Gamma(\omega)}\in\Lbm\infty$ and $c \in \R$,
the derivation $\bb =(\omega+c)\bb_V$ has deformation of type $(4,4)$ according to Definition~\ref{def:deformation} with $q=2$,
and it satisfies
\begin{equation}\label{eq:solvay}
\| D^{sym}\bb\|_{4,4}\leq \nor{ \omega + c}_\infty \nor{ (\Delta V)^2-K\Gamma(V) }_1  + \nor{ \sqrt{ \Gamma(\omega)}}_\infty \nor{ \sqrt{ \Gamma(V)} }_2.
\end{equation}
\end{theorem}
\begin{proof} Assume first that $V \in D_{L^4}(\Delta)$. Let $f,\,g\in D_{L^4}(\Delta)$. After integrating by parts the Laplacians of $f$ and $g$, the very definition of $ D^{sym}\bb$ gives 
\begin{equation}
\label{eq:identity-deformation}
 \int  D^{sym}\bb (f,g) d\mm = \int (\omega +c)H[V](f,g)+ \frac 1 2 \sqa{ \Gamma(\omega, f)\Gamma(V, g) + \Gamma(\omega, g)\Gamma(V, f) }d\mm. 
 \end{equation}
By H\"older inequality, we can use \eqref{eq:hessian-estimate} to estimate $\abs{\int  D^{sym}\bb (f,g) d\mm}$ from above with
\[ \biggl[\nor{ \omega}_\infty \nor{ \sqrt{ \gamma_{2,K}[V] }}_2  + \nor{ \sqrt{ \Gamma(\omega)}}_\infty \nor{ \sqrt{ \Gamma(V)} }_2 \biggr] \nor{ \sqGq f}_4 \nor{ \sqGq g}_4.\]
Thus, by definition of $\| D^{sym}\bb\|_{4,4}$, \eqref{eq:solvay} follows, taking also \eqref{eq:81} into account. To pass to the general case $V \in D(\Delta)$, it is sufficient to approximate $V$ with $V_n\in D_{L^4}(\Delta)$ in such a way that $V_n \to V$  in $\V$ and $\Delta V_n \to \Delta V$ in $\Lbm 2$ and notice that $\int  D^{sym}\bb_n(f,g)d\mm$  converge to $\int  D^{sym}\bb(f,g)d\mm$ directly from \eqref{eq:distributional-deformation}. The existence of such an approximating sequence is obtained arguing as in \cite[Lemma 4.2]{Ambrosio-Mondino-Savare-13}, i.e.\ given $f \in D(\Delta)$, we let $h = f- \Delta f \in \Lbm 2$,
\[ h_n:= \max\cur{ \min\cur{h,n}, -n} \in L^2 \cap \Lbm \infty\]
and define $f_n$ as the unique (weak) solution to $f_n - \Delta f_n = h_n$. The maximum principle for $\Delta$ (or equivalently the fact that the resolvent operator $R_1 = (I - \Delta)^{-1}$ is Markov) gives $f_n \in L^2 \cap \Lbm \infty$, thus $\Delta f_n \in L^2 \cap \Lbm \infty$ and by $L^2$-continuity of $R_1$, as $n \to \infty$, both $h_n$ and $f_n$ converge, respectively towards $h$ and $f$. By difference, also $\Delta f_n$ converge towards $\Delta f$ in $\Lbm 2$ and
this gives also easily convergence of $f_n$ to $f$ in $\V$.
\end{proof}

We end this section with a technical result that will be useful when dealing with probability measures on vector spaces, 
in particular in Section~\ref{sec:examples-log-concave}.

\begin{proposition}\label{prop:hilbert-schmidt}
Assume that $\mm(X) = 1$, that $\BE_2(K,\infty)$ holds and that $\cE$ is quasi-regular. Let 
$(V_i)_{i\ge 1} \subset D_{L^4}(\Delta)$ generate an algebra dense in $\V$ and satisfy $\Gamma(V_i, V_j)=\delta_{i,j}$
 $\mm$-a.e.\ in $X$. Then,
\begin{itemize}
\item[(a)] $\Gamma(f) = \sum_{i\ge 1} \Gamma(V_i, f)^2$ $\mm$-a.e.\ in $X$, for every $f \in \V$;
\item[(b)] $H[V_i] = 0$ for every $i \ge 1$.
\end{itemize}
Moreover, for every $q\in [1,\infty]$ and $b= (b^i) \in L^q(X; \ell^2))$ the associated derivation $\bb$ satisfies
\[  f \mapsto df (\bb) = \sum_i b^i \Gamma(V_i, f), \]
satisfies $\abs{\bb}^2 \le \sum_i |b^i|^2$  and therefore belongs to $L^q$. 
In addition, if $r$, $s \in [4,\infty)$, satisfy $q^{-1} + r^{-1} + s^{-1} =1$, $\div \bb \in \Lbm q$ and $b_i\in\V$ for every $i \ge 1$, then
\begin{equation}\label{eq:claimed} \nor{ D^{sym}\bb}_{r,s} \le \frac 1 2 \big\| \big( \sum_{i,j} \abs{  \Gamma(V_j, b_i) + \Gamma(V_i, b_j)} ^2 \big)^{1/2} \|_q.
\end{equation}
\end{proposition}

\begin{proof}
When $f = \psi(V_1, \ldots, V_n)$ belongs to the algebra generated by $(V_i)$, 
the first identity is immediate from $\Gamma(V_i, V_j) = \delta_{i,j}$. The general case of \emph{(a)} follows by density. 

From the definition \eqref{eq:defHessian} of Hessian it holds $H[V_i](V_j, V_k) = 0$ for every $i,\,j,\,k \ge 1$. For fixed $i,\,j \ge 1$, the derivation 
$g \mapsto H[V_i](V_j, g)$ belongs to $\Lbm 2$ in virtue of \eqref{eq:hessian-estimate}, thus it can be extended by density of $\Algebra$ to all of $\V$. 
By the chain rule, the extended derivation is identically zero on the algebra generated by $(V_i)$ thus by density it is the null derivation. 
In particular, for $g\in\Algebra$, $H[V_i](V_j, g) =0$, for every $j \ge 1$. Keeping fixed $g\in\Algebra$, we argue similarly, 
and obtain that $H[V_i](f, g) =0$ $\mm$-a.e.\ in $X$ for every $f,\,g\in\Algebra$, thus proving \emph{(b)}.

If only a finite number of $b^i$'s is different from $0$, and they belong to $\V$, the claimed estimate \eqref{eq:claimed}
follows immediately by linearity, 
\eqref{eq:identity-deformation} and \emph{(b)} above. The general case follows by ``cylindrical'' approximation, where the assumption $r$, $s \ge 4$ plays a role. Indeed, given $f \in \V_r \cap D_{L^r}(\Delta)$ and $g \in \V_s \cap D_{L^s}(\Delta)$ it holds $f$, $g \in D_{L^4} (\Delta)$, thus $\Gamma(f,g) \in \V$ and we can integrate by parts the last term in \eqref{eq:distributional-deformation}, obtaining
\begin{equation}\label{eq:def-ibp}\int  D^{sym}\bb(f,g) d\mm = -\frac{1}{2} \int df(\bb)\Delta g + dg(\bb) \Delta f + d(\Gamma(f,g))(\bb)  d\mm.\end{equation}
Let $N \ge 1$ and let $\bb_N$ be the derivation associated to the sequence $(b^1, \ldots, b^N, 0, 0, \ldots)$. Given $h \in \V$, it holds
\[\abs{d(h) \bb_N - d(h) \bb } \le \Gamma(h)^{1/2} \big( \sum_{i > N} \abs{b^i}^2\big)^{1/2}, \quad \text{ $\mm$-a.e.\ in $X$.}\]
By this estimate with $h =f$, $h= g$ and $h = \Gamma(f,g)$,  H\"older inequality and dominated convergence we conclude  that the sequence $\int  D^{sym}\bb_N(f,g) d\mm$ converges towards $\int  D^{sym}\bb(f,g) d\mm$ as $N \to \infty$, entailing \eqref{eq:claimed}. \end{proof}

Notice that the assumption $r$, $s \in [4,\infty)$ is used only to obtain  $\Gamma(f,g) \in \V$ and so \eqref{eq:def-ibp}. The same argument indeed shows that, for $r$, $s \in [1,\infty)$ and $q \in (1,\infty]$ with $q^{-1} +r^{-1} +s^{-1} = 1$, if $\Algebra$ is dense in the space $\V_p \cap D_{L^p}(\Delta)$, endowed with the norm $\nor{f} = \nor{f}_{\V_p} + \nor{\Delta f}_{L^2 \cap L^p}$, for $p \in \cur{r,s}$ and it satisfies $\Gamma(f, g) \in \Algebra$ for $f$, $g \in \Algebra$, then the last statement in Proposition~\ref{prop:hilbert-schmidt} holds, regardless of the condition $r$, $s \in [4,\infty)$.

\section{The superposition principle in $\R^\infty$ and in metric measure spaces}\label {sec:superpo}

In this section we denote $\R^\infty=\R^{\N}$ endowed with the product topology and
we shall denote by $\pi^n:=(p_1,\ldots,p_n):\R^\infty\to\R^n$ 
the canonical projections from $\R^\infty$ to $\R^n$. On the space $\R^\infty$ 
we consider the complete and separable distance
$$
d_\infty(x,y):=\sum_{n=1}^\infty 2^{-n}\min \cur{1,|p_n(x)-p_n(y)|}.
$$
Accordingly, we consider the space $C([0,T];\R^\infty)$ endowed with the distance
$$
\delta(\eta,\tilde \eta):=\sum_{n=1}^\infty 2^{-n}\max_{t \in [0,T]}\min\cur{1,|p_n(\eta(t))-p_n(\tilde \eta(t))|},
$$
which makes
$C([0,T];\R^\infty)$ complete and separable as well.
We shall also consider the subspace $AC_w([0,T];\R^\infty)$ of $C([0,T];\R^\infty)$ consisting of all $\eta$ such that
$p_i\circ\eta\in AC([0,T])$ for all $i\geq 1$. Notice that for this class of curves the derivative $\eta'\in\R^\infty$ can still be
defined a.e.\ in $(0,T)$, arguing componentwise. We use the notation $AC_w$ to avoid the confusion with the space
of absolutely continuous maps from $[0,T]$ to $(\R^\infty,d_\infty)$.

It is immediate to check that for any choice of convex superlinear and l.s.c.\ functions $\Psi_n:[0,\infty)\to [0,\infty]$ and for l.s.c.\ functions
$\Phi_n:[0,\infty)\to [0,\infty]$ with $\Phi_n(v)\to\infty$ as $v\to\infty$ the functional
${\cal A}:C([0,T];\R^\infty)\to [0,\infty]$ defined by
$$
{\cal A}(\eta):=\begin{cases}
\sum\limits_{n=1}^\infty \bigl[\Phi_n(p_n\circ\eta(0))+\int_0^T\Psi_n(|(p_n\circ\eta)'|) dt\bigr]
&\text{if $\eta\in AC_w([0,T];\R^\infty)$}\\ \\ \infty &\text{if $\eta\in C([0,T];\R^\infty)\setminus AC_w([0,T];\R^\infty)$}
\end{cases}
$$
is coercive in $C([0,T];\R^\infty)$, i.e.\ all sublevels $\{{\cal A}\leq M\}$ are compact in $C([0,T];\R^\infty)$.

We call smooth cylindrical function any $f:\R^\infty\to\R$ representable in the form 
$$
f(x)=\psi(\pi_n(x))=\psi\bigl(p_1(x),\ldots,p_n(x)\bigr)\qquad x\in\R^\infty,
$$
with $\psi:\R^n\to\R$ bounded and continuously differentiable, with bounded derivative. When we want to
emphasize $n$, we say that $f$ is $n$-cylindrical. Given $\psi$ smooth cylindrical,
we define $\nabla f:\R^\infty\to c_0$ (where $c_0$ is the space of sequences $(x_n)$ null for $n$ large enough)
by
\begin{equation}\label{def:cylindrical}
\nabla f(x):=\bigl(\frac{\partial\psi}{\partial z_1}(\pi_n(x)),\ldots,\frac{\partial\psi}{\partial z_n}(\pi_n(x)),0,0,\ldots).
\end{equation}

We fix a 
Borel vector field $\cc:(0,T)\times\R^\infty\to\R^\infty$ and a weakly continuous (in duality with smooth cylindrical functions)
family of Borel probability measures $\{\nu_t\}_{t\in (0,T)}$ in $\R^\infty$ satisfying 
\begin{equation}\label{eq:conti1}
\int_0^T\int |p_i(\cc_t)| d\nu_tdt <\infty,\qquad\forall i\geq 1
\end{equation}
and, in the sense of distributions,
\begin{equation}\label{eq:conti2}
\frac{d}{dt}\int fd\nu_t=\int\bra{\cc_t,\nabla f} d\nu_t\qquad\text{in $(0,T)$, for all $f$ smooth cylindrical.}
\end{equation}

\begin{theorem}[Superposition principle in $\R^\infty$]\label{thm:superpoRinfty}
Under assumptions \eqref{eq:conti1} and \eqref{eq:conti2}, there exists a Borel probability measure
$\llambda$ in $C([0,T];\R^\infty)$ satisfying $(e_t)_\#\llambda=\nu_t$ for all $t\in (0,T)$, concentrated on $\gamma\in AC_w([0,T];\R^\infty)$
which are solutions
to the ODE $\dot\gamma=\cc_t(\gamma)$ a.e.\ in $(0,T)$. 
\end{theorem}
\begin{proof} The statement is known in finite-dimensional spaces, see e.g.\ \cite[Thm.~8.2.1]{Ambrosio-Gigli-Savare05}
for the case when $\int\int |\cc_t|^r d\nu_t dt<\infty$ for some $r>1$, and \cite[Thm.~12]{bologna} for the case $r=1$.
For $i\geq 1$ we choose convex, superlinear, l.s.c.\ functions $\Psi_i:[0,\infty)\to [0,\infty]$ with
\begin{equation}\label{eq:condpsi1}
\int_0^T\int\Psi_i(|p_i(\cc_t)|) d\nu_t dt\leq 2^{-i}
\end{equation}
and coercive $\Phi_i:[0,\infty)\to [0,\infty)$ satisfying 
\begin{equation}\label{eq:condpsi2}
\int \Phi_i(p_i(x))d\nu_0(x)\leq 2^{-i}
\end{equation}
and define ${\cal A}$ accordingly.
 
Defining $\nu^n_t:=(\pi^n)_\#\nu_t$ and $\cc^n_{t,i}$, $1\leq i\leq n$, as the density of $(\pi^n)_\#(p_i(\cc_t)\nu_t)$ w.r.t.
to $\nu^n_t$, it is immediate to check with Jensen's inequality that 
\begin{equation}\label{eq:conti3}
\int_0^T\int \Psi_i(|\cc^n_{t,i}|)d\nu^n_tdt \leq\int_0^T\int \Psi_i(|p_i(\cc_t)|) d\nu_tdt,\qquad i\geq 1,
\end{equation}
and that $\nu^n_t$ solve the continuity equation in $\R^n$ relative to the vector field $\cc^n=(\cc^n_i,\ldots,\cc^n_n)$. 
Therefore the finite-dimensional
statement provides $\llambda_n$, probability measures in $C([0,T];\R^n)$, concentrated on absolutely continuous a.e.\ solutions to the
ODE $\dot\gamma=\cc^n_t(\gamma)$ and satisfying $(e_t)_\#\llambda_n=\nu^n_t$ for all $t\in [0,T]$. 

In order to pass to the limit as $n\to\infty$ it is convenient to view $\llambda_n$ as probability measures in $C([0,T];\R^\infty)$
concentrated on curves $\gamma$ such that $p_i(\gamma)$ is null for $i>n$ and $\nu^n$ as probability measures in $\R^\infty$
concentrated on $\{x\in\R^\infty:\ p_i(x)=0\,\,\forall i>n\}\subset c_0$. Accordingly, if we set $\cc^n_{t,i}\equiv 0$ for $i>n$, we retain
the property that $\llambda_n$ is concentrated on absolutely continuous solutions to the
ODE $\dot\gamma=\cc^n_t(\gamma)$ and satisfies $(e_t)_\#\llambda_n=\nu^n_t$ for all $t\in [0,T]$. 

Using \eqref{eq:conti3} and our choice of $\Psi_i$ and $\Phi_i$ we immediately obtain
$$
\int {\cal A}(\gamma) d\llambda_n(\gamma) \leq 2,
$$
hence the sequence $(\llambda_n)$ is tight in $\Probabilities{C([0,T];\R^\infty)}$. 

We claim that any limit point $\llambda$ fulfills the properties stated in the lemma. Just for notational simplicity, we assume in the
sequel that the whole family $(\llambda_n)$ weakly converges to $\llambda$. The lower semicontinuity of ${\cal A}$ gives
$\int {\cal A}\,d\llambda<\infty$, hence $\llambda$ is concentrated on $AC_w([0,T];\R^\infty)$. Furthermore, since
$$
\gamma\mapsto \pi_k\circ\gamma(t),\qquad t\in [0,T]
$$
are continuous from $C([0,T];\R^\infty)$ to $\R^k$, passing to the limit as $n\to\infty$ in the identity
$(\pi_k)_\sharp (e_t)_\sharp\llambda_n=(\pi_k)_\sharp\nu^n_t$
it follows that $(\pi_k)_\sharp (e_t)_\sharp\llambda=(\pi_k)_\sharp\nu_t$
for all $k$. We can now use the fact that cylindrical functions generate the Borel $\sigma$-algebra of $\R^\infty$
to obtain that $(e_t)_\sharp\llambda=\nu_t$. 

It remains to prove that 
$\llambda$ is concentrated on solutions to the ODE $\dot\gamma=\cc_t(\gamma)$. To this aim, it suffices to show that
\begin{equation}\label{valleggi}
\int\left|p_i\circ\gamma(t)-p_i\circ\gamma(0)-\int_0^tp_i\circ\cc_s(\gamma(s))\,ds\right|\,d\llambda(\gamma)=0,
\end{equation}
for any $t\in [0,T]$ and $i\geq 1$. The technical difficulty is that this test function,
due to the lack of regularity of $\cc$, is not continuous in $C([0,T];\R^\infty)$. 
To this aim, we prove first that
\begin{equation}\label{valleggi1}
\int\left|p_i\circ\gamma(t)-p_i\circ\gamma(0)-\int_0^t\dd_s(\gamma(s))\,ds\right|\,d\llambda(\gamma)\leq
\int_{(0,T)\times\R^\infty}|p_i\circ\cc-\dd|\,d\nu_tdt,
\end{equation}
for any bounded Borel function $\dd$ with $\dd(t,\cdot)$ $k$-cylindrical for all $t\in (0,T)$, with
$k$ independent of $t$. It is clear that the space
$$
\bigl\{\dd\in L^1(\nu_t dt):\ \text{$\dd(t,\cdot)$ cylindrical for all $t\in (0,T)$}\bigr\}
$$
is dense in $L^1(\nu_t dt)$; by a further approximation, also the space 
$$
\bigcup_{k=1}^\infty\bigl\{\dd\in L^1(\nu_t dt):\ \text{$\dd(t,\cdot)$ $k$-cylindrical for all $t\in (0,T)$}\bigr\}
$$
is dense. Hence, choosing a sequence $(\dd^m)$ of functions admissible for \eqref{valleggi1}
converging to $p_i\circ\cc$ in $L^1(\nu_tdt)$ and noticing that
$$
\int_{(0,T)\times\R^\infty}|p_i\circ\cc_s(\gamma(s))-\dd_s^m(\gamma(s))|\,dsd\llambda(\gamma)
=\int_{(0,T)\times\R^\infty}|p_i\circ\cc-\dd^m|\,d\nu_tdt\rightarrow 0,
$$
we can take the limit in \eqref{valleggi1} with $\dd=\dd^m$ to obtain \eqref{valleggi}.

It remains to show \eqref{valleggi1}. We first prove
\begin{equation}\label{eqn:convergencec}
\limsup_{n\to\infty}\int_{(0,T)\times\R^\infty}|p_i\circ\cc^n-\dd|\,d\nu^n_s\,ds\leq
\int_{(0,T)\times\R^\infty}|p_i\circ\cc-\dd|\,d\nu_tdt
\end{equation}
for all bounded Borel functions $\dd$ with $\dd(t,\cdot)$ $k$-cylindrical for all $t\in (0,T)$, with
$k$ independent of $t$. The proof is elementary, because for $n\geq k$ and $t\in (0,T)$ we have
$$
(p_i\circ\cc^n_t-\dd_t)\nu^n_t=(\pi_n)_\#((p_i\circ\cc_t-\dd_t)\nu_t).
$$ 

Now we can prove \eqref{valleggi1}, with a limiting argument based on the fact
that \eqref{valleggi} holds for $\cc^n$, $\llambda_n$:
\begin{eqnarray*}
&&\int\left|p_i\circ\gamma(t)-p_i\circ\gamma(0)-\int_0^t\dd_s(\gamma(s))\,ds\right|\,d\llambda_n(\gamma)\\&=&
\int\left|\int_0^t\bigl(p_i\circ\cc^n_s(\gamma(s))-\dd_s(\gamma(s))\bigr)\,ds\right|\,d\llambda_n(\gamma)\\
&\leq&
\int\int_0^t|p_i\circ\cc^n_s-\dd_s|(\gamma(s))\,ds d\llambda_n(\gamma)\leq\int_{(0,T)\times\R^\infty}|p_i\circ\cc^n-\dd|\,d\nu^n_s ds.
\end{eqnarray*}
Since $\dd_s$ is cylindrical for all $s$ and uniformly bounded w.r.t.\ $s$, the map
$$
\gamma\mapsto \left|p_i\circ\gamma(t)-p_i\circ\gamma(0)-\int_0^t\dd_s(\gamma(s))\,ds\right|
$$
belongs to $C\bigl(C([0,T];\R^\infty)\bigr)$ and is nonnegative. Hence,
taking the limit in the chain of inequalities above and using \eqref{eqn:convergencec} we obtain \eqref{valleggi1}.
\end{proof}

We next consider the case of a (possibly extended) metric measure space $(X,\tau,\mm,d)$. Starting from the basic setup of Section~\ref{sec:setup},
we have indeed only a topology $\tau$ and the measure $\mm$. We assume the existence of a countable set $\Algebra^*\subset\{f\in\Algebra:\
\|\Gamma(f)\|_\infty\leq1\}$ satisfying:
\begin{equation}\label{eq:propertiesGstar1}
\text{$\R\Algebra^*$ is dense in $\V$ and any function in $\Algebra^*$ has a $\tau$-continuous representative,} 
\end{equation}
\begin{equation}\label{eq:propertiesGstar2}
\text{$\exists\lim_{n\to\infty}f(x_n)$ in $\R$ for all $f\in\Algebra^*$}\quad\Longrightarrow\quad\text{$\exists\lim_{n\to\infty} x_n$ in $X$.} 
\end{equation}
Since $\supp\mm=X$, the $\tau$-continuous representative of a $\mm$-measurable function if exists is unique, and
for this reason we do not use in \eqref{def:d_cE} and in the sequel a distinguished notation for the continuous representative of functions in $\Algebra^*$.
Notice that \eqref{eq:propertiesGstar2} implies that the family $\Algebra^*$ separates the points of $X$. Properties \eqref{eq:propertiesGstar1}
and \eqref{eq:propertiesGstar2} can be fulfilled in explicit cases;  moreover, since the role played by $\tau$ is marginal, one can ``artificially'' satisfy \eqref{eq:propertiesGstar2} by considering, in place of $\tau$, the coarsest topology that makes continuous all the functions in $\Algebra^*$. A posteriori, results about the given topology $\tau$ can be recovered if the topology induced by the distance $d_{\Algebra^*}$, as introduced in the following remark, is stronger than (or equal to) $\tau$.


\begin{remark}[Extended distance induced by $\Algebra^*$]\label{rem:hiddendistance}{\rm
Following \cite{Biroli-Mosco95} (see also \cite{Sturm95,Stollmann10}) we build $d_{\Algebra^*}:X\times X\to [0,\infty]$ as
\begin{equation}\label{def:d_cE}
d_{\Algebra^*}(x,y)=\sup\left\{|f(x)-f(y)|:\ f\in\Algebra^*\right\},\qquad x,\,y\in X.
\end{equation}
A priori, $d_{\Algebra^*}$ is an extended distance in the sense of \cite{AGS11a}, since it may take the value $\infty$; nevertheless, by definition, 
all functions in $\Algebra^*$ are $1$-Lipschitz w.r.t $d_{\Algebra^*}$ and $d_{\Algebra^*}$ is the smallest extended distance with this property. In particular
the derivative $\frac{d}{dt}(f\circ\eta)$ which occurs in the next definition makes sense a.e.\ in $(0,T)$ when $f\in\Algebra^*$ and
$\eta\in AC([0,T];(X,d_{\Algebra^*}))$, because $f\circ\eta$ belongs to $AC([0,T])$. However, we will not use the topology induced by
$d_{\Algebra^*}$, which could be much finer than the topology $\tau$ and, in the next definition, we will require only
continuity of $\eta:[0,T]\to X$ (with the topology $\tau$ in the target space $X$) and $W^{1,1}(0,T)$ regularity of $f\circ\eta$, for $f\in\Algebra$. 
A posteriori, in Lemma~\ref{eq:fromODEtoPDE} we are going to recover some absolute continuity for $\eta$, with respect to $d_{\Algebra^*}$.
In any case, whenever $f\in\Algebra$ has a continuous representative (as it happens when $f\in\Algebra^*$), the continuity of $f\circ\eta$ in
conjunction with Sobolev regularity gives $f\circ\eta\in AC([0,T])$.\fr}
\end{remark}

\begin{definition}[ODE induced by a family $(\bb_t)$ of derivations]\label{def:ourODE}
Let $\eeta\in\Probabilities{C([0,T];X)}$ and let $(\bb_t)_{t\in (0,T)}$ be a Borel family of derivations. We say that
$\eeta$ is concentrated on solutions to the ODE $\dot\eta=\bb_t(\eta)$ if 
 $$
 \text{$f\circ\eta\in W^{1,1}(0,T)$ and $\frac {d}{dt} (f\circ\eta)(t)=df(\bb_t)(\eta(t))$, 
for a.e.\ $t\in (0,T)$,}
$$
for $\eeta$-a.e.\ $\eta \in C([0,T];X)$, for all $f\in\Algebra$.
\end{definition}

Notice that the property of being concentrated on solutions to the ODE implicitly depends on the choice of Borel representatives of
the maps $f$ and $(t,x)\mapsto df(\bb_t)(x)$, $f\in\Algebra$. As such, it should be handled with care. We will see, however, that in 
the class of regular flows of Definition~\ref{def:dregflow} this sensitivity to the choice of Borel representatives disappears,
see Remark~\ref{rem:sensitivity}.

The following simple lemma shows that time marginals of measures $\eeta$ concentrated on solutions
to the ODE $\dot\eta=\bb_t(\eta)$ provide weakly continuous solutions to the continuity equation.

Given a derivation $\bb$, we introduce the following quantity
\begin{equation}
\label{eq:modulus-*} |\bb|_{*} = \sup \cur{ |d f(\bb)|: f \in \Algebra^*}.
\end{equation}
Notice that $|\bb|_*$ is well-defined up to $\mm$-a.e.\ equivalence and that one has
$|\bb|_* \le |\bb|$, $\mm$-a.e.\ in $X$. Also in view of \eqref{eq:inequality-or-equality} below,
it is natural to investigate the validity of the equality $|\bb| = |\bb|_*$, $\mm$-a.e.\ in $X$. We are able to prove this in the 
setting of $\RCD$ spaces, see Lemma~\ref{lem:algebra-*-RCD} below.

\begin{lemma}\label{eq:fromODEtoPDE}
Let $\eeta\in\Probabilities{C([0,T];X)}$ be concentrated on solutions $\eta$ to the ODE $\dot\eta=\bb_t(\eta)$, where $|\bb|\in L^1_t(L^p_x)$ for some $p\in [1,\infty]$ and $\mu_t:=(e_t)_\#\eeta\in\Probabilities{X}$ are representable as 
$u_t\mm$ with $u \in L^\infty_t(L^{p'}_x)$. Then, the following two properties hold:
\begin{itemize}
\item[(a)] the family $(u_t)_{t\in (0,T)}$ is a weakly continuous solution to the continuity equation;
\item[(b)] $\eeta$ is concentrated on $AC([0,T]; (X, d_{\Algebra^*}) )$, with
\begin{equation} \label{eq:inequality-or-equality}\abs{\dot \eta} (t) = \abs{\bb_t}_*(\eta(t))\quad \text{ for a.e.\ $t \in (0,T)$, for $\eeta$-a.e.\ $\eta$.} \end{equation}
\end{itemize}
\end{lemma}

\begin{remark}{\rm 
Arguing as in the last part of \cite[Thm.~8.3.1]{Ambrosio-Gigli-Savare05} one can prove that $u \in L^\infty_t(L^\infty_x)$ implies that
$(\mu_t)_t$ is an absolutely continuous curve in the Wasserstein space $W_p$ naturally associated to $d_{\Algebra^*}$ 
(see \cite{GB-14} for a more systematic investigation of this connection in metric measure spaces).} \fr
\end{remark}
\begin{proof} We integrate w.r.t.\ $\eeta$ the weak formulation
$$
\int_0^t -\psi'(t)f\circ\eta(t) dt=\int_0^T\psi(t)df(\bb_t)(\eta(t)) dt
$$
with $f\in\Algebra$, $\psi\in C^1_c(0,T)$, to recover the weak formulation of the continuity equation for $(u_t)$.

Given $f \in \Algebra^*$, for $\eeta$-a.e.\ $\eta$, the map $t \mapsto f \circ \eta(t)$ is absolutely continuous, with
\[ f\circ\eta(t) - f\circ\eta(s)  = \int_s^t  df( \bb_r) (\eta(r)) dr, \quad\text{for all $s,\, t \in [0,T]$.}\]
In particular one has $df(\bb_t)(\eta(t)=(f\circ\eta)'(t)$ a.e. in $(0,T)$, for $\eeta$-a.e. $\eta$.

By Fubini's theorem and the fact that the marginals of $\eeta$
are absolutely continuous w.r.t.\ $\mm$ we obtain that, for $\eeta$-a.e.\ $\eta$, one has
\[ \sup_{f \in \Algebra^*} \abs{(f \circ \eta)'(t)} =  \sup_{f \in \Algebra^*} \abs{df(\bb_t)(\eta(t))} = \abs{\bb_t}_*(\eta(t)), \quad \text{for a.e.\ $t\in (0,T)$,}\]
and therefore
\[ d_{\Algebra^*} (\eta(t), \eta(s) ) = \sup_{f \in \Algebra^*}\bigl|(f\circ\eta)(t) - (f\circ\eta)(s)\bigr|
\le \int_s^t \abs{\bb_t}_*(\eta(r))dr, \quad \text{for all $s,\,t \in [0,T]$,} \]
proving that $\eta \in AC([0,T]; (X, d_{\Algebra^*}))$, with $\abs{\dot \eta}(t) \le \abs{\bb_t}_*(\eta(t))$, for a.e.\ $t\in (0,T)$. 
The converse inequality follows from the fact that every $f \in \Algebra^*$ is $1$-Lipschitz with respect to $d_{\Algebra^*}$, 
thus for $\eeta$-a.e. $\eta$ one has
\[ \abs{\bb_t}_*(\eta(t)) =  \sup_{f \in \Algebra^*} \abs{(f \circ \eta)'(t)} \le \abs{\dot \eta}(t), \quad \text{for a.e.\ $t\in (0,T)$.}\]
\end{proof}

Even though, as we explained in Remark~\ref{rem:hiddendistance}, the (extended) distance is hidden in the choice
of the family $\Algebra^*$, we call the next result ``superposition in metric measure spaces'', because in most cases
$\Algebra^*$ consists precisely of distance functions from a countable dense set  (see also the recent papers \cite{Bate-12}
and \cite{Schioppa-13} for related results on the existence of suitable measures in the space of curves, and derivations).

\begin{theorem}[Superposition principle in metric measure spaces]\label{thm:superpo}
Assume \eqref{eq:propertiesGstar1}, \eqref{eq:propertiesGstar2}.
Let $\bb=(\bb_t)_{t\in (0,T)}$ be a Borel family of derivations and 
let $\mu_t=u_t\mm\in\Probabilities{X}$, $0\leq t\leq T$, be a weakly continuous solution to the continuity equation
\begin{equation}
	\label{eq:continuity-equation-homo}
	\partial_t\mu_t + \div\bra{\bb_t\mu_t} =0
	\end{equation} 
with
\begin{equation}\label{eq:superpo2bound}
u\in L^\infty_t(L^p_x),\qquad
\int_0^T\int|\bb_t|^rd\mu_tdt<\infty,\qquad \frac{1}{r}+\frac{1}{p}\leq 1/2.
\end{equation}
Then there exists $\eeta\in\Probabilities{C([0,T];X)}$ satisfying:
\begin{itemize}
\item[(a)] $\eeta$ is concentrated on solutions $\eta$ to the ODE $\dot\eta=\bb_t(\eta)$, according to Definition~\ref{def:ourODE};
\item[(b)] $\mu_t=(e_t)_\#\eeta$ for any $t\in [0,T]$.
\end{itemize}
\end{theorem}
\begin{proof} We enumerate by $f_i$, $i\geq 1$, the elements of $\Algebra^*$ and 
define a continuous and injective map $J:X\to\R^\infty$ by
\begin{equation}\label{eq:defJ}
J(x):=\bigl(f_1(x), f_2(x), f_3(x),\ldots\bigr).
\end{equation}
A simple consequence of \eqref{eq:propertiesGstar2},
besides the injectivity we already observed, is that $J(X)$ is a closed subset of $\R^\infty$ and that $J^{-1}$ is continuous from $J(X)$ to $X$. 

Defining $\nu_t\in\Probabilities{\R^\infty}$ by $\nu_t:=J_\#\mu_t$, $\cc:(0,T)\times\R^\infty\to\R^\infty$ by
$$
 \cc^i_t:=\begin{cases}
(d f_i(\bb_t))\circ J^{-1}&\text{on $J(X)$;}\\ \\ 0 &\text{otherwise,}
\end{cases}
$$
and noticing that
\begin{equation}\label{eq:boundcci}
|\cc^i_t|\circ J\leq |\bb_t|,\quad\text{$\mm$-a.e.\ in $X$,}
\end{equation}
the chain rule (see Proposition~\ref{prop-chain-rule-derivations})
$$
d\phi(\bb_t)(x)=\sum_{i=1}^n\frac{\partial\psi}{\partial z_i}( f_1(x),\ldots, f_n(x))\cc^i_t(x)
$$
for $\phi(x)=\psi(f_1(x),\ldots, f_n(x))$ 
shows that the assumption of Theorem~\ref{thm:superpoRinfty} are satisfied by $\nu_t$ with velocity $\cc$, because
\eqref{eq:boundcci} and $\mu_t\ll\mm$ give $|\cc^i_t|\leq |\bb_t|\circ J^{-1}$ $\nu_t$-a.e.\ in $\R^\infty$.

As a consequence we can apply Theorem~\ref{thm:superpoRinfty} to obtain $\llambda\in\Probabilities{C([0,T];\R^\infty)}$
concentrated on solutions $\gamma\in AC([0,T];\R^\infty)$ to the ODE $\dot\gamma=\cc_t(\gamma)$ such that $(e_t)_\#\llambda=\nu_t$ for all $t\in [0,T]$. 
Since all measures $\nu_t$ are concentrated on $J(X)$,
$$
\text{$\gamma(t)\in J(X)$ for $\llambda$-a.e.\ $\gamma$, for all $t\in [0,T]\cap\Q$.}
$$
Then, the closedness of $J(X)$ and the continuity of $\gamma$ give $\gamma([0,T])\subset J(X)$ for $\llambda$-a.e.\ $\gamma$. For this reason, it makes
sense to define 
$$
\eeta:=\Theta_\#\llambda
$$
where $\Theta:C([0,T];J(X))\to C([0,T];X)$ is the map $\gamma\mapsto \Theta(\gamma):=J^{-1}\circ\gamma$. Since
$(J^{-1})_\#\nu_t=\mu_t$, we obtain immediately that $(e_t)_\#\eeta=\mu_t$.

Let $i\geq 1$ be fixed. Since $f_i\circ\Theta(\gamma)=p_i\circ\gamma$, taking the definition of $\cc_i$ into account we obtain that 
$f_i\circ\eta$ is absolutely continuous in $[0,T]$ and that
\begin{equation}\label{eq:partial_derivatives}
\text{$(f_i\circ\eta)'(t)=df_i(\bb_t)(\eta(t))$ a.e.\ in $(0,T)$, for $\eeta$-a.e.\ $\eta$.}
\end{equation}
We will complete the proof by showing that \eqref{eq:partial_derivatives} extends from $\Algebra^*$ to all of $\Algebra$.
By the chain rule we observe, first of all, that \eqref{eq:partial_derivatives} extends from $f_i$ to smooth truncations
of $f_i$. Therefore, by the density of $\Algebra^*$ in $\V$, for any $f\in\Algebra$ we can find $g_n$ satisfying:
\begin{itemize}
\item[(a)] $g_n\to f$ in $\V$ and $\|g_n\|_\infty\leq\|f\|_\infty+1$;
\item[(b)] $g_n\circ\eta\in AC([0,T])$ and $(g_n\circ\eta)'(t)=dg_n(\bb_t)(\eta(t))$ a.e.\ in $(0,T)$, for $\eeta$-a.e.\ $\eta$.
\end{itemize}
Since
\begin{equation}\label{eq:convl1}
\int\int_0^T|(f-g_n)(\eta(t))|dt d\eeta(\eta)=\int_0^T\int |f-g_n|u_t d\mm dt\rightarrow 0
\end{equation}
we can assume, possibly refining the sequence $(g_n)$, that $g_n\circ\eta\to f\circ\eta$ in $L^1(0,T)$ for $\eeta$-a.e.\ $\eta$. 

In order to achieve Sobolev regularity of $f\circ\eta$ it remains to show convergence of the derivatives of $g_n\circ\eta$, namely
$dg_n(\bb_t)(\eta(t))$, to $df(\bb_t)(\eta(t))$. Arguing
as in \eqref{eq:convl1} we get
$$
\int\int_0^T|df(\bb_t)(\eta(t))-dg_n(\bb_t)(\eta(t))|dt d\eeta(\eta)=
\int_0^T\int |d(f-g_n)(\bb_t)| u_td\mm dt\rightarrow 0
$$
because of \eqref{eq:superpo2bound} and the convergence $\Gq{f-g_n}\to 0$ in $\Lbm 1$. Therefore, possibly refining once
more $(g_n)$, $dg_n(\bb)(\eta)\to df(\bb)(\eta)$ in $L^1(0,T)$ for $\eeta$-a.e.\ $\eta$.
\end{proof}

\section{Regular Lagrangian flows} \label{sec:RLF}

In this section we consider a Borel family of derivations $\bb=(\bb_t)_{t\in (0,T)}$ satisfying
\begin{equation}\label{eq:basicbb}
\bb\in L^1_t(L^1_x+L^\infty_x).
\end{equation}
Under the assumption that the continuity equation has uniqueness of solutions in the class
\begin{equation}\label{mathcalL}
{\mathcal L}_+:=\bigl\{u\in L^\infty_t(L^1_x \cap L^\infty_x):\ \text{$t\mapsto u_t$ is weakly continuous in $[0,T]$, $u\geq 0$}\bigr\}
\end{equation}
for any initial datum $\bar u\in L^1\cap\Lbm\infty$, and existence of solutions in the class 
\begin{equation}\label{eq:morenarrow}
\bigl\{u\in{\mathcal L}_+:\ \|u_t\|_{\infty}\leq  C(\bb)\|u_0\|_\infty\,\,\,\forall t\in [0,T]\bigr\},
\end{equation}
for any nonnegative initial datum $\bar u\in L^1\cap\Lbm\infty$,
we prove existence and uniqueness of the regular flow $\XX$ associated to $\bb$. 
Here, the need for a class as large as possible where uniqueness holds is hidden in the proof of Theorem~\ref{thm:nosplitting}, where
solutions are built by taking the time marginals of suitable probability measures on curves and uniqueness leads to a non-branching
result.
The concept of regular flow, adapted from \cite{Ambrosio03}, is the following:

\begin{definition} [Regular flows]\label{def:dregflow} We say that $\XX: [0,T]\times X\to X$ is a regular flow (relative
to $\bb$) if the following two properties hold: 
\begin{itemize}
\item[(i)] $\XX(0,x)=x$ and $\XX(\cdot,x)\in C([0,T];X)$ for all $x\in X$;
\item[(ii)] for all $f\in\Algebra$, $f(\XX(\cdot,x))\in W^{1,1}(0,T)$ and $\frac{d}{dt} f(\XX(t,x))=df(\bb_t)(\XX(t,x))$ 
for a.e.\ $t\in (0,T)$, for $\mm$-a.e.\ $x\in X$;
\item[(iii)] there exists a constant $C=C(\XX)$ satisfying $\XX(t,\cdot)_\#\mm\leq C\mm$ for all $t\in [0,T]$.
\end{itemize}
\end{definition}

\begin{remark}[Invariance under modifications of $\bb$ and $f$] \label{rem:sensitivity} {\rm Assume that $\bb$ and $\tilde\bb$ satisfy
\begin{equation}\label{eq:equibb}
\text{for all $f\in\Algebra$, $df(\bb)=df(\tilde\bb)$ $\Leb{1}\otimes\mm$-a.e.\ in $(0,T)\times X$.}
\end{equation}
Then $\XX$ is a regular flow relative to $\bb$ if and only if $\XX$ is a regular flow relative to $\tilde\bb$. Indeed,
let us fix $f\in\Algebra$ and let us notice that for all $t\in (0,T)$ such that
$\mm(\{df(\bb_t)\neq df(\tilde\bb_t)\})=0$, condition (iii) of Definition~\ref{def:dregflow}
gives
$$
df(\bb_t)(\XX(t,x))=df(\tilde\bb_t)(\XX(t,x)),\qquad\text{for $\mm$-a.e.\ $x\in X$.}
$$
Thanks to \eqref{eq:equibb} and Fubini's theorem, the condition $\mm(\{df(\bb_t)\neq df(\tilde\bb_t)\})=0$ is satisfied for a.e.\ $t\in (0,T)$. 
Hence, we may apply once more Fubini's theorem to get
$$
df(\bb_t)(\XX(t,x))=df(\tilde\bb_t)(\XX(t,x))\quad\text{a.e.\ in $(0,T)$},\qquad\text{for $\mm$-a.e.\ $x\in X$.}
$$
With a similar argument, one can show that if we modify not only $df(\bb)$, but also $f$ in a $\mm$-negligible set,
to obtain a Borel representative $\tilde f$, then
$f(\XX(\cdot,x))\in W^{1,1}(0,T)$ and $\frac{d}{dt} f(\XX(t,x))=\bb_t(\XX(t,x))$ 
for a.e.\ $t\in (0,T)$ if and only if $\tilde f(\XX(\cdot,x))\in W^{1,1}(0,T)$ and $\frac{d}{dt} \tilde f(\XX(t,x))=\bb_t(\XX(t,x))$ 
for a.e.\ $t\in (0,T)$, because Fubini's theorem gives $\tilde f(\XX(t,x))=f(\XX(t,x))$ for a.e.\ $t\in (0,T)$, for $\mm$-a.e.\ $x\in X$. 
For this reason the choice of a Borel representative of $f\in\Algebra$ is not really important. Whenever
this is possible, the natural choice of course is given by the continuous representative.\fr
}\end{remark}

The main result of the section is the following existence and uniqueness result. We stress that uniqueness is understood in the
pathwise sense, namely $\XX(\cdot,x)=\YY(\cdot,x)$ in $[0,T]$ for $\mm$-a.e.\ $x\in X$, whenever $\XX$ and $\YY$ are
regular Lagrangian flows relative to $\bb$.

\begin{theorem}[Existence and uniqueness of the regular Lagrangian flow]\label{thm:uniflow}
Assume \eqref{eq:basicbb}, and that the continuity equation induced by $\bb$ has uniqueness of solutions in ${\mathcal L}_+$ for all
initial data $\bar u\in L^1\cap\Lbm\infty$, as well as existence of solutions in the class \eqref{eq:morenarrow} for all
nonnegative initial datum $\bar u\in L^1\cap\Lbm\infty$. Then there exists a unique regular Lagrangian flow relative to $\bb$.
\end{theorem}
\begin{proof} Let $B\in\BorelSets{X}$ with positive and finite $\mm$-measure and let us build first a ``generalized'' flow starting from
$B$. To this aim, we take $\bar u=\chi_B/\mm(B)$ as initial datum and we apply first the assumption on existence of a solution
$u\in{\mathcal L}_+$ starting from $\bar u$, with $u_t\leq C(\bb)/\mm(B)$, and then
the superposition principle stated in
Theorem~\ref{thm:superpo} to obtain $\eeta\in\Probabilities{C([0,T];X)}$ whose time marginals are $u_t\mm$, concentrated
on solutions to the ODE $\dot\eta=\bb_t(\eta)$. Then, Theorem~\ref{thm:nosplitting} below (which uses the uniqueness
part of our assumptions relative to the continuity equation) provides a representation
$$
\eeta= \frac{1}{\mm(B)}\int_B \delta_{\eta_x} d\mm(x),
$$
with $\eta_x\in C([0,T];X)$, such that $\eta_x(0)=x$ and $\dot\eta_x=\bb_t(\eta)$. Setting $\XX(\cdot,x)=\eta_x(\cdot)$ for $x\in B$, it follows that $\XX:B\times [0,T]$ is a regular flow, relative to $\bb$, 
with the only difference that (i) and (ii) in Definition~\ref{def:dregflow} have
to be understood for $\mm$-a.e.\ $x\in B$, and 
\begin{equation}\label{eq:oracena}
\XX(t,\cdot)_\#(\bar u\mm)=(e_t)_\#\eeta=u_t\mm\leq \frac{C(\bb)}{\mm(B)}\mm.
\end{equation}

Next we prove consistency of these ``local'' flows $\XX_B$. If $B_1\subset B_2$ with $\mm(B_1)>0$ and $\mm(B_2)<\infty$,
we can consider the measure
$$
\eeta:=\frac{1}{2\mm(B_1)}\int_{B_1} \bigl(\delta_{\sXX_{B_1}(\cdot,x)}+\delta_{\sXX_{B_2}(\cdot,x)}\bigr)\,d\mm(x)
\in\Probabilities{C([0,T];X)}
$$
to obtain from Theorem~\ref{thm:nosplitting} that $\XX_{B_1}(\cdot,x)=\XX_{B_2}(\cdot,x)$ for $\mm$-a.e.\ $x\in B_1$.

Having gained consistency, we can build a regular Lagrangian flow by considering a nondecreasing sequence of
a Borel sets $B_n$ with positive and finite $\mm$-measure whose union covers $\mm$-almost all of $X$ and the
corresponding local flows $\XX_n:B_n\times [0,T]\to X$. Notice
that we needed a quantitative upper bound on $\XX_n(t,\cdot)_\#(\chi_{B_n}\mm)$ precisely in order to be able to
pass to the limit in condition (iii) of Definition~\ref{def:dregflow}, since \eqref{eq:oracena} gives
$\XX(t,\cdot)_\#(\chi_B\mm)\leq C(\bb)\mm$.

This completes the existence part. The uniqueness part can be proved using once more Theorem~\ref{thm:nosplitting}
and the same argument used to show consistency of the ``local'' flows. 
\end{proof}

\begin{theorem}[No splitting criterion]\label{thm:nosplitting}
Assume \eqref{eq:basicbb}, and that the continuity equation induced by $\bb$ has at most one solution in ${\mathcal L}_+$ for all
$\bar u\in L^1\cap\Lbm\infty$. Let $\eeta\in\Probabilities{C([0,T];X)}$ satisfy:
\begin{itemize}
\item[(i)] $\eeta$ is concentrated on solutions $\eta$ to the ODE $\dot\eta=\bb_t(\eta)$;
\item[(ii)] there exists $L_0\in [0,\infty)$ satisfying
\begin{equation}\label{eqn:noconcentration}
(e_t)_\#\eeta\leq L_0\mm\qquad\forall \,t\in [0,T].
\end{equation}
\end{itemize}
Then the conditional measures $\eeta_x\in\Probabilities{C([0,T];X})$ induced by the map $e_0$ are Dirac masses for
$(e_0)_\#\eeta$-a.e.\ $x$; equivalently, there exist $\eta_x\in C([0,T];X)$ such that $\eta_x(0)=x$ and solving the ODE
$\dot\eta_x=\bb_t(\eta_x)$, satisfying
$\eeta= \int \delta_{\eta_x} d(e_0)_\#\eeta(x)$.
\end{theorem}
\begin{proof} Using the uniqueness assumption at the level of the continuity equation, as well as the implication provided by Lemma~\ref{eq:fromODEtoPDE},
the decomposition procedure of \cite[Thm.~18]{bologna} (that slightly improves the original argument of \cite[Thm.~5.4]{Ambrosio03}, 
where comparison principle for the continuity equation was assumed) gives the result.
\end{proof}

\section{Examples}\label{sec:examples}

In this section, on one hand we illustrate relevant classes of metric measure spaces for which our abstract theory applies. On the other hand we try
to compare our results on the well-posedness of the continuity equation with the ones obtained in other papers, for particular classes of spaces.
Several variants of the existence and uniqueness results are possible, varying the regularity and the growth conditions imposed on $\bb$ and
on the density $u_t$; we focus mainly on the issue of uniqueness, since existence in particular classes of spaces (e.g.\ the Euclidean ones)
can be often be obtained by ad hoc methods (e.g.\ convolution of the components of the vector field, which preserve bounds on divergence) 
not available in general spaces. Also,
we will not discuss the existence/uniqueness of the flow, which follow automatically from well-posedness at the PDE level 
using the transfer mechanisms presented in Section~\ref{sec:RLF}. We list the examples following, to some extent, chronological order 
and level of complexity.

\subsection{Euclidean spaces: the DiPerna-Lions theory}\label{sec:diperna-lions}

The theory of well posedness for flows and for transport and continuity equations was initiated by 
DiPerna-Lions in \cite{DiPerna-Lions89} and it (quite obviously) fits into our abstract setting. More explicitly we let, in the basic setup 
\eqref{E-assumptions}, $X=\R^n$, $\mm = {\mathscr L}^n$ the Lebesgue measure and
$$\mathcal{E}(f) = \int \abs{\nabla f}^2(x) d{\mathscr L}^n(x),\quad \text{for $f\in W^{1,2}(\R^n)$,}$$
so that $\Delta$ is the usual Laplacian and $(\sfP_t)_t$ is the heat semigroup, that corresponds (up to a factor $2$ in the time scale) to the transition semigroup of the Brownian motion, which is conservative.
The algebra $\Algebra$ of Section~\ref{subsec:algebra} can be chosen to be the space of Lipschitz functions with compact support.

Given a Borel vector field $b = \sum_{i=1}^n b^i e_i$, with $b \in (L^1+L^\infty)^n$, its associated derivation $\bb$ is 
\[ \Algebra \ni f \mapsto df(\bb) = \sum_{i =1}^n b_i \frac{\partial f}{\partial x^i}. \]
Obviously, $\div \bb$ is the usual distributional divergence and $ D^{sym}\bb$ is the symmetric part of the distributional derivative of $\bb$.
Then, the uniqueness Theorem~\ref{thm:uniqueness} above corresponds to \cite[Corollary~II.1]{DiPerna-Lions89}, as long as $q\in (1,\infty]$.

On the other hand, in Euclidean spaces the strong local convergence of commutators depends on local regularity assumptions
on $\bb$ (and the use of convolutions with compact support), while our setting is intrinsically global. In order to adapt our methods
to this case, one could ``localize the Dirichlet form'' by considering $X = B_r(0)$ and the form
$$\mathcal{E}_r(f) = \int_{B_r} \abs{\nabla f}^2 d{\mathscr L}^n, \quad\text{for $f\in H^1(B_r)$.}$$
Thus $\Delta$ would be the Laplacian with Neumann boundary conditions and $(\sfP_t)_t$ would be the semigroup correspondent to the 
Brownian motion reflected at the boundary $\partial B_r(0)$, which is still conservative. Since the ball is convex, it can be proved that $\BE_2(0,\infty)$
still holds, see for instance \cite[Thm.~6.20]{AGS11b}.

A second major difference is that uniqueness assuming the regularity $b\in (W^{1,1})^n$ (or even $b \in (BV)^n$, the case
considered in \cite{Ambrosio03}) is not covered. Indeed, the $BV$ case seems difficult to reach in the abstract setting, 
due to the present lack of a covariant derivative (but see \cite{gigli-WIP}).

\subsection{Weighted Riemannian manifolds}\label{sec:riemannian}

Our arguments extend the classical DiPerna-Lions theory to the setting of weighted Riemannian manifolds. Of course, 
in order to prove strong convergence of commutators and the fact that solutions are renormalized 
one can always argue by local charts, but computations become more cumbersome, compared to the
Euclidean case, and here the 
advantages of our intrinsic approach become more manifest.

Let $(M,\gg)$ be a smooth Riemannian manifold and let $\mu$ be its associated Riemannian volume measure. Assume that the Ricci curvature tensor $\mathrm{Ric}_{\sgg}$ is pointwise bounded from below (in the sense of  quadratic forms) by some constant $K \in \R$. More generally, one can add a ``weight'' $V: M \to \R$ to the measure, i.e.\ consider a smooth non-negative function and assume that the Bakry-\'Emery curvature tensor is bounded from below by $K\in\R$, i.e.\
\[ \mathrm{Ric}_\sgg + \mathrm{Hess}(V) \ge K.\]
The form (on smooth compactly supported functions)
\[ f \mapsto \mathcal{E}_V (f) = \int_M \gg(\nabla f, \nabla f) e^{-V} d\mu,\]
is closable and we are in the setup \eqref{E-assumptions}.
Once more, the algebra $\Algebra$ of Section~\ref{subsec:algebra} can be chosen to be the space of Lipschitz functions with compact support.

When $V = 0$, Bochner's formula entails that $\BE_2(K,\infty)$ holds and it is a classical result due to S.-T.\ Yau that the heat semigroup is conservative. In the case of weighted measures, analogous results can be found in \cite[Prop.~6.2]{Bakry-94b} for the curvature bound and in \cite[Thm.~9.1]{grigorian-99} for the conservativity of $\sfP$, relying on a correspondent volume comparison theorem, see e.g.\ \cite[Thm.~1.2]{wei09}.

Given a Borel vector field $b$, i.e.\ a Borel section of the tangent bundle of $M$, its associated derivation $\bb$ acts on smooth functions by
$$ f \mapsto df (\bb) = \gg(b,\nabla f).$$
The divergence can be given in terms of the $\mu$-distributional divergence of $b$ by
\[ \div\bb = \div b - \gg(\bb,\nabla V), \]
while the deformation is the symmetric part of the distributional covariant derivative, see 
Remark~\ref{rem:deformation-smooth}.

\subsection{Abstract Wiener spaces}\label{sec:examples-wiener}

Let $(X,\gamma,\mathcal{H})$ be an abstract Wiener space, i.e.\ $X$ is a separable Banach space, 
$\gamma$ is a centered non-degenerate Gaussian measure on $X$, with covariance operator $Q: X^* \mapsto X$ and $\mathcal{H} \subset X$ 
is its associated Cameron-Martin space, which is naturally endowed with a Hilbertian norm. Moreover, $QX^*\subset\mathcal{H}$.

We define the set of smooth cylindrical functions $\mathcal{FC}_b^\infty(X)$ as the set of all functions $f(x)$ representable as 
 $\varphi(x_1^*(x),\ldots,x_n^*(x))$, with $\varphi: \R^n\to\R$ smooth and bounded, $x_i^*\in X^*$ for $i \in\cur{1,\ldots,n}$, for some integer $n\geq 1$.

We introduce a notion of ``gradient'' on functions $f \in \mathcal{FC}_b^\infty(X)$ letting $\nabla_\mathcal{H} f = Q df$, 
where $df$ is the Fr\'echet differential of $f$. With these definitions, for $f=\varphi(x_1^*,\ldots,x_n^*)$, there holds
\[\nabla_{\mathcal{H}} f (x) = \sum_{j=1}^n \frac{\partial\varphi}{\partial z_j} Qx_j^*=
\sum_i \frac{\partial f}{\partial h_i}  (x) h_i,\quad \text{where } \quad\frac{\partial f}{\partial h_i} (x) 
= \lim_{\veps \to 0} \frac{ f\bra{x+ \veps h_i} - f\bra{x} }{\veps}\]
where $(h_i)$ is any orthonormal basis of $\mathcal{H}$.

It is well-known \cite{Bouleau-Hirsch91} that Sobolev-Malliavin calculus on $(X,\gamma,\mathcal{H})$ fits into the setting \eqref{E-assumptions}, 
considering the closure of the quadratic form
\[ \mathcal{E}(f) = \int_X \abs{\nabla_{\mathcal{H}} f }^2_{\mathcal{H}} d\gamma,\quad \text{ for every $f \in \mathcal{FC}_b^\infty(X)$.}\]
The domain $\V$ coincides with the space $W^{1,2}(X,\gamma)$. The semigroup $\sfP$ is the Ornstein-Uhlenbeck semigroup, given by Mehler's formula
\[\sfP_t f (x) = \int_X f( e^{-t} x + \sqrt{1 - e^{-2t}} y ) d\gamma(y), \quad\text{for $\gamma$-a.e.\ $x \in X$.}\]
From this expression, it is easy to show that $\BE_2(1,\infty)$ holds (indeed, on cylindrical functions $\nabla_\mathcal{H} \sfP_t f=e^{-t}\sfP_t\nabla_\mathcal{H} f$,
understanding the action of the semigroup componentwise); it is a classical result that $\cE$ is \emph{quasi-regular}, see e.g.\ \cite[Thm.~5.9.9]{Bogachev98}. 
We let $\Algebra = \mathcal{FC}_b^\infty(X)$, which is well-known to be dense in every $L^p$-space and satisfy \eqref{eq:Feller} by Mehler's formula above: in particular we obtain density in $\V_p$ spaces by Lemma~\ref{lemma:feller-density} and Lemma~\ref{lem:refining-algebra}. 

Given an $\mathcal{H}$-valued field $b = \sum_i b^i h_i$, we introduce the derivation $f \mapsto \bb(f)  = \sum_i b^i \frac{\partial f}{\partial h_i}$
and we briefly compare our well-posedness results for the continuity equation with those contained in \cite{AF-09}.
 Combining Proposition~\ref{prop:hilbert-schmidt} and the subsequent remark, we obtain that our notion of deformation for  $\bb$ is comparable to that of $(\nabla b)^{sym}$ introduced in \cite[Def.~2.6]{AF-09}. Precisely, it can be proved that if $b \in LD^q(\gamma;\mathcal{H})$ for some $q>1$, then the deformation of $\bb$ is of type $(r,s)$, for any $r$, $s$, with $q^{-1}+r^{-1} +s^{-1} = 1$. It is then easy to realize that Theorem~\ref{thm:uniqueness} entails the uniqueness part of \cite[Thm.~3.1]{AF-09}, with the 
exception, as we observed in connection to the Euclidean theory, of the case $b\in W^{1,1}(X,\gamma;\mathcal{H})$ (the case $b\in BV(X,\gamma;\mathcal{H})$ 
has been recently settled in \cite{Trevisan-13}).

\subsection{Gaussian Hilbert spaces} \label{sec:examples-daprato}

We let $X = H$ be a separable Hilbert space, with norm $\abs{\cdot}$, in the setting introduced in the previous section, namely
$\gamma$ is a Gaussian centered and nondegenerate measure in $H$. By identifying $H = H^*$ via the Riesz isomorphism induced by the norm, 
the covariance operator $Q: H \to H$ is a symmetric positive trace class operator, thus compact. 
In this setting the Cameron-Martin space is $\mathcal{H} = Q^{1/2} H$, with the norm $\abs{h}_{\mathcal{H}} = \abs{Q^{-1/2} h}$. 

We let $(e_i)\subset H$ be an orthonormal basis of $H$ consisting of eigenvectors of $Q$, with eigenvalues $(\lambda_i)$, i.e.\ $Qe_i = \lambda_i e_i$ for every $i\ge 1$: in this setting, we define the class of smooth cylindrical functions $\mathcal{FC}_b^\infty(H)$ as those functions $f: X \to \R$ of the form $f(x) = \varphi( \ang{e_i, x}, \ldots \ang{e_n, x})$, with  $\varphi: \R^n \to \R$ smooth and bounded. Given $f \in \mathcal{FC}_b^\infty(H)$, from its Fr\'echet derivative $d f: H \mapsto H^*$ we introduce $\nabla f: H \mapsto H$ via $H = H^*$, in coordinates:
\[ \nabla f (x) = \sum_{i} \partial_i f (x) e_i, \quad\text{where } \partial_i f(x) = \lim_{\veps \to 0} \frac{ f\bra{x+ \veps e_i} - f\bra{x} }{\veps}. \]

To recover the abstract setting of the previous section, notice that the family $h_i = \lambda_i^{1/2} e_i$ is an orthonormal basis of $\mathcal{H}$ and 
that $\partial/\partial h_i = \lambda^{-1/2}_i \partial_i$, thus it holds that $Q \nabla f = \nabla_\mathcal{H} f$.

For $\alpha \in \R$, we introduce the form
\[  \mathcal{E}^\alpha(f) = \int_X \bigl|Q^{{(1-\alpha)}/{2}} \nabla f \bigr|^2 d\gamma,\quad \text{ $f \in \mathcal{FC}_b^\infty(H)$},\]
which is closable: its domain is the space $W_\alpha^{1,2} (H,\gamma)$, see \cite[Chapters 1 and 2]{DP04} for more details. Evidently, we recover \eqref{E-assumptions}, with $\Gamma(f) = \sum_i \lambda_i^{1-\alpha} \bigl|\partial_i f\bigr|^2$. Notice that the associated distance is the one induced by the norm 
$|Q^{{(\alpha-1)}/{2} } x|$, which is extended if and only if $\alpha <1$.

The associated semigroup can be still be seen as the transition semigroup of an infinite dimensional SDE,
and its infinitesimal generator $\Delta_\alpha$ is given by
\[ \Delta_\alpha f (x) =   \mathrm{Tr}\sqa{Q^{1-\alpha} D^2f (x) } -\ang{ x, Q^{-\alpha} \nabla f(x)}, \quad \text{$f \in \mathcal{FC}_b^\infty(H)$.}\]
It can be shown that $\BE_2(1,\infty)$ holds \cite[Prop.~2.60]{DP04}. We let $\Algebra = \mathcal{FC}_b^\infty(H)$, which is dense in every $\Lbm p$ space and satisfies \eqref{eq:Feller}, thus obtaining density results in $\V_p$ ($p \in [1,\infty)$) by Lemma~\ref{lemma:feller-density} and Lemma~\ref{lem:refining-algebra}.

For $\alpha = 0$, we recover the abstract Wiener space setting discussed above, while for $\alpha =1$ we obtain the setting of \cite{DaPrato-Flandoli-Rockner-13}. We show that our results encompass those in \cite{DaPrato-Flandoli-Rockner-13} and analogues hold for any $\alpha \in \R$.

Given $b: H \mapsto H$,  $b = \sum_i b_i e_i$ Borel, we consider the map
\[ \mathcal{FC}_b^\infty(H) \ni f \mapsto df(\bb):= \ang{b, \nabla f}_H = \sum_i b_i \partial_i f.\]
If $\abs{Q^{(\alpha-1)/2} b} \in L^q(H,\gamma)$ for some $q \in [1,\infty]$, then $\bb$ is a well-defined derivation, with $\abs{\bb} \le \abs{Q^{(\alpha-1)/2} b}$.

The Cameron-Martin theorem entails an integration by parts formula \cite[Thm.~1.4 and Lemma 1.5]{DP04} that reads in our notation as
\[ \div \ee_i (x) = -\frac{ \ang{e_i, x} }{\lambda_i}, \quad \text{where $df(\ee_i) = \partial_i f$.} \]
On smooth ``cylindrical'' fields $b = \sum_i^n b_i e_i$, this gives
\[  \div  \bb (x) = \sum_{i} \partial_i b_i(x) - \frac{ \ang{e_i, x} }{\lambda_i} b_i,\]
where the series reduces to a finite sum. Notice that the expression does not depend on $\alpha$ but only on $\gamma$, 
in agreement with the notion of divergence as dual to derivation.

Notice also that the boundedness of the Gaussian Riesz transform \cite[Prop.~5.88]{Bogachev98} entails that if $\bb \in W^{1,p}(H,\gamma, \mathcal{H})$, then $\div \bb \in L^p(H,\gamma)$. These are only sufficient conditions and their assumptions would force us to limit the discussion to $\mathcal{H}$-valued fields, as in \cite{DaPrato-Flandoli-Rockner-13} (see Section 5 therein). Our results hold even for some classes of
fields not taking values in $\mathcal{H}$, see at the end of this section. 

Arguing on smooth cylindrical functions,
\begin{equation} \label{eq:hilbert-schmidt-deformation} \int  D^{sym}\bb (f,g) d \gamma = \int \sum_{i,\,j} \frac 1 2 \sqa{ \bra{ \frac{ \lambda_i} {\lambda_j}}^{{(1-\alpha)}/{2}}\hskip -18pt  \partial_i b_j + \bra{ \frac{ \lambda_j} {\lambda_i}}^{{(1-\alpha)}/{2}} \hskip -18pt\partial_j b_i} \bra{ \lambda_i^{{(1-\alpha)}/{2}} \partial_i f} \bra{ \lambda_j^{{(1-\alpha)}/{2}} \partial_j f} d\gamma,\end{equation}
thus our bound on $ D^{sym}\bb$ is implied by an $L^q$ bound of the Hilbert-Schmidt norm of the expression is square brackets above (a fact that could also be seen as a consequence of Proposition~\ref{prop:hilbert-schmidt} and the subsequent remark, setting $V_i(x) = \lambda_i^{(\alpha-1)/2}\ang{e_i, x}$, for $i\ge 1$).

Comparing our setting with that in \cite{DaPrato-Flandoli-Rockner-13}, it is clear that Theorem 2.3 therein is a consequence of Theorem~\ref{thm:uniqueness}.

We end this section considering a field $b$ taking values outside $\mathcal{H}$, to which our theory applies (although well-posedness was already shown in \cite{MW-Zakai-05}). Assume that that each eigenvalue of $Q$ admits a two-dimensional eigenspace thus, slightly changing the notation, we write $(e_i, \tilde{e}_i)$ for an orthonormal basis of $H$ consisting of eigenvectors of $Q$. We let
\[b = \sum_{i=1}^\infty \lambda_i^{1/2} \sqa{(\div \tilde{\ee}_{i}) e_{i} - (\div \ee_{i}) \tilde{e}_{i}}, \quad \text{thus}  \quad \int \bigl| 
Q^{{(\alpha-1)}/{2}} b\bigr|^2 d\gamma = \sum_{i=1}^\infty \lambda_i^\alpha.\]
The series above converges if $\alpha = 1$, and it does not if $\alpha =0$. Since $(\div \ee_{i}, \div \tilde{\ee}_{i})_i$ are independent, Kolmogorov's $0$-$1$ law entails that $b$ is well defined as an $H$-valued map, but $b(x) \notin \mathcal{H}$ for $\gamma$-a.e.\ $x\in H$. 
The derivation $\bb$ is therefore well-defined if $\alpha =1$, and $\abs{\bb} \in \Lbm 2$. From its structure and \eqref{eq:hilbert-schmidt-deformation}, both its divergence and its deformation are seen to be identically $0$, thus our results apply.

\subsection{Log-concave measures} \label{sec:examples-log-concave}

Let $(H, \abs{\cdot})$ be a separable Hilbert space and let $\gamma$ be a \emph{log-concave} probability measure on $H$, i.e.\ for all open sets $B,\,C \subset H$,
\[ \log \gamma\bra{ (1-t)B + tC} \ge (1-t) \log \gamma\bra{ B}  + t \log \gamma\bra{ C}, \text{for every $t\in [0,1]$.} \]
Assume also that $\gamma$ is non-degenerate, namely that it is not concentrated on a proper
closed subspace of $H$. Consider the quadratic form
\[ \mathcal{E}_\gamma (f) = \int \abs{ \nabla f}^2 d\gamma, \quad \text{defined for $f \in C^1_b(H)$, }\]
where $C^1_b(H)$ denotes the space of continuously Fr\'echet differentiable functions which are bounded together with their differential.

It is shown in \cite{AmbrosioSavareZambotti09} that the $\mathcal{E}_\gamma$ is closable, extending previous results obtained under more restrictive
assumptions on $\gamma$. Actually, since in \cite{AmbrosioSavareZambotti09} the so-called ${\sf EVI}$ property for the associated semigroup
$\sfP$ is proved, and since in \cite{AGS11b} this is proved to be one of the equivalent characterizations of $\RCD$, it follows that
$(H,\abs{\cdot},\gamma)$ is an $\RCD(0,\infty)$ space, thus the results in Section~\ref{sec:RCD} below apply and we already obtain abstract well-posedness result under no additional assumption on $\gamma$. Recall that in that abstract setting $\Algebra$ can be taken as the space of 
Lipschitz functions with bounded support.

Let $(e_i)_{i\geq 1}\subset H$ be an orthonormal basis. For every $f \in \V$, there exist
$f_n\in C^1_b(H)$ such that $f_n \to f$ in $L^2(\gamma)$ and
\[ \lim_{n,\,m \to \infty}  \mathcal{E}_\gamma( f_n - f_m) \to 0,\]
thus an $H$-valued ``gradient'' $\nabla f = \sum_i \partial_i f e_i$ is $\gamma$-a.e.\ defined in $H$.

Let $b: H \mapsto H$, $b = \sum_i b_i e_i$: we associate the derivation $f \mapsto df(\bb) = \sum_i b_i f_i$, thus $\abs{\bb} \le \abs{b}$. For $v\in H$, we write $\vv$ for the constant derivation corresponding to 
the constant vector field equal to $v$, and $\ee_i$ for the derivation corresponding to $e_i$.

Let us remark that, in this very general setting setting, bounds on the divergence of a given field $\bb$ seem to be difficult to obtain, even
for constant vector fields: this is due to the fact that presently it is not known whether every log-concave measure 
$\gamma$ admits at least one nonzero direction $v$ such that $\div \vv \in\Lbm 1$, \cite[\S 4.3]{Bogachev-10}.
On the other hand, our abstract arguments do not require any absolute continuity of $\gamma$ with respect to a
Gaussian or other product measures and combining our abstract well-posedness results with Theorem~\ref{gradients_have_deformation}, we are able to provide non-trivial derivations that admit a well-posed flow, e.g.\ gradient derivations of functions in $D_{L^4}(\Delta)$, such as those of the form $\int_{1}^2 \sfP_t f dt$, for $f \in \Lbm 4$.

To state an explicit sufficient condition to bound the deformation of $\bb$, we assume that $\div\ee_i\ll\mm$ and, denoting by $\beta_i$ the density, we require that, for $i \ge 1$, $\beta_i \in \V$ or equivalently that the function $x \mapsto V_i(x) = \ang{e_i, x}$ satisfies $\Delta V_i \in \V$, thus  
Proposition~\ref{prop:hilbert-schmidt} gives $\nor{ D^{sym}\bb}_{r,s}<\infty$ if $\sqa{ \partial_i b_j + \partial_j b_i}_{i,j} \in L^q(\gamma; \ell^2(\N \otimes \N))$, provided that $r$, $s \ge 4$.

We conclude by comparing our results in this setting with \cite[Thm.~7.6]{Kolesnikov-Rockner-13}, where uniqueness for the continuity equation is obtained in the case of log-concave measures formally given by $\gamma = \text{``}e^{-V} d{\mathscr L}^\infty\text{''}$, for convex Hamiltonians $V$ of specific form. In particular, the assumptions on $\beta_i$ imposed therein are stronger than ours. Their assumptions on the field $b$ in \cite{Kolesnikov-Rockner-13} entail that $\abs{\bb} \in L^{a_1}(\gamma)$, for some $a_1 > 1$ and that $\sqa{ \partial_i b_j + \partial_j b_i}_{i,j} \in L^{a_2}(\gamma; \ell^2(\N \otimes \N))$, for some $a_2 > 4$. Moreover, to deduce uniqueness, $\div \bb \in L^q(\gamma)$, for some $q>1$ is also assumed. Therefore, if $a_1 \ge 2$ and $q \ge 2$, we are in a position to recover, via Proposition~\ref{prop:hilbert-schmidt}, such a uniqueness result as a special case of Theorem~\ref{thm:uniqueness}.

\subsection{$\RCD(K,\infty)$ metric measure spaces}\label{sec:RCD}

Recall that the class $\CD(K,\infty)$, introduced and deeply studied in \cite{Lott-Villani09}, \cite{Sturm06I}, \cite{Sturm06II} consists of complete metric measure spaces such that
the Shannon relative entropy w.r.t.\ $\mm$ is $K$-convex along Wasserstein geodesics, see \cite{Villani09} for a full account of the theory and its geometric and
functional implications.
The class of $\RCD(K,\infty)$ metric measure spaces was first introduced in \cite{AGS11b}, from a metric perspective, as class of spaces smaller than that of 
$\CD(K,\infty)$ metric measure spaces. The additional requirement, in this class of spaces, is that the so-called Cheeger energy is quadratic; with this
axiom, Finsler geometries are ruled out and stronger structural (and stability) properties can be estabilished. 
Subsequently, connections with the theory of Dirichlet forms gave rise to a series of works, \cite{AGS12}, \cite{Savare-13} and 
\cite{Ambrosio-Mondino-Savare-13}. For a brief introduction to the setting and its notation, we refer to Sections~4.1 and 4.2 in \cite{Savare-13}, 
and in particular to Theorem 4.1 therein, which collects non-trivial equivalences among different conditions.

We will use the notation $W^{1,2}(X,d,\mm)$ for the Sobolev space, $\textsf{Ch}$ for the Cheeger energy arising from the relaxation
in $L^2(X,\mm)$  of the local Lipschitz constant 
\begin{equation}
  \label{eq:20}
  |\rmD f|(x):=\limsup_{y\to x}\frac{|f(y)-f(x)|}{d(y,x)}
\end{equation}
of $\Lbm2$ and Lipschitz maps $f:X\to\R$.

To introduce $\RCD(K,\infty)$ spaces we restrict the discussion to metric measure spaces $(X,d,\mm)$ satisfying the following three conditions:
\begin{itemize}
\item[(a)] $(X,d)$ is a complete and separable length space;
\item[(b)] $\mm$ is a nonnegative Borel measure with $\supp(\mm)=X$, satisfying
\begin{equation}\label{eq:grygorian}
\mm(B_r(x))\leq c\,\rme^{Ar^2}
\end{equation}
for suitable constants $c\geq 0$, $A\geq 0$;
\item[(c)] $(X,\sfd,\mm)$ is infinitesimally Hilbertian according to the terminology introduced in \cite{Gigli2012}, i.e., the Cheeger
energy $\textsf{Ch}$ is a quadratic form. 
\end{itemize}
As explained in \cite{AGS11b}, \cite{AGS12}, the quadratic form $\textsf{Ch}$ canonically 
induces a strongly regular, strongly local Dirichlet form $\cE$ in $(X,\tau)$ (where $\tau$ is the topology induced by the distance $d$), 
as well as a Carr\'e du champ $\Gamma: D(\cE)\times D(\cE)\to\Lbm 1$. Thus, we recover the basic setting of 
\eqref{E-assumptions} and we can identify $W^{1,2}(X,d,\mm)$ with
$\V$. In addition, $\sfP$ is conservative because of \eqref{eq:grygorian} and the definition of $\textsf{Ch}$ provides the
approximation property
\begin{equation}\label{eq:approximation_chee} 
\text{$\exists\,\,f_n\in {\rm Lip}(X)\cap \Lbm 2$ with $f_n\to f$ in $\Lbm 2$ and $|\rmD f_n|\to \sqrt{\Gamma(f)}$ in $\Lbm 2$}
\end{equation}
for all $f\in\V$.

The above discussions justify the following definition of $\RCD(K,\infty)$. It is not the original one given in
\cite{AGS11b} we mentioned at the beginning of this section, but it is more appropriate for our purposes;
the equivalence of the two definitions is given in \cite{AGS12}.

\begin{definition}[$\RCD(K,\infty)$ metric measure spaces]
  \label{def:topo-BE}
  We say that 
  $(X,d,\mm)$, satisfying \emph{(a)},  \emph{(b)},  \emph{(c)} above, is an $\RCD(K,\infty)$ space if:
  \begin{itemize}
  \item[(a)]
  the Dirichlet form associated to the Cheeger
  energy of $(X,d,\mm)$ satisfies $\BE_2(K,\infty)$ according to
  Definition~\ref{def:BE};
  \item[(b)]
  any $f\in W^{1,2}(X,d,\mm)\cap\Lbm\infty$ with
      $\big\|\Gq f\big\|_\infty\le 1$ has a $1$-Lipschitz representative.
\end{itemize}
\end{definition}

From \cite[Lemma 6.7]{AGS11b} we obtain that $\cE$ is \emph{quasi-regular}. We set throughout $\Algebra$ be the class of
Lipschitz functions with bounded support. It is easily seen that $\Algebra$ is dense in $\V$.

\begin{lemma}\label{lem:algebra-*-RCD}
There exists a countable set $\Algebra^* \subset \Algebra$ with $\nor{\Gamma(f)}_\infty \le 1$ for every $f \in \Algebra^*$, such that \eqref{eq:propertiesGstar1} is satisfied, the distance $d_{\Algebra^*}$ defined by \eqref{def:d_cE} in Remark~\ref{rem:hiddendistance} coincides with $d$ and
for any derivation $\bb$ one has
\begin{equation} \label{eq:b=b_*}\text{$|\bb|= |\bb|_*$, $\mm$-a.e.\ in $X$, where $|\bb|_*$ is defined in \eqref{eq:modulus-*}.}\end{equation}
\end{lemma}

We do not deduce that \eqref{eq:propertiesGstar2} holds for $\Algebra^*$: however, since we prove that $d_{\Algebra^*}$ coincides with $d$, Lemma \ref{eq:fromODEtoPDE} entails that solutions to the ODE $\dot{\eta}=\bb_t (\eta)$ are absolutely continuous curves. Moreover, in order to apply Theorem \ref{thm:superpo} rigorously in this setting, e.g.\ in Theorem \ref{thm:ediss} below, it is sufficient to endow $X$ with the coarsest topology that makes all the functions in $\Algebra^*$ continuous (see also Remark~\ref{rem:hiddendistance}).

\begin{proof} Since both $(X,d)$ and $\V$ are separable, it is not difficult to exhibit a countable family $\Algebra^* \subset\Algebra$ such that \eqref{eq:propertiesGstar1} is satisfied: let $(x_h) \subset X$ be dense, and set $f_{h,k} := (d(x_h, \cdot) - k)^- \in \Algebra$ for $h,\,k \in \N$;  then,
define $${\mathscr B} := \bigcup_{h,\,k=0}^\infty \{f_{h,k}\}\cup\bigcup_{h=0}^\infty\{g_h\},$$ 
with $(g_h)\subset\Algebra$ dense in $\V$. Then, defining $\Algebra^*=\{f\in{\mathscr B}:\ \|\Gamma(f)\|_\infty\leq 1\}\subset\Algebra$, since
$\R\Algebra^*={\mathscr B}$ we obtain \eqref{eq:propertiesGstar1}.

To show that the distances coincide, notice that $d_{\Algebra^*} \le d$ is obvious, while $d \le d_{\Algebra^*}$ follows from taking $f = f_{h,k}$ in \eqref{def:d_cE}, with $x_h$ arbitrarily close to $x$ and $k$ larger than $d(x,y)$.

We show that, up to a further enlargement of $\Algebra^*$, \eqref{eq:b=b_*} holds (notice that $d=d_{\Algebra^*}$ holds automatically for the enlargement, and we need only to retain
\eqref{eq:propertiesGstar1}). The problem reduces to argue that, for $f \in \Algebra^*$ one can improve the inequality $\abs{df(\bb)} \le \abs{\bb}_*$ into $\abs{df(\bb)} \le \Gamma(f)^{1/2}\abs{\bb}_*$, $\mm$-a.e.\ in $X$. This is based on a localization procedure akin to \cite[Proposition~3.11]{AGS12}, which leads to the key inequality \eqref{eq:intermediate-b=b*} below: the thesis follows then by a density argument, using the curvature assumption.

For any $\veps >0$, let $S_\veps \in C^1(\R)$ be a $1$-Lipschitz truncation function, given by $S_\veps(r) = \veps S_1(r/\veps)$, where $S_1(r)$ is a $1$-Lipschitz functions with
\begin{equation}
\begin{cases}
S_1(r)= 1&\text{for $r \le 1$;}\\
S_1(r)= 0 &\text{for $r \ge 3$.}
\end{cases}
\end{equation} 
We notice that $S_\veps (r) = \veps$ if $r \le \veps$.

Fix $f \in \V \cap C_b(X)$ and assume that $\zeta: X \to [0,\infty)$ is a bounded upper semicontinuous function, such that $\Gamma(f)^{1/2} \le \zeta$, $\mm$-a.e.\ in $X$. For any $h \ge 1$, $\veps>0$, $M >0$, such that $M \ge \sup_{B(x_h, 3\veps)} \zeta$, we introduce the following ``localization'' of $f$ at $x_h \in X$  (as above, $(x_h)_{h\ge1} \subset X$ is dense in $X$):
\[ T_{h,\veps,M} (f)(\cdot) := \frac{f(\cdot) - f(x_h)}{M} \land \sqa{S_\veps \circ d(\cdot, x_h)} \lor \sqa{- S_\veps \circ d(\cdot,x_h)}. \]
The following properties are easy to check:
\begin{itemize}
\item[(i)] $T_{h,\veps,M} (f) \in \V$ is supported in $B(x_h, 3\veps)$, and $T_{h,\veps,M} (f)(x_h)=0$;
\item[(ii)] $\Gamma(T_{h,\veps,M} (f))^{1/2} \le (\Gamma(f)^{1/2}/M) \le 1$ on $B(x_h, 3\veps)$ and $\Gamma(T_{h,\veps,M} (f)) =0$ outside $B(x_h, 3\veps)$, thus $T_{h,\veps,M} (f)$ and is $1$-Lipschitz by condition $(b)$ in Definition \ref{def:topo-BE};
\item[(iii)] combining (i) and (ii), it holds $\abs{T_{h,\veps,M} (f)(x)} \le d(x_h, x)$ in $X$, thus $\abs{T_{h,\veps,M} (f)(x)} < \veps$ in $B(x_h, \veps)$, so that $T_{h,\veps,M} (f) = (f - f(x_h))/M$ in $B(x_h, \veps)$.
\end{itemize}

From (i) and (ii) we obtain $T_{h,\veps, M} f  \in \Algebra$, which together with (iii) leads to the identity
\begin{equation}
\label{eq:identity-localized-derivation} df(\bb) = M d T_{h,\veps, M}( f )(\bb), \quad \text{$\mm$-a.e.\ in $B(x_h, \veps)$.}
\end{equation}
Indeed, for every  $g \in \Algebra$ and $a \in \R$,  it holds $dg(\bb)=0$ $\mm$-a.e.\ in the set $\cur{g = a}$ as a consequence  of \eqref{eq:72}, with $N = \cur{a}$, and the upper bound $\abs{dg(\bb)} \le \abs{\bb} \Gamma(g)^{1/2}$. In the situation that we are considering, take $g = f - M T_{h,\veps, M} f$ and $a = f(x_h)$.

Let us assume that $T_{h,\veps,M} (f) \in \Algebra^*$, for every $h \ge 1$ and rational numbers $\veps$, $M >0$, such that $M \ge \sup_{B(x_h, 3\veps)} \zeta$. We claim that it holds
\begin{equation}\label{eq:intermediate-b=b*} \abs{df(\bb)} \le \zeta \abs{\bb}_*, \quad \text{$\mm$-a.e.\ in $X$.}\end{equation}
Indeed, from \eqref{eq:identity-localized-derivation}, we deduce
\[  \abs{df(\bb)} (x)  = M \abs{ d T_{h,\veps,M} (f)(\bb) }  \le  M \abs{\bb}_*(x), \quad \text{$\mm$-a.e.\ $x \in B(x_h, 3\veps).$}\]
We pass to the infimum upon $M$ (which is rational and greater than $\sup_{B(x_h, 3\veps)} \zeta$) and $h \ge 1$, then we let $\veps \downarrow 0$, to obtain
\[  \abs{df(\bb)}(x) \le \limsup_{\veps \down 0} \inf_{h: d(x_h,x)< \veps} \sup_{B(x_h, 3\veps)} \zeta \abs{\bb}_*(x)\le  \zeta(x) \abs{\bb}_*(x), \quad \text{$\mm$-a.e.\ $x \in X$,}\]
where the second inequality holds by upper semicontinuity of $\zeta$.

Thanks to the curvature assumption, it is not difficult to show that the class of functions $f \in \V \cap C_b(X)$ that admit functions $\zeta$ as above is not empty: by \cite[Theorem~3.17]{AGS12}
the operator $\sfP_t$ maps $L^2\cap \Lbm\infty$ into $C_b(X)$ for every $t>0$. In addition, if $f\in \V$, it holds
\[ \Gamma( \sfP_{t} f)\le  e^{-2Kt}\sfP_t(\Gamma(f)), \quad \text{$\mm$-a.e.\ in $X$.}\]
Thus, if $\Gamma(f)\in\Lbm\infty$ we may let $\zeta^2$ be the continuous version of the expression in the right hand side above that we denote by $e^{-2Kt}\tilde{\sfP}_t(\Gamma(f))$ (see also \cite[Proposition 3.2]{AGS12}).
 
We are in a position to prove that, up to enlarging $\Algebra^*$, \eqref{eq:b=b_*} holds. More precisely, we let $(f_n)_{n \ge 1} \subseteq \Algebra$ be any countable family, with $\Gamma(f_n) \le 1$, $\mm$-a.e.\ in $X$, for $n\ge 1$, and such that the dilations $(\lambda f_n)_{\lambda \in \R, n \ge 1}$ provide a dense set in $\V$. We enlarge $\Algebra^*$ with the union of all functions 
\[ T_{\veps, h, M} (\sfP_t f_n),\]
for $n$, $h \ge 1$ and rational numbers $t$, $\veps$, $M >0$, such that $M^2 \ge \sup_{B(x_h, 3\veps)} e^{-2Kt}\tilde{\sfP}_t(\Gamma(f))$.

For every $n \ge 1$ and rational $t>0$, \eqref{eq:intermediate-b=b*} with $\sfP_t f_n$ in place of $f$ and $\zeta = e^{-Kt}\sqa{\tilde{\sfP}_t(\Gamma(f))}^{1/2}$ gives
\[ \abs{d \sfP_t f_n (\bb)} \le e^{-Kt}\sqa{\tilde{\sfP}_t(\Gamma(f_n))}^{1/2} \abs{\bb}_*, \quad \text{$\mm$-a.e.\ in $X$.}\]
We let $t \downarrow 0$ to deduce that
\[  \abs{d f_n (\bb)} \le \Gamma(f_n)^{1/2} \abs{\bb}_*, \quad \text{$\mm$-a.e.\ in $X$.}\]
By homogeneity, a similar inequality holds for $\lambda f_n$ in place of $f_n$ for every $\lambda \in \R$. To conclude, let $g \in \Algebra$ and let $(g_k)_k \subseteq (\lambda f_n)_{\lambda \in \R, n\ge 1}$ converge towards $g$ in $\V$. It holds
\[  \abs{d g (\bb)} \le \liminf_{k \to \infty}  \Gamma(g_k)^{1/2} \abs{\bb}_* + \Gamma(g_k -g)^{1/2} \abs{\bb} =  \Gamma(g)^{1/2} \abs{\bb}_*,\quad\text{$\mm$-a.e.\ in $X$,}\]
and we deduce that $\abs{\bb} \le \abs{\bb}_*$, $\mm$-a.e.\ in $X$.
\end{proof}

We discuss now the fine regularity properties of functions in $\V$, recalling some results
developed in \cite{AGS11a}. We start with the notion of $2$-plan.

\begin{definition}[$2$-plans]
We say that a positive finite measure $\eeta$ in $\Probabilities{C([0,T];X)}$ is a $2$-plan if 
$\eeta$ is concentrated on $AC([0,T]; (X,d))$
and the following two properties hold:
\begin{itemize}
\item[(a)] $\int \int_0^T \abs{\dot \eta}^2 (t) dt d\eeta(\eta)<\infty$;
\item[(b)] there exists $C\in [0,\infty)$ satisfying $(e_t)_\#\eeta\leq C\mm$ for all $t\in [0,T]$.
\end{itemize}
\end{definition}

Accordingly, we say that $V:X\to\R$ is {\it $W^{1,2}$ along $2$-almost every curve} if, for all $s\leq t$ and
all $2$-plans $\eeta$, the family of inequalities
\begin{equation}\label{eq:wash}
\int |V(\eta(s))-V(\eta(t))| d\eeta(\eta)\leq \int \int_s^t g(\eta(r))|\dot\eta(r)| dr d\eeta(\eta),
\quad\text{for all $s,\,t \in [0,T)$ with $s\le t$}
\end{equation}
holds for some $g\in\Lbm 2$.
Since Lipschitz functions with bounded support are dense in $\V$, a density argument
\cite[Thm.~5.14]{AGS11a} based on \eqref{eq:approximation_chee} provides the following result:

\begin{proposition} \label{prop:soboreg} Any $V\in\V$ is $W^{1,2}$ along $2$-almost every curve. In addition,
\eqref{eq:wash} holds with $g=\sqrt{\Gamma(V)}$.
\end{proposition}

Actually, a much finer result could be established (see \cite[\S 5]{AGS11a}), namely the existence of a representative $\tilde V$
of $V$ in the $\Lbm 2$ equivalence class, with the property that $\tilde V\circ\eta$ is absolutely continuous
$\ppi$-a.e.\ $\eta$ for any $2$-plan $\ppi$, with $|(\tilde V\circ\eta)'|\leq\sqrt{\Gamma(V)}|\dot\eta|$ a.e.\ in $(0,T)$.
However, we shall not need this fact in the sequel. Here we notice only that since $\chi_B\eeta$ is a $2$-plan for any Borel
set $B\subset C([0,T];X)$, it follows from \eqref{eq:wash} with $g=\sqrt{\Gamma(V)}$ that
\begin{equation}\label{eq:wash2}
 |V(\eta(s))-V(\eta(t))| \leq \int \int_s^t \sqrt{\Gamma(V)}(\eta(r))|\dot\eta(r)| dr, \qquad\text{for $\eeta$-a.e.\ $\eta$}
\end{equation}
for all $s,\,t \in [0,T)$ with $s\le t$.

Now, we would like to relate these known facts to solutions to the ODE $\dot\eta=\bb_t(\eta)$.
The first connection between $2$-plans and probability measures concentrated on solutions to the ODE 
is provided by the following proposition.

\begin{proposition}\label{proposition-test-plan}
Let $\bb = (\bb_t)$ be a Borel family of derivations with $\abs{\bb} \in L^1_t(L^2)$ and let
$u\in \Lbtx \infty \infty$. Let $\eeta$ be concentrated on solutions to the ODE $\dot\eta=\bb_t(\eta)$,
with $(e_t)_\#\eeta=u_t\mm$ for all $t\in (0,T)$.  Then $\eeta$ is a $2$-plan.
\end{proposition}
\begin{proof}
The fact that $\eeta$ has bounded marginals follows from the assumption $u\in \Lbtx \infty \infty$. By Lemma~\ref{eq:fromODEtoPDE} and the identification $d = d_{\Algebra^*}$, $\eeta$ is concentrated on $AC([0,T]; (X,d))$, with $\abs{\dot \eta}(t) = \abs{\bb_t} (\eta(t))$, $\Leb{1}$-a.e.\ in $(0,T)$ 
for $\eeta$-a.e.\ $\eta$. Thus,
 \[ \int \int_0^T \abs{\dot \eta}^2 (t) dt d\eeta(\eta) = \int_0^T \int \abs{\bb_t}^2 u_t d\mm dt < \infty.\]
\end{proof}

We now focus on the case of a ``gradient'' and time-independent derivation $\bb_V$ associated to $V\in\V$.
Recall that in this case $|\bb_V|^2=\Gamma(V)$ $\mm$-a.e.\ in $X$.

\begin{theorem}\label{thm:ediss}
Let $V \in D(\Delta)$ with $\Delta V^-\in\Lbm\infty$. Then, 
there exist weakly continuous solutions (in $[0,T)$, in duality with $\Algebra$) $u\in L^\infty_t(L^1_x\cap L^\infty_x)$ to the continuity equation, 
for any initial condition $\bar u\in L^1\cap\Lbm\infty$. In addition, if $\eeta$ is given by Theorem~\ref{thm:superpo} (namely $\eeta$ is concentrated on solutions to the ODE $\dot\eta=\bb_V(\eta)$ and
$(e_t)_\#\eeta=u_t$ for all $t\in (0,T)$), then:
\begin{itemize}
\item[(a)] $\eeta$ is concentrated on curves $\eta$ satisfying 
$\abs{\dot \eta} (t) = \Gamma(V)^{1/2} (\eta(t))$ for a.e.\ $t \in (0,T)$;
\item[(b)] for all  $s,\,t \in [0,T)$ with $s\le t$, there holds
\[  V\circ\eta(t)- V\circ\eta(s) = \int_s^t  \Gamma(V)(\eta(r)) dr, \quad \text{for $\eeta$-a.e.\ $\eta$.} \]
\end{itemize}
\end{theorem}

\begin{proof} The proof of the first statement follows immediately by Theorem~\ref{thm:existence-lp} with $r=\infty$.
Since
$$
 \int_s^t \int \Gamma(V,f) u_r d\mm dr = \int f u_t - \int f u_s \quad\text{for all $s,\,t\in [0,T)$ with $s\leq t$}
$$
for all $f\in\Algebra$, we can use the density of $\Algebra$ in $\V$ and a simple limiting procedure to obtain
\begin{equation}\label{eq:Vdiss1}
 \int_s^t \int \Gamma(V) u_r d\mm dr = \int V u_t - \int V u_s \quad\text{for all $s,\,t\in [0,T)$ with $s\leq t$.}
\end{equation}

If $\eeta$ is as in the statement of the theorem, since $\eeta$ is a $2$-plan we can combine Proposition~\ref{prop:soboreg} and the inequality
$|\dot\eta|\leq |\bb_V|(\eta)$ stated in Lemma~\ref{eq:fromODEtoPDE} to get
\[ \int V(\eta(t)) - V(\eta(s)) d\eeta(\eta)  \le \int \int_s^t \Gamma(V)^{1/2} (\eta(r)) \abs{\dot \eta}(r)dr d\eeta(\eta)\le 
\int\int_s^t\Gamma(V)(\eta(r)) d\eeta(\eta),\] 
for all $s,\,t \in [0,T)$ with $s\leq t$.
Since $(e_r)_\#\eeta=u_r\mm$ for all $r\in [0,T)$, it follows that
\begin{equation}\label{eq:Vdiss2}
 \int V u_t - \int V u_s = \int V (\eta(t)) - V (\eta(s))  d\eeta(\eta) \le \int_s^t\Gamma(V) u_r d\mm dr.\end{equation}
Combining \eqref{eq:Vdiss1} and \eqref{eq:Vdiss2} it follows that all the intermediate inequalities we integrated
w.r.t.\ $\eeta$ are actually identities, so that for $\eeta$-a.e.
$\eta$ it must be $|\dot\eta|=\sqrt{\Gamma(V)}\circ\eta$ a.e.\ in $(0,T)$ and equality holds in \eqref{eq:wash2}. 
\end{proof}

In particular, one could prove that $\eeta$ is a $2$-plan representing the $2$-weak gradient of $V$, according to \cite[Def.~3.7]{Gigli2012}, where
a weaker asymptotic energy dissipation inequality was required at $t=0$. Our global energy dissipation is stronger, but it requires additional bounds on
the Laplacian.

We can also prove uniqueness for the continuity equation, considering just for simplicity still the autonomous
version. 

\begin{theorem}\label{thm:ediss2}
Let $V \in D(\Delta)$ with $\Delta V^-\in\Lbm\infty$. Then
the continuity equation induced by $\bb_V$ has existence and uniqueness in $L^\infty_t(L^1_x\cap L^\infty_x)$ for
any initial condition $\bar u\in L^1\cap\Lbm\infty$. 
\end{theorem}
\begin{proof} We already discussed existence in Theorem~\ref{thm:ediss}. For uniqueness, 
we want to apply
Theorem~\ref{thm:uniqueness} with $q=2$ and $r=s=4$ (which provides uniqueness in the larger
class $L^2\cap\Lbm 4$). In order to do this we need only to know that 
\eqref{eq:mass_preserving} holds (this follows by conservativity of $\sfP$ and $\BE_2(K,\infty)$), that
$L^4$-$\Gamma$ inequalities
hold in $\RCD(K,\infty)$ spaces (this follows by $\BE_2(K,\infty)$ thanks to Corollary~\ref{coro:LpGamma}) 
and that the deformation of $\bb_V$ is of type $(4,4)$ (this follows by Theorem~\ref{gradients_have_deformation}).
\end{proof}


%
%
%

\def\cprime{$'$}
\providecommand{\bysame}{\leavevmode\hbox to3em{\hrulefill}\thinspace}
\providecommand{\MR}{\relax\ifhmode\unskip\space\fi MR }
\providecommand{\MRhref}[2]{%
  \href{http://www.ams.org/mathscinet-getitem?mr=#1}{#2}
}
\providecommand{\href}[2]{#2}

\end{document}